\documentclass[12pt]{article}
\usepackage[utf8]{inputenc}
\usepackage[margin=2.5cm]{geometry}
\usepackage{amsthm}
\usepackage{enumitem}
\setlist[itemize]{topsep=0pt}
\setlist[enumerate]{topsep=0pt}
\usepackage{amsmath}
\usepackage{amsfonts}
\usepackage{ amssymb }
\usepackage[title]{appendix}
\usepackage{adjustbox}
\usepackage{multirow}
\usepackage{multicol}
\usepackage{color}
\usepackage{array}
\usepackage{graphicx}
\usepackage{hyperref}
\usepackage{mathtools}
\usepackage{xcolor}
\usepackage{tikz}
\usetikzlibrary{decorations.markings}
\usepackage{bm}
\usepackage{amsbsy}
\usepackage{titlefoot}
\usepackage{titlesec}
\titleformat*{\section}{\centering\bfseries\small\MakeUppercase}
\titleformat*{\subsection}{\normalsize\bfseries}
\titlespacing\section{0pt}{6pt plus 4pt minus 2pt}{0pt plus 2pt minus 2pt}
\titlespacing{\subsection}{0pt}{-2pt}{-4pt}
\titlespacing{\subsubsection}{0pt}{-2pt}{-4pt}
\usepackage{thm-restate}
\usepackage[capitalise]{cleveref}

\newtheoremstyle{exampstyle}
  {\topsep} 
  {0pt} 
  {\slshape} 
  {} 
  {\bfseries} 
  {.} 
  {.5em} 
  {} 
\theoremstyle{exampstyle}
\newtheorem{thm}{Theorem}[section]
\newtheorem{cor}[thm]{Corollary}
\newtheorem{lemma}[thm]{Lemma}
\newtheorem{conj}[thm]{Conjecture}
\newtheorem{prop}[thm]{Proposition}

\newtheoremstyle{def}
  {\topsep} 
  {0pt} 
  {} 
  {} 
  {\bfseries} 
  {.} 
  {.5em} 
  {} 
\theoremstyle{def}
\newtheorem{defn}[thm]{Definition}
\newtheorem{theorem}[thm]{Theorem}
\newtheorem{rmk}[thm]{Remark}
\newtheorem{exmp}[thm]{Example}

\newcommand{\comm}[1]{}
\newcommand{\Q}{\mathbb{Q}}
\newcommand{\Z}{\mathbb{Z}}

\newcommand{\ep}{\epsilon}
\newcommand{\e}{\varepsilon}
\newcommand{\Pos}{\mathrm{Pos}}
\newcommand{\Neg}{\mathrm{Neg}}
\newcommand{\myRed}[1]{\textbf{\textcolor{red}{#1}}}

\title{\large{\uppercase{\textbf{Conjugacy geodesics and growth in dihedral Artin groups}}}}
\date{}
\author{Laura Ciobanu, Gemma Crowe}
\begin{document}

\maketitle
\begin{abstract}
In this paper we describe conjugacy geodesic representatives in any dihedral Artin group $G(m)$, $m\geq 3$, which we then use to calculate asymptotics for the conjugacy growth of $G(m)$, and show that the conjugacy growth series of $G(m)$ with respect to the `free product' generating set $\{x, y\}$ is transcendental. We prove two additional properties of $G(m)$ that connect to conjugacy, namely that the permutation conjugator length function is constant, and that the falsification by fellow traveler property (FFTP) holds with respect to $\{x, y\}$. These imply that the language of all conjugacy geodesics in $G(m)$ with respect to $\{x, y\}$ is regular. \\

2020 Mathematics Subject Classification: 20E45, 20F36, 05E16.
\end{abstract}

\unmarkedfntext{\emph{Keywords}: Conjugacy growth, dihedral Artin groups, conjugator length, FFTP.}

\section{Introduction}
Let $G$ be a finitely generated group with generating set $X$. For any $n \geq 0$, the \emph{conjugacy growth function} $c(n)=c_{G,X}(n)$ counts the number of conjugacy classes with a minimal length representative of length $n$ with respect to $X$. The \emph{conjugacy growth series} of $G$ with respect to $X$ is then defined as the generating function for $c(n)$. Conjugacy growth has been studied in a variety of different groups \cite{AC2017, Guba2010, Hull2013, Mercier2016,  Rivin2010}, most recently including soluble Baumslag-Solitar groups \cite{CiobanuE2020} and graph products \cite{Ciobanu2023}. All known results support the following conjecture.
\begin{conj}\label{conj:evetts}\cite[Conjecture 7.2]{CiobanuE2020}
    Conjugacy growth series of finitely presented groups that are not virtually abelian are transcendental.
\end{conj}
In this paper we study the conjugacy growth of dihedral Artin groups, that is, Artin groups 
with generating set of size two; these are the groups $G(m) = \langle a,b \; | \; _{m}(a,b) =  {}_{m}(b,a) \rangle$, where $m\geq 3$ and $_{m}(a,b)$ is the word $abab \dots$ of length $m$. The standard growth $s(n)$, which counts elements rather than conjugacy classes, of dihedral Artin groups, has been computed over the standard Artin generators $\{a,b\}$ and the Garside generators \cite{MAIRESSE2006}. It has also been computed (by Fujii \cite{Fujii2018}, and Edjvet and Johnson \cite{EdjJohn92}) over what we denote as \emph{free product generators} (see \cref{defn: gen set free product}), because the quotient of $G(m)$ by its centre is isomorphic to a free product of cyclic groups. Here we use their geodesic normal forms to compute conjugacy geodesic representatives (minimal length representatives, over the generators, for conjugacy classes) and obtain the asymptotics of $c(n)$, which together with analytic combinatorics tools give our main result, Theorem \ref{main result}. This provides further evidence supporting \cref{conj:evetts}.
\comm{
In this paper we study conjugacy growth of Artin groups of XXL-type. Artin groups are defined by a finite simple graph with edge labellings from $\Z_{\geq 2}$. For this paper, we focus on XXL-type Artin groups, where we have the added restriction that all edge labellings are at least 5. Various results are known for this subset of Artin groups, such as solvable conjugacy problem in cubic time \cite{Appel1983, Holt2015}, acylindrical hyperbolicity \cite{Haettel2022} and biautomaticity \cite{Peifer1996}. 

We first reduce the problem of computing the conjugacy growth of Artin groups of XXL-type, to considering conjugacy growth in \emph{dihedral Artin groups} of XXL-type, where our generating set $S$ has size two. }

\begin{thm}\label{main result}(\cref{thm:odd dihe trans} and \cref{thm:even dihe trans})
    The conjugacy growth series of any dihedral Artin group $G(m)$ is transcendental, with respect to the free product generating set.
\end{thm}

An immediate application of \cref{main result} is that the conjugacy growth series for the braid group $B_{3}=G(3)$ is transcendental, with respect to the free product generating set. Whilst conjugacy growth has been previously studied in braid groups \cite{Ali2010, XU1992}, we believe this is the first result which provides information about its generating function. Since $G(2p)=\mathrm{BS}(p,p)$, we also get transcendental conjugacy growth series for a new class of Baumslag-Solitar groups with respect to their standard generating sets (see \cref{defn: gen set free product}).
\comm{
A second application of \cref{main result} concerns Artin groups of XXL-type, where all edge labellings are $\geq 5$ (see Section \ref{sec:Artin}). The XXL-Artin groups 
have a contracting element (rank 1 isometry while acting on a specific CAT(0) space) \cite[Theorem 1.4]{Haettel2022} if their rank is $\geq 3$. The existence of a contracting element gives Proposition \ref{prop:rank least 3} via \cite[Theorem 1.7]{GekhtYang} for rank $\geq 3$ Artin groups; this together with \cref{main result} can be phrased as:
\begin{restatable*}{cor}{XXLtype}\label{thm:XXL type trans}
    The conjugacy growth series for Artin groups of XXL-type is transcendental, with respect to some generating set.
\end{restatable*}}

The main tool used in \cite{AC2017, CiobanuE2020} to prove that the conjugacy growth series of a group is transcendental was to show that conjugacy growth has asymptotics of the form $\sim \frac{\alpha^n}{n}$ (see \cref{def:eq}), where $s(n)\sim \alpha^n$ is the (standard) growth of elements. Sequences with asymptotics of the form $\sim \frac{\alpha^n}{n}$ have transcendental generating functions by \cite[Thm. D]{Flajolet1987}. However, as the result below shows, we don't get this behaviour here  in all cases, so a more involved argument to find the conjugacy representatives and to count them is needed. In particular, our analysis shows (collecting the results from Section \ref{subsec:odd} and \ref{sec:even}), that
\begin{restatable*}{prop}{asymptotics}\label{prop:asymptotics}
    The asymptotics of the conjugacy growth with respect to the free product generating set in a dihedral Artin group $G(m)$ are given by
 \begin{equation*}
c_{G(m),\{x,y\}}(n)\sim \begin{cases}
              \alpha^n & \text{if } m \text{ odd or } m=4k, k \geq 1 \\
             \frac{\alpha^n}{n}  & \text{if } m=4k+2, k \geq 1 ,
       \end{cases} \quad
\end{equation*}
where $\alpha$ is the standard growth rate of $G(m)$.
\end{restatable*}
Finally, we consider the language $\mathsf{ConjGeo}(G(m),X)$ of conjugacy geodesics for dihedral Artin groups (see Section \ref{subsec:conjgeo}). This is known to be regular (that is, recognised by a finite state automaton) with respect to the standard Artin and Garside generators \cite[Cor. 3.9, Thm. 3.15]{Ciobanu2016}, but regularity is not necessarily preserved under different generating sets (see for example \cite[Section 5]{Ciobanu2016}). We provide a third case of regularity, with respect to the free product generating set.
\begin{restatable*}{thm}{conjgeos}\label{conjgeo result}
    The language $\mathsf{ConjGeo}(G(m),X)$ is regular for dihedral Artin groups, with respect to the free product generating set $X=\{x,y\}$.
\end{restatable*}
To prove this result, we first study the \emph{permutation conjugator length function} $\mathrm{PCL}_{G,X}$, defined by Antol\'{i}n and Sale in \cite{Antolin2016}. This function is of interest for groups with (sub) linear time complexity for the conjugacy problem, and so dihedral Artin groups provide an interesting candidate to study \cite[Prop. 3.1]{Holt2015}. The PCL function has been shown to be constant in hyperbolic groups, and we provide another example where the PCL function is constant.

\begin{restatable*}{prop}{PCLconstant}\label{thm:PCL constant}
    Let $G(m)$ be a dihedral Artin group, and let $X=\{x,y\}$ be the free product generating set. Then $\mathrm{PCL}_{G(m),X}$ is constant. 
\end{restatable*}
Secondly, we establish the \emph{falsification by fellow traveler property} (FFTP) in dihedral Artin groups. We note that the FFTP depends on the generating set \cite[Prop. 4.1]{Neumann1995}, and that the FFTP has been shown for Artin groups of large type, with respect to the standard generating set \cite[Thm. 4.1]{Holt2012}, and Garside groups, with respect to the Garside generators \cite[Thm. 2.9]{Holt2010}. 
We add a further positive result with respect to the free product generating set. 

\begin{restatable*}{thm}{FFTP}\label{prop:FFTP}
    Let $G(m)$ be a dihedral Artin group, and let $X$ be the free product generating set. Then $(G(m),X)$ satisfy the FFTP.
\end{restatable*}
\cref{conjgeo result} follows from 
\cref{thm:PCL constant} and \cref{prop:FFTP}, using \cite[Prop. 2.3]{Antolin2016}.

The structure of this paper is as follows. We provide information on conjugacy growth functions and Artin groups in \cref{sec:prel}, where we define the free product generating set in dihedral Artin groups. In \cref{sec:odd conj geos} and \cref{even conj geos}, we describe conjugacy geodesic representatives, by splitting our cases for when the edge labelling of our graph is either odd or even. In \cref{growth} we give asymptotics for $c(n)$ and prove \cref{main result}\comm{and \cref{thm:XXL type trans}}. Finally, in \cref{sec:conj geos language}, we use results from \cref{sec:odd conj geos} and \cref{even conj geos} to study the PCL function and the FFTP property, to prove \cref{thm:PCL constant} and \cref{prop:FFTP}. 

\section{Preliminaries}\label{sec:prel}
All groups in this paper are finitely generated, and all finite generating sets are inverse-closed.  
\subsection{Conjugacy geodesics and growth}\label{subsec:conjgeo}
We fix a group $G$ and a finite generating set $X$ of $G$. For words $u,v \in X^{\ast}$, we use $u=v$ to denote equality of words, and $u =_{G} v$ to denote equality of the group elements represented by $u$ and $v$. The \emph{(word) length} of an element $g\in G$, denoted by $|g|$, is the length of a shortest word in $X$ that represents $g$, i.e.\ $|g| = \min \{ |w| \mid w\in X^*, w =_G g\}$. In this case, we say $w$ is a \emph{geodesic word}, or simply a geodesic. If there exists a unique word $w$ of minimal length representing $g$, then we say $w$ is a \emph{unique geodesic}. Otherwise $w$ is a non-unique geodesic.
\comm{
For words $u,v \in X^{\ast}$, we use $u = v$ to denote equality of words, and $u =_{G} v$ to denote equality of the groups elements represented by $u$ and $v$. }

We will often write $g \sim h$ to denote that $g$ and $h$ are conjugate, and write $[g]$ for the conjugacy class of $g$. The \emph{length} of $[g]$, denoted by $|[g]|$, is the shortest length among all elements in $[g]$, i.e.\ $|[g]| = \min \{ |h| \mid h \sim g \}$. A word $w$ is a \emph{conjugacy geodesic} for $[g]$ if it is a geodesic, and if it moreover represents an element of shortest length in $[g]$. We define the \emph{conjugacy geodesic language} of $G$, with respect to $X$, where $\pi \colon X^{\ast} \rightarrow G$ is the natural projection, as 
\[ \mathsf{ConjGeo}(G,X) := \{ w \in X^{\ast} \; | \; l(w) = |[\pi(w)]| \}.
\]
We define the \emph{cumulative conjugacy growth function} of $G$ with respect to $X$ to be the number of conjugacy classes whose length is $\leq n$, the \emph{strict conjugacy growth function}, denoted as $c(n) = c_{G,X}(n)$, to be the number of conjugacy classes of length $=n$, i.e.\ 
\[ c(n) = \#\{[g] \mid |[g]| = n\},\]
and the \emph{conjugacy growth rate} as $\lim_{n \rightarrow \infty}\sqrt[n]{c_{G,X}(n)}$.
For ease of computation we shall work only with the strict version, and call that the \emph{conjugacy growth function}. 
The \emph{conjugacy growth series} $C(z) = C_{G,X}(z)$ is defined to be the (ordinary) generating function of $c(n)$, so
\[ C(z) = \sum\limits_{n = 0}^{\infty} c(n) z^n. \]
All results in this paper can be easily extended to the cumulative version of the conjugacy growth function and series (see \cite{AC2017}).

We call a formal power series $f(z)$ \emph{rational} if it can be expressed (formally) as the ratio of two polynomials with integral coefficients, or equivalently, the coefficients of $f(z)$ satisfy a finite linear recursion. In the language of polynomial rings, this is to say $f(z) \in \Q(z)$. Furthermore, $f(z)$ is \emph{irrational} if it is not rational. 
A formal power series is \emph{algebraic} if it is in the algebraic closure of $\Q(z)$, i.e.\ it is the solution to a polynomial equation with coefficients from $\Q(z)$. It is called \emph{transcendental} if it is not algebraic.
\subsection{Artin groups} \label{sec:Artin}
\begin{defn}
    Let $\Gamma$ be a finite simple graph, with vertex set $V(\Gamma)$ and with edges labelled by integers $m_{i,j} \in \Z_{\geq 2}$. The Artin group $A(\Gamma)$ is the group defined by the following presentation:
    \[ A(\Gamma) = \langle V(\Gamma) \; | \; _{m_{i,j}}(a_i,a_j) =  {}_{m_{i,j}}(a_j,a_i) \; \text{if the edge} \; \{a_i,a_j\} \; \text{is labelled} \; m_{i,j} \rangle,
    \]
    where $_{m_{i,j}}(a,b)$ is the word $abab \dots$ of length $m_{i,j}$. If $|V(\Gamma)| = 2$, we say $A(\Gamma)$ is a \emph{dihedral Artin group}.
    \comm{
    The Artin group $A(\Gamma)$ is 
    \begin{itemize}
        \item[(i)] of \emph{large type} if all edge labels $m_{i,j}$ are $\geq 3$,
        \item[(ii)] of \emph{extra large type} (XL) if all edge labels $m_{i,j}$ are $\geq 4$, and
        \item[(iii)] of \emph{extra extra large type} (XXL) if all edge labels $m_{i,j}$ are $\geq 5$. 
    \end{itemize}}    
\end{defn}
We will refer to $V(\Gamma)$ as the standard (Artin) generating set. For dihedral Artin groups, the presentation with respect to the standard generating set is
\begin{equation}\label{eq:all groups}
    G(m) = \langle a,b \; | \; _{m}(a,b) =  {}_{m}(b,a) \rangle.
\end{equation}
If $m=2$, then $G(m)$ is the free abelian group of rank two, and its conjugacy growth series is rational \cite{Evetts2019}, so we assume $m \geq 3$ throughout this paper.

We now define our alternative generating set for any dihedral Artin group, which we will use to study conjugacy geodesics. 

\begin{defn}\label{defn: gen set free product}
    Let $G(m)$ be the dihedral Artin group with presentation (\ref{eq:all groups}). For $m$ odd,
    \[ G(m) \cong \langle x,y \; | \; x^{2} = y^{m} \rangle,
    \]
    by setting $x = {}_{m}(a,b), y = ab$. For $m$ even,
    \[ G(m) \cong \langle x,y \; | \; y^{-1}x^{p}y = x^{p} \rangle = \mathrm{BS}(p,p),
    \]
    where $p= \frac{m}{2}$, by setting $x = ab, y = a$; $\mathrm{BS}(p,p)$ denotes a Baumslag-Solitar group. In both cases we refer to $X = \{x,y\}$ as the \emph{free product generating set}.
\end{defn}
To compute the conjugacy growth of $G(m)$, with respect to the free product generating set, we consider conjugacy geodesic representatives in the case where $m$ is odd or even in turn. 

\section{Odd dihedral Artin groups: G$(2k+1)$}\label{sec:odd conj geos}
We first consider $G(m)$ for $m$ odd as given in \cref{defn: gen set free product}, i.e.
$ G(m) \cong \langle x,y \; | \; x^{2} = y^{m} \rangle.$
Let $\Delta = x^{2}$ and note the centre of $G(m)$ is generated by $\Delta$. We often refer to $\Delta^{c}$, $c \in \Z$, as the \emph{Garside element/part} of a word. For notation, we let $m = 2k+1$.  
\subsection{Classification of geodesics}
When considering geodesics, we use the normal forms derived in \cite{Fujii2018}. An independent classification of geodesic normal forms was given in \cite[Section 2.2]{GILL1999}. We collect those results here while additionally justifying when the geodesics are unique per element or not. 

All types of geodesics are presented in Table \ref{table:oddDA}, and the reader may skip this section and consult the overview in Table \ref{table:oddDA} directly.
\begin{prop}\cite[Lemma 3.1, p. 484]{Fujii2018}\label{prop:geos}
    Let $g \in G(m)$ and let $w$ be a geodesic representative for $g$. Then $w$ can be represented as $w = x^{a_{1}}y^{b_{1}}\dots x^{a_{\tau}}y^{b_{\tau}}\Delta^{c}$, where $w$ satisfies one of the following mutually exclusive conditions (where $\tau \in \Z_{>0}$):
\begin{enumerate}
        \item[$(1)$] $\begin{cases}
        c>0, \; 0 \leq a_{i} \leq 1 \; (1 \leq i \leq \tau), & a_{i} \neq 0 \; (2 \leq i \leq \tau),\\
        -(k-1) \leq b_{i} \leq k+1 \; (1 \leq i \leq \tau), & b_{i} \neq 0 \; (1 \leq i \leq \tau - 1).
        \end{cases}$
        \item[$(2)$] $\begin{cases}
            c<0, \; -1 \leq a_{i} \leq 0 \; (1 \leq i \leq \tau), & a_{i} \neq 0 \; (2 \leq i \leq \tau), \\
            -(k+1) \leq b_{i} \leq k-1 \; (1 \leq i \leq \tau), & b_{i} \neq 0 \; (1 \leq i \leq \tau - 1).
        \end{cases}$
        \item[$\left(3^{+}\right)$] 
    $\begin{cases}
       c=0, \; 0 \leq a_{i} \leq 1 \; (1 \leq i \leq \tau), & a_{i} \neq 0 \; (2 \leq i \leq \tau),\\
        -(k-1) \leq b_{i} \leq k+1 \; (1 \leq i \leq \tau), & b_{i} \neq 0 \; (1 \leq i \leq \tau - 1).
    \end{cases}$
    \item[$(3^{-})$] 
    $\begin{cases}
        c=0, \; -1 \leq a_{i} \leq 0 \; (1 \leq i \leq \tau), & a_{i} \neq 0 \; (2 \leq i \leq \tau), \\
            -(k+1) \leq b_{i} \leq k-1 \; (1 \leq i \leq \tau), & b_{i} \neq 0 \; (1 \leq i \leq \tau - 1).
    \end{cases}$
    \item[$(3^{+}\cap 3^{-})$] $w = y^{b}$ where $-(k-1) \leq b \leq k-1$.  
    \item[$\left(3^{0}\right)$] 
    $\begin{cases}
        c=0, \; -1 \leq a_{i} \leq 1 \; (1 \leq i \leq \tau), & a_{i} \neq 0 \; (2 \leq i \leq \tau), \\
        -(k+1) \leq b_{i} \leq k+1 \; (1 \leq i \leq \tau), & b_{i} \neq 0 \; (1 \leq i \leq \tau -1).
     \end{cases}$
    \end{enumerate}
\end{prop}
Later in the paper we will use $(3) = (3^{+}) \sqcup (3^{-}) \sqcup (3^{+} \cap 3^{-}) \sqcup \left(3^{0}\right)$. 
By \cite[Prop. 3.4]{Fujii2018}, all words of the form $(1) \sqcup (2) \sqcup (3^{+}) \sqcup (3^{-}) \sqcup \left(3^{+} \cap 3^{-}\right)$ are geodesic, so it remains to check when $w \in \left(3^{0}\right)$ is geodesic. The proof of the following result can be found in \cref{append1:odd}.

\begin{prop}\label{30 cases}
    If $w = x^{a_{1}}y^{b_{1}}\dots x^{a_{\tau}}y^{b_{\tau}}\Delta^{c} \in \left(3^{0}\right)$ is geodesic, then $c=0$ and $w$ satisfies one of the following mutually exclusive conditions (where $\tau \in \Z_{>0}$):
    \begin{enumerate}
    \item[$\left(3^{0+}\right)$] 
    $\begin{cases}
        0 \leq a_{i} \leq 1 \; (1 \leq i \leq \tau), & a_{i} \neq 0 \; (2 \leq i \leq \tau),\\ 
        -k \leq b_{i} \leq k+1 \; (1 \leq i \leq \tau), & b_{i} \neq 0 \; (1 \leq i \leq \tau -1), \; \text{there exists at least one} \; y^{-k} \; \text{term}.
    \end{cases}$
    \item[$\left(3^{0-}\right)$]
    $\begin{cases}
        -1 \leq a_{i} \leq 0 \; (1 \leq i \leq \tau), & a_{i} \neq 0 \; (2 \leq i \leq \tau),\\ 
        -(k+1) \leq b_{i} \leq k \; (1 \leq i \leq \tau), & b_{i} \neq 0 \; (1 \leq i \leq \tau -1), \; \text{there exists at least one} \; y^{k} \; \text{term}.
    \end{cases}$
    \item[$\left(3^{0*}\right)$]
    $\begin{cases}
        -1 \leq a_{i} \leq 1 \; (1 \leq i \leq \tau), & a_{i} \neq 0 \; (2 \leq i \leq \tau), \; \text{there exist both} \; x, x^{-1} \; \text{terms},\\ 
        -k \leq b_{i} \leq k \; (1 \leq i \leq \tau), & b_{i} \neq 0 \; (1 \leq i \leq \tau -1).
    \end{cases}$
\end{enumerate}
\end{prop}
From (\cite{Fujii2018}, p. 489), all words in $(1) \sqcup (2) \sqcup (3^{+}) \sqcup (3^{-}) \sqcup (3^{+} \cap 3^{-})$ are unique geodesics. 
To determine which words in $\left(3^{0}\right)$ are unique geodesics we use Lemma \ref{lem:non unique} to establish when elements in $(3)$ are not uniquely geodesic.
\begin{defn}
For $w = x^{a_{1}}y^{b_{1}}\dots x^{a_{\tau}}y^{b_{\tau}} \in (3)$, define
\begin{alignat*}{2}
    \Pos_{x}(w) &:= \mathrm{max}\{a_{i} \; : \; a_{i} \geq 0, 1 \leq i \leq \tau \}, \quad \Neg_{x}(w) &&:= \mathrm{max}\{-a_{i} \; : \; a_{i} \leq 0, 1 \leq i \leq \tau \}, \\
    \Pos_{y}(w) &:= \mathrm{max}\{b_{i} \; : \; b_{i} \geq 0, 1 \leq i \leq \tau \}, \quad \Neg_{y}(w) &&:= \mathrm{max}\{-b_{i} \; : \; b_{i} \leq 0, 1 \leq i \leq \tau \}.
\end{alignat*} 
\end{defn}
\begin{lemma}\label{lem:non unique}\cite[Prop. 3.8]{Fujii2018}
Let $w \in (3)$ be geodesic. Then $w$ is a non-unique geodesic representative if at least one of the following conditions holds:
(1) Both $x$ and $x^{-1}$ terms exist in $w$, or
   (2) $\Pos_{y}(w) + \Neg_{y}(w) = m$.
\end{lemma}
We consider the conditions from \cref{lem:non unique} for each type from \cref{30 cases} in turn. If $w \in \left(3^{0+}\right)$, then the only rewrite rule we can apply to non-unique geodesics is 
\begin{equation}\label{eq:a non unique}
    sy^{-k}ty^{k+1}z =_{G} sy^{k+1}ty^{-k}z,
\end{equation}
for some $s,t,z \in X^{\ast}$. In order to apply this rule, we need at least one $y^{k+1}$ term in $w$. Therefore the set $\left(3^{0+}\right)$ splits into two disjoint sets, which we denote by $\left(3^{0+}U\right) \sqcup \left(3^{0+}N\right)$, where $\left(3^{0+}U\right)$ is the set of unique geodesics ($-k \leq b_{i} \leq k$), and $\left(3^{0+}N\right)$ is the set of non-unique geodesics, i.e. there exist both $y^{-k}$ and $y^{k+1}$ terms.
\begin{defn}\label{rmk:alternative odd a}
    Let $\tau = \tau_{1} + \tau_{2}$. Any Type $\left(3^{0+}\right)$ geodesic can be written as 
    \begin{equation}\label{eq:alt odd 3a}
    w = A_{1}y^{-k}A_{2}y^{-k}\dots A_{\tau_{1}}y^{-k}A_{\tau_{1} + 1}y^{k+1}A_{\tau_{1}+2}y^{k+1}\dots y^{k+1}A_{\tau_{1} + \tau_{2} + 1},
\end{equation}
    where each $A_{s}$ $(1 \leq s \leq \tau + 1)$ are reduced words over $x,y$, with any maximal subword of the form $y^{\beta}$ in $ A_{s}$ satisfying $\beta \in [-(k-1), k] \cap \mathbb{Z}_{\neq 0}$. Here $\tau_{1} \geq 1, \tau_{2} \geq 0$, $A_{1}$ ($A_{\tau + 1}$ resp.) is either empty or ends (starts) with $x$, and $A_{s}$ starts and ends with $x$ $(2 \leq s \leq \tau)$.

     Type $\left(3^{0+}U\right)$ is the set of all Type $\left(3^{0+}\right)$ geodesics where $\tau_{2} = 0$, i.e. of the form 
     \[         w = A_{1}y^{-k}A_{2}y^{-k}\dots A_{\tau_{1}}y^{-k}A_{\tau_{1}+1},
     \]
    and Type $\left(3^{0+}N\right)$ is the set of all remaining Type $\left(3^{0+}\right)$ geodesics, i.e. with $\tau_{2} \geq 1$ and 
    \[w =_{G} A_{1}y^{-k}A_{2}y^{-k}\dots A_{\tau_{1}}y^{-k}A_{\tau_{1} + 1}y^{k+1}A_{\tau_{1}+2}y^{k+1}\dots y^{k+1}A_{\tau_{1} + \tau_{2} + 1}.\]
\end{defn}
Similarly, for $\left(3^{0-}\right)$ words, which are negative counterparts to $\left(3^{0+}\right)$, 
 we can split $\left(3^{0-}\right)$ into two disjoint sets $\left(3^{0-}U\right) \sqcup \left(3^{0-}N\right)$, where $\left(3^{0-}U\right)$ is the set of unique geodesics ($-k \leq b_{i} \leq k$), and $\left(3^{0-}N\right)$ is the set of non-unique geodesics. 
\begin{defn}\label{rmk:alternative odd b}
    Let $\tau = \tau_{1} + \tau_{2}$. Any Type $\left(3^{0-}\right)$ geodesic can be written as
    \begin{equation}\label{eq:alt odd 3b}
    w = A_{1}y^{-(k+1)}A_{2}y^{-(k+1)}\dots A_{\tau_{1}}y^{-(k+1)}A_{\tau_{1}+1}y^{k}A_{\tau_{1}+2}y^{k}\dots y^{k}A_{\tau_{1}+\tau_{2}+1},
\end{equation}
    where each $A_{s}$ $(1 \leq s \leq \tau + 1)$ are reduced words over $x,y$, with any maximal subword of the form $y^{\beta}$ in $ A_{s}$ satisfying $\beta \in [-k, k-1] \cap \mathbb{Z}_{\neq 0}$. Here $\tau_{1} \geq 0, \tau_{2} \geq 1$, $A_{1}$ ($A_{\tau + 1}$ resp.) is either empty or ends (starts) with $x^{-1}$, and $A_{s}$ starts and ends with $x^{-1}$ $(2 \leq s \leq \tau)$.

     Type $\left(3^{0-}U\right)$ is the set of Type $\left(3^{0-}\right)$ geodesics with $\tau_{1} = 0$, so of the form 
     $A_{1}y^{k}A_{2}y^{k}\dots A_{\tau_{2}+1}y^{k},$
    and Type $\left(3^{0-}N\right)$ is the set of all remaining Type $\left(3^{0-}\right)$ geodesics, i.e. where $\tau_{1} \geq 1$, so of the form $A_{1}y^{-(k+1)}A_{2}y^{-(k+1)}\dots A_{\tau_{1}}y^{-(k+1)}A_{\tau_{1}+1}y^{k}A_{\tau_{1}+2}y^{k}\dots y^{k}A_{\tau_{1}+\tau_{2}+1}$. 
\end{defn}
Finally, all words in $\left(3^{0*}\right)$ are non-unique geodesics, since we can apply the rewrite rule
\begin{equation}\label{rewrite pairs}
    sxtx^{-1}z =_{G} sx^{-1}txz,
\end{equation}
for any $s,t,z \in X^{\ast}$ and pair of $x$ and $x^{-1}$. This is the only possible rewrite rule by \cref{lem:non unique}.
\begin{defn}\label{defn:30*}
    Let $\tau = \tau_{1}+\tau_{2}$, $a_i=\pm 1$, and define the following tuple
    \[ X = \left(x^{a_{1}}, x^{a_{2}}, \dots, x^{a_{\tau}}\right) =(\underbrace{x^{-1}, x^{-1}, \dots, x^{-1}}_{\tau_{1}}, \underbrace{x, x, \dots, x}_{\tau_{2}}).
    \]
    Any Type $\left(3^{0*}\right)$ geodesic can be written as a word $w$ respecting the $X$-tuple structure:
    \begin{align*}
        w &= x^{a_{1}}y^{b_{1}}x^{a_{1}}y^{b_{2}}\dots y^{b_{\tau_{1}}}x^{a_{\tau_{1}+1}}y^{b_{\tau_{1}+1}}x^{a_{\tau_{1}+2}}\dots x^{a_{\tau_{1}+\tau_{2}}}y^{b_{\tau_{1}+\tau_{2}}}\\
        &= x^{-1}y^{b_{1}}x^{-1}y^{b_{2}}\dots y^{b_{\tau_{1}}}xy^{b_{\tau_{1}+1}}x\dots y^{b_{\tau_{1}+\tau_{2}}}.
    \end{align*}
    
\end{defn}
\comm{
\begin{cor}\label{cor:overcount}
    Let $w = x^{a_{1}}y^{b_{1}}\dots x^{a_{\tau}}y^{b_{\tau}} \in (3^{0})$. Let $\tau_{1}$ be the number of $x$ terms, and $\tau_{2}$ be the number of $x^{-1}$ terms which occur in $w$ (i.e. $\tau = \tau_{1} + \tau_{2}$). Then the number of geodesic representatives for $w$ is
    \[ \binom{\tau}{\tau_{1}} = \binom{\tau}{\tau_{2}}.
    \]
\end{cor}}


\subsection{Minimal length conjugacy representatives in odd dihedral Artin groups}
We now describe conjugacy geodesic representatives. To do this, we use the sets $\mathcal{A}$ and $\mathcal{B}$ introduced in Definition \ref{def:AB} separately.

\begin{defn}\label{def:AB}
    Let $\mathcal{A} = (1) \sqcup (2) \sqcup (3^{+}) \sqcup (3^{-}) \sqcup (3^{+} \cap 3^{-}) \sqcup \left(3^{0+}U\right) \sqcup \left(3^{0-}U\right)$ be the set of unique geodesics, and let $\mathcal{B} = \left(3^{0+}N\right) \sqcup \left(3^{0-}N\right) \sqcup \left(3^{0*}\right)$ be the set of non-unique geodesics. 
\end{defn}
\begin{defn}\label{A bar}
    Let $\overline{\mathcal{A}}$ be the set of all words in $\mathcal{A}$ which start and end with opposite letters:  
    \[ \overline{\mathcal{A}} = \{ w = x^{a_{1}}y^{b_{1}}\dots x^{a_{\tau}}y^{b_{\tau}}\Delta^{c} \in \mathcal{A} \; | \; a_{1} = 0 \Leftrightarrow b_{\tau} = 0 \}.
\]
\end{defn}

\comm{
Note $\overline{\mathcal{A}} \subset \mathsf{CycGeo}$. }
\begin{prop}\label{prop:prod equal}
    Let $w,v \in \overline{\mathcal{A}}$ have the form
    \[ w = x^{\alpha_{1}}y^{\beta_{1}}\dots x^{\alpha_{n_{w}}}y^{\beta_{n_{w}}}\Delta^{p}, \quad v = x^{\sigma_{1}}y^{\gamma_{1}}\dots x^{\sigma_{n_{v}}}y^{\gamma_{n_{v}}}\Delta^{q}.
    \]
    Then $w \sim v$ if and only if $n_{w}=n_{v}, \; p=q$, and the words $x^{\alpha_{1}}y^{\beta_{1}}\dots x^{\alpha_{n_{w}}}y^{\beta_{n_{w}}}$ and $x^{\sigma_{1}}y^{\gamma_{1}}\dots x^{\sigma_{n_{v}}}y^{\gamma_{n_{v}}}$ are cyclic permutations of each other. 
\end{prop}
\begin{proof}
The reverse direction is clear, so suppose $w \sim v$. Let $u = x^{a_{1}}y^{b_{1}}\dots x^{a_{\tau}}y^{b_{\tau}}\Delta^{c}$ be geodesic and consider $u^{-1}wu$. We first note the Garside element  remains unchanged since $\Delta^{-c}\Delta^{p}\Delta^{c} = \Delta^{p}$. Also, we only have elements from the same free product factor concatenating in either $u^{-1}w$ or $wu$, but not both. We consider each type of concatenation in turn, and show that if we change the power of the Garside element, we can reverse this procedure, and so preserve the type of geodesic as defined in \cref{A bar}. 
    \par 
    First suppose the matching terms from the same factor consist of $y$ letters. Assuming free cancellation where possible, we have
    \[ u^{-1}wu =_{G} y^{-b_{\tau}}x^{-a_{\tau}}\dots y^{-b_{1}}\cdot x^{\alpha_{1}}y^{\beta_{1}}\dots x^{\alpha_{n_{w}}}y^{\beta_{n_{w}}} \cdot y^{b_{1}}\dots x^{a_{\tau}}y^{b_{\tau}}\Delta^{p}.
    \]
    Note if $\beta_{n} = -b_{1}$, then up to cyclic reduction $u^{-1}wu$ is a cyclic permutation of $w$, which remains in the same geodesic type as defined in \cref{A bar}. Since $u,w$ are both geodesic, we know that $\beta_{n_{w}} + b_{1} < m$ or $\beta_{n_{w}} + b_{1} > -m$ except in the following cases, where $\varepsilon = \pm 1$:  
    \begin{enumerate}
        \item[(i)] $\beta_{n_{w}} = b_{1} = \varepsilon (k+1)$,
        \item[(ii)] $\beta_{n_{w}} = \varepsilon k, \; b_{1} = \varepsilon (k+1)$ (and vice versa).
    \end{enumerate}
    Suppose $\beta_{n_{w}} = b_{1} = k+1$. After concatenation, $y^{\beta_{n_{w}}+b_{1}} = y^{m+1}$, which can be moved to the Garside element as follows:  
    \begin{align*}
        u^{-1}wu &=_{G} y^{-b_{\tau}}x^{-a_{\tau}}\dots y^{-b_{1}}\cdot x^{\alpha_{1}}y^{\beta_{1}}\dots x^{\alpha_{n_{w}}}y^{m+1}\dots x^{a_{\tau}}y^{b_{\tau}}\Delta^{p}\\
    &=_{G} y^{-b_{\tau}}x^{-a_{\tau}}\dots y^{-b_{1}}\cdot x^{\alpha_{1}}y^{\beta_{1}}\dots x^{\alpha_{n_{w}}}y\dots x^{a_{\tau}}y^{b_{\tau}}\Delta^{p+1}.
    \end{align*}
    Since $\beta_{n_{w}} = k+1$, we can assume $w$ is of Type $(1)$ or $(3^{+})$. Also $u^{-1}$ ends with $y^{-b_{1}} = y^{-(k+1)}$, and so $u^{-1}wu$ can be rewritten as follows:
    \begin{align*}
        u^{-1}wu &=_{G} y^{-b_{\tau}}x^{-a_{\tau}}\dots y^{-(k+1)}\cdot x^{\alpha_{1}}y^{\beta_{1}}\dots x^{\alpha_{n_{w}}}x^{2}y\dots x^{a_{\tau}}y^{b_{\tau}}\Delta^{p}\\
    &=_{G} y^{-b_{\tau}}x^{-a_{\tau}}\dots y^{-(k+1)}\cdot x^{\alpha_{1}}y^{\beta_{1}}\dots x^{\alpha_{n_{w}}}y\dots x^{a_{\tau}}y^{b_{\tau}}\Delta^{p+1} \\
    &=_{G} y^{-b_{\tau}}x^{-a_{\tau}}\dots y^{-m}y^{k}\cdot x^{\alpha_{1}}y^{\beta_{1}}\dots x^{\alpha_{n_{w}}}y\dots x^{a_{\tau}}y^{b_{\tau}}\Delta^{p+1} \\
     &=_{G} y^{-b_{\tau}}x^{-a_{\tau}}\dots y^{k}\cdot x^{\alpha_{1}}y^{\beta_{1}}\dots x^{\alpha_{n_{w}}}y\dots x^{a_{\tau}}y^{b_{\tau}}\Delta^{p}.
    \end{align*}
    Now $u^{-1}wu$ is the same type (either $(1)$ or $(3^{+})$) as $w$, and up to cyclic reduction is a cyclic permutation of $w$. The remaining cases follow a similar strategy. Otherwise, if $\beta_{n_{w}} + b_{1} \neq 0$, then $u^{-1}wu$ is equal to $w$ up to cyclic reduction.  
    \par 
    Now suppose the matching terms from the same factor consist of $x$ letters. Assuming free cancellation where possible, we have
    \[ u^{-1}wu =_{G} y^{-b_{\tau}}x^{-a_{\tau}}\dots y^{-b_{1}}x^{-a_{1}}\cdot x^{\alpha_{1}}y^{\beta_{1}}\dots x^{\alpha_{n_{w}}}y^{\beta_{n_{w}}} \cdot x^{a_{1}}y^{b_{1}}\dots x^{a_{\tau}}y^{b_{\tau}}\Delta^{p}.
    \]
    Again if $a_{1} = \alpha_{1}$, then up to cyclic reduction $u^{-1}wu$ is a cyclic permutation of $w$, which remains in the required geodesic form as defined in \cref{A bar}. If $a_{1} \neq \alpha_{1}$, then by \cref{prop:geos} we can assume $\alpha_{1} - a_{1} = \pm 2$. First suppose $\alpha_{1} - a_{1} = 2$. Then we can move the $x^{\alpha_{1}-a_{1}} = x^{2}$ term to the Garside element, which becomes $\Delta^{p+1}$. Since $\alpha_{1} = 1$, we can assume $w$ is of Type $(1)$ or $(3^{+})$. We note that $u$ starts with $x^{a_{1}} = x^{-1}$, so we can rewrite as:
    \begin{align*}
        u^{-1}wu &=_{G} y^{-b_{\tau}}x^{-a_{\tau}}\dots y^{-b_{1}}x^{2}y^{\beta_{1}}\dots x^{\alpha_{n_{w}}}y^{\beta_{n_{w}}} \cdot x^{-1}y^{b_{1}}\dots x^{a_{\tau}}y^{b_{\tau}}\Delta^{p} \\
        &=_{G} y^{-b_{\tau}}x^{-a_{\tau}}\dots y^{-b_{1}}y^{\beta_{1}}\dots x^{\alpha_{n_{w}}}y^{\beta_{n_{w}}} \cdot x^{-1}y^{b_{1}}\dots x^{a_{\tau}}y^{b_{\tau}}\Delta^{p+1}\\
        &=_{G} y^{-b_{\tau}}x^{-a_{\tau}}\dots y^{-b_{1}}y^{\beta_{1}}\dots x^{\alpha_{n_{w}}}y^{\beta_{n_{w}}} \cdot x^{-2}xy^{b_{1}}\dots x^{a_{\tau}}y^{b_{\tau}}\Delta^{p+1} \\
        &=_{G} y^{-b_{\tau}}x^{-a_{\tau}}\dots y^{-b_{1}}y^{\beta_{1}}\dots x^{\alpha_{n_{w}}}y^{\beta_{n_{w}}} \cdot xy^{b_{1}}\dots x^{a_{\tau}}y^{b_{\tau}}\Delta^{p}.
    \end{align*}
    This term is now the same type of geodesic as $w$, and up to cyclic reduction is a cyclic permutation of $w$. The case for $\alpha_{1}-a_{1} = -2$ follows a symmetric proof. This method also holds when $y$ letters concatenate in $u^{-1}w$ or $x$ letters concatenate in $wu$.      
\end{proof}
\comm{
\myRed{Keep this remark?}
\begin{rmk}
    Result implies $S = \mathsf{ConjGeo}(G,X) \subset \mathsf{CycGeo}(G,X)$. This matches from 4 authors paper 'Conjugacy languages in groups' - over standard generating set $Y$, $\mathsf{ConjGeo}(G,Y) = \mathsf{CycGeo}\setminus A$ where $A$ is finite (see Page 13), using locally testable results.
\end{rmk}}
We now consider the set $\mathcal{B}$, and construct a method which ensures we select a unique representative from all possible geodesic forms.
First recall that any word $w$ of Type $\left(3^{0+}N\right)$ can only be rewritten using \cref{eq:a non unique}. For choosing a unique representative, we will choose words where all $y^{-k}$ appear leftmost in the word, and all $y^{k+1}$ appear rightmost the word, to match the form defined in \cref{rmk:alternative odd a}. For conjugation, we no longer preserve this unique choice up to cyclic permutation. For example, consider $w \in \left(3^{0+}N\right)$ of the form in \cref{eq:alt odd 3a}, i.e.
\[ w = A_{1}y^{-k}A_{2}y^{-k}\dots A_{\tau_{1}}y^{-k}A_{\tau_{1} + 1}y^{k+1}A_{\tau_{1}+2}y^{k+1}\dots y^{k+1}A_{\tau_{1} + \tau_{2} + 1},
\]
and its cyclic permutation
$ w' = y^{k+1}A_{\tau + 1}A_{1}y^{-k}A_{2}y^{-k}\dots A_{\tau_{1}}y^{-k}A_{\tau_{1}+1}y^{k+1}A_{\tau_{1}+2}y^{k+1}\dots A_{\tau}.$
This is no longer the correct geodesic form of \cref{rmk:alternative odd a}, as we can apply \cref{eq:a non unique} to get
\[ w' =_{G} y^{-k}A_{\tau +1}A_{1}y^{-k}A_{2}y^{-k}\dots A_{\tau_{1}}y^{k+1}A_{\tau_{1}+1}y^{k+1}A_{\tau_{1}+2}y^{k+1}\dots A_{\tau}.
\]
This rewrite only occurs when we cyclically permute $y^{-k}$ or $y^{k+1}$ terms. Therefore any cyclic permutation of $w$, written as our geodesic normal form from \cref{rmk:alternative odd a}, is a cyclic permutation of elements in $(A_{1}, \dots, A_{\tau + 1})$, such that blocks are separated by $Y_{a} = \left(y^{-k}, \dots, y^{-k}, y^{k+1}, y^{k+1}, \dots, y^{k+1}\right)$, where $Y_{a}$ is fixed. 

This also occurs for any $w$ of Type $\left(3^{0-}N\right)$: any cyclic permutation of $w$, written as our geodesic normal form from \cref{rmk:alternative odd b}, is a cyclic permutation of elements in $(A_{1}, \dots, A_{\tau + 1})$, such that blocks are separated by $Y_{b} = \left(y^{-(k+1)}, \dots, y^{-(k+1)}, y^{k}, y^{k}, \dots, y^{k}\right)$, where $Y_{b}$ is fixed.

\begin{defn}(split cyclic permutation for Types $\left(3^{0+}N\right)$ and $\left(3^{0-}N\right)$)\label{defn: split odds}

Let $\tau = \tau_{1} + \tau_{2}$ with $\tau_{1}, \tau_{2} \geq 1$.
\begin{itemize}
    \item[(i)] 
    Define tuples $Y_{a_{\tau_{1}, \tau_{2}}}, Y_{b_{\tau_{1}, \tau_{2}}}$ of powers of $y$ with $a_i \in \{-k, k+1\}, b_{i} \in \{-(k+1), k\}$ as:
    \begin{align*}
        Y_{a_{\tau_{1},\tau_{2}}} &= \left(y^{a_{1}}, y^{a_{2}}, \dots, y^{a_{\tau}}\right) = (\underbrace{y^{-k}, y^{-k}, \dots, y^{-k}}_{\tau_{1}}, \underbrace{y^{k+1}, y^{k+1}, \dots, y^{k+1}}_{\tau_{2}}), \\
        Y_{b_{\tau_{1},\tau_{2}}} &= \left(y^{b_{1}}, y^{b_{2}}, \dots, y^{b_{\tau}}\right) = (\underbrace{y^{-(k+1)}, y^{-(k+1)}, \dots, y^{-(k+1)}}_{\tau_{1}}, \underbrace{y^{k}, y^{k}, \dots, y^{k}}_{\tau_{2}}).
    \end{align*}
    Let $w \in \left(3^{0+}N\right)$ be a geodesic based on $Y_{a_{\tau_{1},\tau_{2}}}$ of the form
    \begin{align*}
        w &= A_{1}y^{a_{1}}A_{2}y^{a_{2}}\dots A_{\tau_{1}}y^{a_{\tau_{1}}}A_{\tau_{1} + 1}y^{a_{\tau_{1}+1}}A_{\tau_{1} + 2}y^{a_{\tau_{1} +2}}\dots y^{a_{\tau}}A_{\tau + 1} \\
        &= A_{1}y^{-k}A_{2}y^{-k}\dots A_{\tau_{1}}y^{-k}A_{\tau_{1} + 1}y^{k+1}A_{\tau_{1} + 2}y^{k+1}\dots y^{k+1}A_{\tau + 1},
    \end{align*}
    where the $A_{s}$ are words over $x,y$ satisfying the conditions from \cref{eq:alt odd 3a}.
Alternatively, let $w \in \left(3^{0-}N\right)$ be a geodesic based on $Y_{b_{\tau_{1}, \tau_{2}}}$ of the form
    \begin{align*}
        w &= A_{1}y^{a_{1}}A_{2}y^{a_{2}}\dots A_{\tau_{1}}y^{a_{\tau_{1}}}A_{\tau_{1} + 1}y^{a_{\tau_{1}+1}}A_{\tau_{1} + 2}y^{a_{\tau_{1} +2}}\dots y^{a_{\tau}}A_{\tau + 1} \\
        &= A_{1}y^{-(k+1)}A_{2}y^{-(k+1)}\dots A_{\tau_{1}}y^{-(k+1)}A_{\tau_{1} + 1}y^{k}A_{\tau_{2} + 2}y^{k}\dots y^{k}A_{\tau + 1},
    \end{align*}
    where the $A_{s}$ are words over $x,y$ satisfying the conditions from \cref{eq:alt odd 3b}.
  \item[(ii)]  
    A \emph{split cyclic permutation} of $w$ as in (i) is one where the tuple $Y_{a_{\tau_{1}, \tau_{2}}}$ ($Y_{b_{\tau_{1},\tau_{2}}}$ resp.) is preserved and the $A_{s}$ blocks are cyclically permuted; that is, it is a geodesic $w' \in \left(3^{0+}N\right)$ ($w' \in \left(3^{0-}N\right)$ resp.) of one of the following forms:\\
    Case 1: all $A_{s}$ blocks are preserved, and for $1 \leq t \leq \tau+1$
    \[ w' = y^{a_{1}}A_{t}y^{a_{2}}A_{t+1}\dots y^{a_{\tau - t + 2}}A_{\tau + 1}A_{1}y^{a_{\tau -t+3}}A_{2}\dots y^{a_{\tau}}A_{t-1}.
    \]
    Case 2: one of the $A_{s}$ blocks is split; that is, for $1 \leq t \leq \tau+1$ write $A_{t} = A_{t_{1}}A_{t_{2}}$ as a reduced product of prefix and suffix. Then 
    \[  w' = A_{t_{2}}y^{a_{1}}A_{t+1}y^{a_{2}}\dots y^{a_{\tau-t+1}}A_{\tau + 1}A_{1}y^{a_{\tau - t + 2}}A_{2}\dots y^{a_{\tau}}A_{t_{1}}.
    \]
    
\end{itemize}
\end{defn}
\begin{exmp}\label{ex:split}
    Consider the word $w \in (3^{0+}N)$ of the form
    \[ w = xyx\textcolor{red}{y^{-2}}x\textcolor{red}{y^{-2}}x\textcolor{red}{y^{3}}xy^{-1}x\textcolor{red}{y^{3}}.
    \]
   We have highlighted the terms from the tuple $Y_{a_{\tau_{1}}, a_{\tau_{2}}}$ as in \cref{defn: split odds} (here $k = 2$). The following is an example of a cyclic permutation of $w$:
    \[ w' = xy^{-1}x\textcolor{red}{y^{3}}xyx\textcolor{red}{y^{-2}}x\textcolor{red}{y^{-2}}x\textcolor{red}{y^{3}}.
    \]
    This differs from a split cyclic permutation of $w$, for example 
    \[ w'' = xy^{-1}x\textcolor{red}{y^{-2}}xyx\textcolor{red}{y^{-2}}x\textcolor{red}{y^{3}}x\textcolor{red}{y^{3}}
    \]
    is a split cyclic permutation of $w$.
\end{exmp}

Recall for any word in $\left(3^{0*}\right)$, we can rewrite any pairs of $x$ and $x^{-1}$ terms using \cref{rewrite pairs}. For choosing a unique geodesic representative, we will choose words where all $x^{-1}$ terms appear leftmost in the word, and all $x$ terms appear rightmost in the word, to match the form defined in \cref{defn:30*}. For conjugation, again we no longer preserve this unique choice up to cyclic permutation. For example, if 
$w = x^{-1}y^{b_{1}}x^{-1}y^{b_{2}}\dots y^{b_{\tau_{1}}}xy^{b_{\tau_{1}+1}}x\dots y^{b_{\tau_{1}+\tau_{2}}}
$
is of the correct form, we can consider a cyclic permutation of $w$ of the form
\[ w' = xy^{b_{\tau_{1}+s}}x\dots y^{b_{\tau_{1}+\tau_{2}}}x^{-1}y^{b_{1}}\dots y^{b_{\tau_{1}}}xy^{b_{\tau_{1}+1}}\dots y^{b_{\tau_{1}+s-1}},
\]
for some $\tau_{1} \leq s \leq \tau_{1}+\tau_{2}$. This is no longer of the correct form from \cref{defn:30*}, since we have $x$ terms occurring before $x^{-1}$ terms. In this case, we apply \cref{rewrite pairs} to $w'$ to move all $x^{-1}$ terms to the left, and all $x$ terms to the right.  
\begin{defn}(split cyclic permutation for Type $\left(3^{0*}\right)$)\label{def:X}

    Let $\tau = \tau_{1}+\tau_{2}$, and define the following tuple
    \[ X_{\tau_{1}, \tau_{2}} = \left(x^{a_{1}}, x^{a_{2}}, \dots, x^{a_{\tau}}\right) = (\underbrace{x^{-1}, x^{-1}, \dots, x^{-1}}_{\tau_{1}}, \underbrace{x, x, \dots, x}_{\tau_{2}}).
    \]
    Let $w \in \left(3^{0*}\right)$ be a geodesic as in \cref{defn:30*}:
    \begin{align*}
        w &= x^{a_{1}}y^{b_{1}}x^{a_{1}}y^{b_{2}}\dots y^{b_{\tau_{1}}}x^{a_{\tau_{1}+1}}y^{b_{\tau_{1}+1}}x^{a_{\tau_{1}+2}}\dots x^{a_{\tau_{1}+\tau_{2}}}y^{b_{\tau_{1}+\tau_{2}}}\\
        &= x^{-1}y^{b_{1}}x^{-1}y^{b_{2}}\dots y^{b_{\tau_{1}}}xy^{b_{\tau_{1}+1}}x\dots y^{b_{\tau_{1}+\tau_{2}}}.
    \end{align*}
    A \emph{split cyclic permutation} of $w$ is a geodesic
    $ w' = x^{a_{1}}y^{\beta_{1}}x^{a_{2}}y^{\beta_{2}}\dots y^{\beta_{\tau}} \in \left(3^{0*}\right)$,
    such that $\left(y^{\beta_{1}}, y^{\beta_{2}}, \dots, y^{\beta_{\tau}}\right)$ is a cyclic permutation of $\left(y^{b_{1}}, y^{b_{2}}, \dots, y^{b_{\tau}}\right)$; $w'$ may start with $y^{\beta_{1}}$. 
\end{defn}
\begin{defn}\label{B bar}
    Let $\overline{\mathcal{B}} = \overline{\mathcal{B}_{+}} \sqcup \overline{\mathcal{B}_{-}} \sqcup \overline{\mathcal{B}_{*}}$, where
    \begin{enumerate}
        \item[(i)] $\overline{\mathcal{B}_{+}} \subset \left(3^{0+}N\right)$ is the set of all words of the form in \cref{eq:alt odd 3a},
        \item[(ii)] $\overline{\mathcal{B}_{-}} \subset \left(3^{0-}N\right)$ is the set of all words of the form in \cref{eq:alt odd 3b}, 
        \item[(iii)] $\overline{\mathcal{B}_{*}} \subset \left(3^{0*}\right)$ is the set of all words of the form in \cref{defn:30*},
    \end{enumerate}
    and for all $w \in \overline{\mathcal{B}}$, $w$ starts and ends with opposite letters. 
\end{defn}

\begin{prop}\label{prop:conj geo B}    
    Let $w \in \overline{\mathcal{B}}$ be based on a tuple $Y_{a_{\tau_{1}}, a_{\tau_{2}}}, Y_{b_{\tau_{1}}, b_{\tau_{2}}}$ or $X_{\tau_{1}, \tau_{2}}$, and let $v \in \overline{\mathcal{B}}$ be based on $Y_{a_{\sigma_{1}}, a_{\sigma_{2}}}, Y_{b_{\sigma_{1}}, b_{\sigma_{2}}}$ or $X_{\sigma_{1}, \sigma_{2}}$. Then $w \sim v$ if and only if $Y_{a_{\tau_{1}}, a_{\tau_{2}}} = Y_{a_{\sigma_{1}}, a_{\sigma_{2}}}, \; Y_{b_{\tau_{1}}, b_{\tau_{2}}} = Y_{b_{\sigma_{1}}, b_{\sigma_{2}}}$ or $X_{\tau_{1}, \tau_{2}} = X_{\sigma_{1}, \sigma_{2}}$, and $w$ and $v$ are equal up to a split cyclic permutation.  
\end{prop}

\begin{proof}
    We follow a similar method as in \cref{prop:prod equal}, but we need to additionally consider 
    split cyclic permutations. 
Let $w = x^{\alpha_{1}}y^{\beta_{1}}\dots y^{\beta_{n}}$ and consider the word $u^{-1}wu$ where $u = x^{a_{1}}y^{b_{1}}\dots x^{a_{\tau}}y^{b_{\tau}}\Delta^{c}$ is geodesic. Firstly suppose the matching terms from the same factor consist of $y$ letters. Assuming free cancellation where possible, we have
    \[ u^{-1}wu =_{G} y^{-b_{\tau}}x^{-a_{\tau}}\dots y^{-b_{1}}\cdot x^{\alpha_{1}}y^{\beta_{1}}\dots x^{\alpha_{n}}y^{\beta_{n}} \cdot y^{b_{1}}\dots x^{a_{\tau}}y^{b_{\tau}}.
    \]
    Since $u,w$ are both geodesic and $w \in (3^{0})$, we have $-m < \beta_{n}+b_{1} < m$ except in the three cases as in the proof of \cref{prop:prod equal}. We consider the case where $\beta_{n} = b_{1} = k+1$. We can move $y^{m}$ to the Garside element:
    \begin{align*}
        u^{-1}wu &=_{G} y^{-b_{\tau}}x^{-a_{\tau}}\dots y^{-b_{1}}\cdot x^{\alpha_{1}}y^{\beta_{1}}\dots x^{\alpha_{n}}y^{m+1}\dots x^{a_{\tau}}y^{b_{\tau}} \\
        &=_{G} y^{-b_{\tau}}x^{-a_{\tau}}\dots y^{-b_{1}}\cdot x^{\alpha_{1}}y^{\beta_{1}}\dots x^{\alpha_{n}}yx^{a_{1}}\dots x^{a_{\tau}}y^{b_{\tau}}\Delta.
    \end{align*}
    Now $u^{-1}$ ends with $y^{-b_{1}} = y^{-(k+1)}$, which can be rewritten as 
    \begin{align*}
        u^{-1}wu &=_{G} y^{-b_{\tau}}x^{-a_{\tau}}\dots y^{-b_{1}}\cdot x^{\alpha_{1}}y^{\beta_{1}}\dots x^{\alpha_{n}}y^{m+1}\dots x^{a_{\tau}}y^{b_{\tau}} \\
        &=_{G} y^{-b_{\tau}}x^{-a_{\tau}}\dots y^{-m}y^{k}\cdot x^{\alpha_{1}}y^{\beta_{1}}\dots x^{\alpha_{n}}yx^{a_{1}}\dots x^{a_{\tau}}y^{b_{\tau}}\Delta \\
        &=_{G} y^{-b_{\tau}}x^{-a_{\tau}}\dots y^{k}\cdot x^{\alpha_{1}}y^{\beta_{1}}\dots x^{\alpha_{n}}yx^{a_{1}}\dots x^{a_{\tau}}y^{b_{\tau}},
    \end{align*}
    by moving $y^{-m}$ to cancel with the Garside element. Up to cyclic reduction, in order for $u^{-1}wu \in \overline{\mathcal{B}}$, we then have to take a split cyclic permutation, to move all $y^{-k}$ terms to the left of all $y^{k+1}$ terms.  
    The remaining cases follow a similar strategy. Otherwise, if $|\beta_{n}+b_{1}| \neq 0$, then we can cyclically reduce to $w$.
    
    Now suppose the matching terms consist of $x$ letters. Assuming potential free cancellations 
    \[ u^{-1}wu =_{G} y^{-b_{\tau}}x^{-a_{\tau}}\dots y^{-b_{1}}x^{-a_{1}}\cdot x^{\alpha_{1}}y^{\beta_{1}}\dots x^{\alpha_{n}}y^{\beta_{n}} \cdot x^{a_{1}}y^{b_{1}}\dots x^{a_{\tau}}y^{b_{\tau}}.
    \]
    Again we can assume $\alpha_{1}-a_{1} = \pm 2$. Suppose $\alpha_{1}-a_{1} = 2$, i.e. $\alpha_{1} = 1, a_{1} = -1$. We can move the $x^{2}$ term to the Garside, and then reverse this procedure by moving the $x^{a_{1}} = x^{-1}$ term to cancel with the Garside:
    \begin{align*}
        u^{-1}wu &=_{G} y^{-b_{\tau}}x^{-a_{\tau}}\dots y^{-b_{1}}x^{2}y^{\beta_{1}}\dots x^{\alpha_{n}}y^{\beta_{n}} \cdot x^{-1}y^{b_{1}}\dots x^{a_{\tau}}y^{b_{\tau}}\\
        &=_{G} y^{-b_{\tau}}x^{-a_{\tau}}\dots y^{-b_{1}}y^{\beta_{1}}\dots x^{\alpha_{n}}y^{\beta_{n}} \cdot x^{-1}y^{b_{1}}\dots x^{a_{\tau}}y^{b_{\tau}}\Delta\\
        &=_{G} y^{-b_{\tau}}x^{-a_{\tau}}\dots y^{-b_{1}}y^{\beta_{1}}\dots x^{\alpha_{n}}y^{\beta_{n}} \cdot x^{-2}xy^{b_{1}}\dots x^{a_{\tau}}y^{b_{\tau}}\Delta \\
        &=_{G} y^{-b_{\tau}}x^{-a_{\tau}}\dots y^{-b_{1}}y^{\beta_{1}}\dots x^{\alpha_{n}}y^{\beta_{n}} \cdot xy^{b_{1}}\dots x^{a_{\tau}}y^{b_{\tau}}.
    \end{align*}
    Up to cyclic reduction, in order for $u^{-1}wu \in \overline{\mathcal{B}}$, we have to take a split cyclic permutation, to move the $x^{-1}$ terms to the left of all $x$ terms. The case for $\alpha_{1}-a_{1} = -2$ is analogous. 
\end{proof}

\begin{cor}
    Modulo cyclic permutations, the set $\overline{\mathcal{A}}$ gives a set of minimal length conjugacy representatives. Modulo split cyclic permutations, the set $\overline{\mathcal{B}}$ gives a set of minimal length conjugacy representatives. 
\end{cor}

\begin{proof}
    We have shown that elements of $\overline{\mathcal{A}}$ are conjugacy geodesics, unique up to cyclic permutation by \cref{prop:prod equal}. Similarly elements of $\overline{\mathcal{B}}$ are conjugacy geodesics, unique up to a split cyclic permutation by \cref{prop:conj geo B}. It remains to show that any element in $G(m)$ is conjugate to an element represented by a word in $\overline{\mathcal{A}}$ or $\overline{\mathcal{B}}$. This follows immediately since any geodesic of the form in \cref{prop:geos} can be cyclically permuted to a word in either $\overline{\mathcal{A}}$ or $\overline{\mathcal{B}}$, using rewrite rules if necessary (for example \cref{eq:a non unique} 
or \cref{rewrite pairs}). 
\end{proof}

\begin{exmp}\label{example:G(3)}
    We consider examples of minimal length conjugacy representatives for 
    \[ G(3) = \langle x,y \mid x^{2} = y^{3} \rangle.
    \]
    Take the word $u_{1} = xyxyxy^{-1}\Delta^{3} \in \overline{\mathcal{A}}$, which is of Type $(1)$. Here $u_{1}$ is conjugate to $v_{1} = yxy^{-1}xyx\Delta^{3}$, which is a cyclic permutation of $u_{1}$ (up to the Garside element) by:
    \[ w^{-1}u_{1}w = (xyx)^{-1}\cdot xyxyxy^{-1}\Delta^{3}\cdot (xyx) =_{G} yxy^{-1}xyx\Delta^{3} = v_{1}.
    \]
    Now consider the word $u_{2} = x^{-1}yx^{-1}y^{-1}xy \in \overline{\mathcal{B}}$, which is of Type $(3^{0*})$. Here $u_{2}$ is conjugate to $v_{2} = y^{-1}x^{-1}yx^{-1}yx \in \overline{\mathcal{B}}$, which is a split cyclic permutation of $u_{2}$, by:
    \[ w^{-1}u_{2}w = \left(x^{-1}yx^{-1}\right)^{-1}\cdot x^{-1}yx^{-1}y^{-1}xy \cdot \left(x^{-1}yx^{-1}\right) =_{G} y^{-1}xyx^{-1}yx^{-1} =_{G} y^{-1}x^{-1}yx^{-1}yx = v_{2}.
    \]
\end{exmp}

\section{Even dihedral Artin groups: $G(2p)$}\label{even conj geos}
We now consider $G(m)$ for $m$ even, i.e. 
$G(m) \cong \langle x,y \; | \; y^{-1}x^{p}y = x^{p} \rangle,$
where $p=\frac{m}{2} \geq 2$. Let $\Delta = x^{p}$ and note the centre of $G(m)$ is generated by $\Delta$. When considering geodesics, we use normal forms derived from \cite{EdjJohn92} and similar to \cite{EdjJohn92}, we split cases for when $p$ is even or odd. 
\subsection{Even dihedral groups $G(2p)$: Case $p=2k\geq 4$}\label{sec:2k}
In this section we assume $p = 2k \geq 4$, and provide details of the case $p=2$ in \cref{subsec: p=2}. We collect all types of geodesics in Tables \ref{table:BS(2,2)} and \ref{tab:BS(p,p)}, so the reader may skip the details here.
\subsubsection{Classification of geodesics in $G(2p)$, $p=2k\geq 4$ }
For any word in $G(m)=G(4k)$, we can apply the following rewrite rules together with free reduction without changing the group element represented:
\begin{enumerate}
    \item[RR1:] Collect any power of the central term $x^{2k}$ to the right.
    \item[RR2:] If $\sigma > 0 > \delta$ and $\sigma + |\delta| > 2k$, then replace a word $ux^{\delta}vx^{\sigma}w$ 
    ($ux^{\sigma}vx^{\delta}w$ resp.) by $ux^{\delta+2k}vx^{\sigma-2k}w$ ($ux^{\sigma-2k}vx^{\delta+2k}w$ resp.).
    \item[RR3:] If $\sigma > 0 > \delta$ and $\sigma + |\delta| = 2k$, then replace a word $ux^{\delta}vx^{\sigma}w$ by $ux^{\sigma}vx^{\delta}w$.
\end{enumerate}
Note that RR1 and RR3 preserve word length, whilst RR2 decreases word length. 

\begin{prop}\label{prop:geos 2k}
    Every element in $G(4k)$ has a geodesic representative in one of the following (freely reduced) forms:
    \begin{enumerate}
        \item[Type 1.] $(x^{\mu_{1}}y^{\ep_{1}}\dots x^{\mu_{n}}y^{\ep_{n}}\Delta^{c})^{\pm 1}$ where $c, n \geq 1$, $\mu_{1} \in [0, 2k-1] \cap \Z$,  $\mu_{i} \in [1, 2k-1] \cap \Z$ $(2 \leq i \leq n)$, $\ep_{i} \in \Z$ $(1 \leq i \leq n)$, $\ep_{i} \neq 0$ $(1 \leq i \leq n-1)$. 
        \item[Type 2.] $[(x^{\mu_{1j}}y^{\ep_{1j}}\dots y^{\ep_{n_{j}j}})x^{j}(y^{\delta_{1j}}x^{\gamma_{1j}}\dots x^{\gamma_{m_{j}j}})]^{\pm 1}$ for some $k \leq j \leq 2k-1$, where $n_{j}, m_{j} \geq 0$, $\ep_{sj}, \delta_{tj} \in \Z_{\neq 0}$ $(1 \leq s \leq n_{j}, 1 \leq t \leq m_{j})$.
        Three separate cases for different values of $j$:
        \begin{enumerate}
            \item If $j = k$, $\mu_{sk} \in [-(k-1), k] \cap \Z_{\neq 0}, \; \gamma_{tk} \in [-k, k] \cap \Z_{\neq 0}$ $(1 \leq s \leq n_{j}, 1 \leq t \leq m_{j})$.
            \item If $j = 2k-1$, $\mu_{s(2k-1)} \in [1, 2k-1] \cap \Z, \; \gamma_{t(2k-1)} \in [-1, 2k-2] \cap \Z_{\neq 0}$ \\ $(1 \leq s \leq n_{j}, 1 \leq t \leq m_{j})$.
            \item If $k < j < 2k-1$, $\mu_{sj} \in [-(2k-j-1), j] \cap \Z_{\neq 0}, \; \gamma_{tj} \in [-(2k-j), j-1] \cap \Z_{\neq 0}$ $(1 \leq s \leq n_{j}, 1 \leq t \leq m_{j})$.
        \end{enumerate}
        In all cases, $x^{\mu_{1j}}, x^{\gamma_{m_{j}j}}$ may also equal 0.
        \item[Type 3.] $x^{\alpha}y^{\ep_{1}}x^{\nu_{1}}\dots x^{\nu_{n}}$ where $n \geq 0$, $\ep_{i} \in \Z_{\neq 0}$ $(1 \leq i \leq n)$, $\alpha, \nu_{n} \in [-(k-1), k-1] \cap \Z$, $\nu_{i} \in [-(k-1), k-1] \cap \Z_{\neq 0}$ $(1 \leq i \leq n-1)$.
    \end{enumerate}
    Moreover, Types 1,2 and 3 are mutually exclusive.
\end{prop}
The proof of this result can be found in \cref{appen2: even}. We now propose an alternative form for Type 2 geodesics from \cref{prop:geos 2k}. The idea is to gather all $x^{j}$ and $x^{-(2k-j)}$ terms, and note that all $x^{j}$ terms must precede all $x^{-(2k-j)}$ terms in the word due to RR3. 
\begin{defn}\label{rmk:alternative}
    Let $\tau = \tau_{1}+\tau_{2}$. Any Type 2 geodesic from \cref{prop:geos 2k} can be written in the form $w$ as follows, for all $k \leq j \leq 2k-1$:
    \begin{equation}\label{eq:any j}
    w = A_{1}x^{j}A_{2}x^{j}\dots A_{\tau_{1}}x^{j}A_{\tau_{1}+1}x^{-(2k-j)}A_{\tau_{1} + 2}x^{-(2k-j)}\dots x^{-(2k-j)}A_{\tau_{1} + \tau_{2} + 1},
\end{equation}
    where $A_{s}$ $(1 \leq s \leq \tau + 1)$ are reduced words over $x,y$, with any maximal subword of the form $x^{\alpha}$ in $ A_{s}$ satisfying $\alpha \in [-2k+j+1, j-1] \cap \mathbb{Z}_{\neq 0}$. Here $\tau_{1} \geq 1, \tau_{2} \geq 0$, $A_{1}$ ($A_{\tau + 1}$ resp.) is either empty or ends (starts) with $y^{\beta}$ for some $\beta \in \Z_{\neq 0}$, and $A_{s}$ blocks start and end with $y^{\beta}$ for some $\beta \in \Z_{\neq 0}$ $(2 \leq s \leq \tau)$.

     In this form, we split Type 2 geodesics into two disjoint sets, denoted Type $2a \; \sqcup$ Type $2b$, where Type $2a$ is the set of all Type 2 geodesics where $\tau_{2} = 0$, i.e. of the form 
     \[ w = A_{1}x^{j}A_{2}x^{j}\dots A_{\tau_{1}}x^{j}A_{\tau_{1}+1},
     \]
    and Type $2b$ is the set of all remaining Type 2 geodesics with $\tau_{2} \geq 1$, i.e. of the form
    \[ w =_{G} A_{1}x^{j}A_{2}x^{j}\dots A_{\tau_{1}}x^{j}A_{\tau_{1}+1}x^{-(2k-j)}A_{\tau_{1} + 2}x^{-(2k-j)}\dots x^{-(2k-j)}A_{\tau_{1} + \tau_{2} + 1}.
    \]

\end{defn}

\subsubsection{Minimal length conjugacy representatives in $G(4k)$}
We now describe conjugacy geodesic representatives. To do this, we split geodesics into sets $\mathcal{C}$ and $\mathcal{D}$: Let 
     $\mathcal{C}_{1}$ be the set of all Type 1 geodesics,
   $\mathcal{C}_{2}$ be the set of all Type $2a$ geodesics, and
     $\mathcal{C}_{3}$ be the set of all Type 3 geodesics, 
and define $\mathcal{C}$ as $\mathcal{C} := \mathcal{C}_{1} \sqcup \mathcal{C}_{2} \sqcup \mathcal{C}_{3,}$ whilst
$\mathcal{D}$ is the set of all remaining geodesics, i.e. all Type $2b$ geodesics. The proof of the following result can be found in \cref{appen2: even}.

\begin{prop}\label{prop:even unique non unique}
    The set $\mathcal{C}$ consists of unique geodesics only, whereas $\mathcal{D}$ consists of non-unique geodesics. 
\end{prop}

\begin{defn}\label{defn:C bar}
    We define the following subset of $\mathcal{C}$:
    \[
    \overline{\mathcal{C}} = 
    \begin{cases}
        w = (x^{\mu_{1j}}y^{\ep_{1}}\dots x^{\mu_{n}}y^{\ep_{n}}\Delta^{c})^{\pm 1} & w \in \mathcal{C}_{1}, \; \mu_{1j} = 0 \Leftrightarrow \ep_{n} = 0, \\
        w = x^{\mu_{1j}}y^{\ep_{1j}}\dots y^{\ep_{n_{j}j}} & w \in \mathcal{C}_{2}, \; \mu_{1j} = 0 \Leftrightarrow \ep_{n_{j}j} = 0,\\
        w = x^{\mu_{1}}y^{\ep_{1}}\dots y^{\ep_{n}} & w \in \mathcal{C}_{3}, \; \mu_{1} = 0 \Leftrightarrow \ep_{n} = 0.
    \end{cases}
    \]
For notation, we let $\overline{\mathcal{C}} = \overline{\mathcal{C}_{1}} \sqcup \overline{\mathcal{C}_{2}} \sqcup \overline{\mathcal{C}_{3}}$, where $\overline{\mathcal{C}_{i}}=\mathcal{C}_{i} \cap \overline{\mathcal{C}}$, for all $1 \leq i \leq 3$.
\end{defn}

\begin{prop}\label{prop:conj geos cyc perm}
Let $w,v \in \overline{\mathcal{C}}$ have the form
    \[ w = x^{\alpha_{1}}y^{\beta_{1}}\dots x^{\alpha_{n_{w}}}y^{\beta_{n_{w}}}\Delta^{r}, \quad v = x^{\sigma_{1}}y^{\gamma_{1}}\dots x^{\sigma_{n_{v}}}y^{\gamma_{n_{v}}}\Delta^{q}.
    \]
    Then $w \sim v$ if and only if $n_{w}=n_{v}, \; r=q$ and the words $x^{\alpha_{1}}y^{\beta_{1}}\dots x^{\alpha_{n_{w}}}y^{\beta_{n_{w}}}$ and $x^{\sigma_{1}}y^{\gamma_{1}}\dots x^{\sigma_{n_{v}}}y^{\gamma_{n_{v}}}$ are cyclic permutations of each other.  
\end{prop}
\begin{proof}
The proof follows a similar method to \cref{prop:prod equal}. Suppose $w \in \overline{\mathcal{C}_{1}}$ and, without loss of generality, assume $r>0$. We consider conjugating $w$ by an element $u = x^{a_{1}}y^{b_{1}}\dots x^{a_{\tau}}y^{b_{\tau}}\Delta^{c}$, where $a_{i} \in [-(2k-1), 2k-1] \cap \Z_{\neq 0}, \; b_{j} \in \Z_{\neq 0}, \; (1 \leq i \leq \tau, 1 \leq j \leq \tau),$ and $c \in \Z_{\neq 0}.$ Note $u$ need not be geodesic. 

We can immediately ignore concatenation of $y$ letters as these do not alter the Garside element. Indeed if $y$ letters concatenate, then this is equivalent to a cyclic permutation. Therefore, up to free cancellation where possible, suppose
\[ u^{-1}wu =_{G} y^{-b_{\tau}}\dots x^{-a_{1}}x^{\alpha_{1}}\dots y^{\beta_{n_{w}}}x^{a_{1}}\dots y^{b_{\tau}}\Delta^{r}.
\]
Since $w \in \overline{\mathcal{C}_{1}}$ and $r>0$, we can assume $\alpha_{1} \in \{1, 2, \dots, 2k-1 \}$, and so $-(2k-2) \leq \alpha_{1}-a_{1} \leq 4k-2$. Suppose $\alpha_{1}-a_{1} = 2k + d$ for some $d \in \{0,1,\dots, 2k-2\}$. We can move $x^{2k}$ to the Garside element as follows:
\[ u^{-1}wu =_{G} y^{-b_{\tau}}\dots x^{d}x^{2k}\dots y^{\beta_{n_{w}}}x^{a_{1}}\dots y^{b_{\tau}}\Delta^{r} =_{G} y^{-b_{\tau}}\dots x^{d}\dots y^{\beta_{n_{w}}}x^{\alpha_{1}-2k-d}\dots y^{b_{\tau}}\Delta^{r+1}.
\]
We can then reverse this procedure by moving $x^{-2k}$ from the $x^{a_{1}}$ term:
\begin{align*}
    u^{-1}wu &=_{G} y^{-b_{\tau}}\dots x^{d}x^{2k}\dots y^{\beta_{n_{w}}}x^{a_{1}}\dots y^{b_{\tau}}\Delta^{r} \\
    &=_{G} y^{-b_{\tau}}\dots x^{d}\dots y^{\beta_{n_{w}}}x^{\alpha_{1}-2k-d}\dots y^{b_{\tau}}\Delta^{r+1} \\
    &=_{G} y^{-b_{\tau}}\dots x^{d}\dots y^{\beta_{n_{w}}}x^{\alpha_{1}-d}\dots y^{b_{\tau}}\Delta^{r}.
\end{align*}
Now $u^{-1}wu$ is a cyclic permutation of $w$, up to cyclic reduction, as required. For all other values of $\alpha_{1} - a_{1}$, we have that $u^{-1}wu$ is equal to $w$ up to cyclic reduction. A symmetric argument holds when $x$ letters concatenate in $wu$, which completes the proof for words in $\overline{\mathcal{C}_{1}}$. The proof for $\overline{\mathcal{C}_{2}}$ and $\overline{\mathcal{C}_{3}}$ follow a similar case by case argument. 
\end{proof}
Now we consider cyclic geodesics in the set $\mathcal{D}$. Recall $\mathcal{D}$ is the set of all Type 2 geodesics where $m_{j} \geq 1$. For words in $\mathcal{D}$, we do not have the same equivalent geodesic types up to cyclic permutation. For example, recall \cref{eq:any j} and suppose we have
\[ w = A_{1}x^{j}A_{2}x^{j}\dots A_{\tau_{1}}x^{j}A_{\tau_{1}+1}x^{-(2k-j)}A_{\tau_{1} + 2}x^{-(2k-j)}\dots x^{-(2k-j)}A_{\tau_{1} + \tau_{2} + 1}.
\]
Consider the cyclic permutation
\[ w' = x^{-(2k-j)}A_{\tau + 1}A_{1}x^{j}A_{2}x^{j}\dots A_{\tau_{1}}x^{j}A_{\tau_{1} + 1}x^{-(2k-j)}A_{\tau_{1} + 2}x^{-(2k-j)}\dots A_{\tau}.
\]
This word is no longer of the correct geodesic form, since we can apply RR3 to $w'$ to get 
\[ w' =_{G} x^{j}A_{\tau + 1}A_{1}x^{j}A_{2}x^{j}\dots A_{\tau_{1}}x^{-(2k-j)}A_{\tau_{1} + 1}x^{-(2k-j)}A_{\tau_{1}+2}x^{-(2k-j)}\dots A_{\tau}.
\]
This rewrite can only occur when we cyclically permute an $x^{j}$ or $x^{-(2k-j)}$ term. Therefore any cyclic permutation of $w$ written in geodesic normal form is a cyclic permutation of elements in $(A_{1},\dots, A_{\tau + 1})$, with the $A_i$s separated by $X = \left(x^{j}, \dots x^{j}, x^{-(2k-j)}, x^{-(2k-j)}, \dots x^{-(2k-j)}\right)$, where $X$ is fixed. This idea holds for all $k \leq j \leq 2k-1$. 
\begin{defn}\label{def:even_split_reps}
Let $\tau = \tau_{1}+\tau_{2}$, let $k \leq j \leq 2k-1$.
\begin{itemize}
    \item[(i)] 
    For $\tau_{1}, \tau_{2} \geq 1$, define a tuple $X_{\tau_{1},\tau_{2}}$ of powers of $x$ with $a_i \in \{j, -(2k-j)\}$ where:
    \[X_{\tau_{1},\tau_{2}} = (x^{a_{1}}, x^{a_{2}}, \dots, x^{a_{\tau}}) = (\underbrace{x^{j}, x^{j}, \dots, x^{j}}_{\tau_{1}}, \underbrace{x^{-(2k-j)}, x^{-(2k-j)}, \dots, x^{-(2k-j)}}_{\tau_{2}}).
    \]
    Consider a geodesic $w \in \mathcal{D}$ based on $X_{\tau_{1},\tau_{2}}$ of the form
    \begin{align*}
        w &= A_{1}x^{a_{1}}A_{2}x^{a_{2}}\dots A_{\tau_{1}}x^{a_{\tau_{1}}}A_{\tau_{1}+1}x^{a_{\tau_{1}+1}}A_{\tau_{1}+2}x^{a_{\tau_{1}+2}}\dots x^{a_{\tau}}A_{\tau + 1} \\
        &= A_{1}x^{j}A_{2}x^{j}\dots A_{\tau_{1}}x^{j}A_{\tau_{1}+1}x^{-(2k-j)}A_{\tau_{1}+2}x^{-(2k-j)}\dots x^{-(2k-j)}A_{\tau + 1},
    \end{align*}
    where the $A_{s}$ blocks are all words over $x,y$ satisfying the conditions from \cref{eq:any j}.
  \item[(ii)]  
    A \emph{split cyclic permutation} of $w$ is one where the tuple $X_{\tau_{1},\tau_{2}}$ is preserved and $A_{s}$ blocks are cyclically permuted: it is a geodesic $w' \in \mathcal{D}$ as in: \\
    Case 1: all $A_{s}$ blocks are preserved, and for $1 \leq t \leq \tau+1$
    \[ w' = x^{a_{1}}A_{t}x^{a_{2}}A_{t+1}\dots x^{a_{\tau - t + 2}}A_{\tau + 1}A_{1}x^{a_{\tau -t+3}}A_{2}\dots x^{a_{\tau}}A_{t-1}.
    \]
    Case 2: one of the $A_{s}$ blocks is split, that is, for $1 \leq t \leq \tau+1$ write $A_{t} = A_{t_{1}}A_{t_{2}}$ as a reduced product of prefix and suffix. Then 
    \[  w' = A_{t_{2}}x^{a_{1}}A_{t+1}x^{a_{2}}\dots x^{a_{\tau-t+1}}A_{\tau + 1}A_{1}x^{a_{\tau - t + 2}}A_{2}\dots x^{a_{\tau}}A_{t_{1}}.
    \]
\item[(iii)]
Define a set $\overline{\mathcal{D}}$ to be all words in $\mathcal{D}$ of the form in \cref{rmk:alternative}, such that words start and end with opposite factors, i.e. if $A_{1}$ starts with an $x$-term ($y$-term resp.), then $A_{\tau+1}$ must end with a $y$-term ($x$-term resp.). Note $A_{1}$ and $A_{\tau + 1}$ may be empty, but they cannot both be empty - if this were the case, we could cyclically permute the $x^{-(2k-j)}$ which cancels and adds to the Garside element when concatenated with $x^{j}$. 
\end{itemize}
\end{defn}
\begin{prop}\label{prop:sets B}
    Let $w,v \in \overline{\mathcal{D}}$ be based on tuples $X_{\tau_{1},\tau_{2}}$ and $X_{\sigma_{1},\sigma_{2}}$. 
     Then $w \sim v$ if and only if $X_{\tau_{1},\tau_{2}} = X_{\sigma_{1},\sigma_{2}}$, and $w$ and $v$ are equal up to a split cyclic permutation. 
\end{prop}

\begin{proof}
The proof is similar to \cref{prop:conj geos cyc perm}, and note by the discussion above that if we consider a cyclic permutation of $w$, we need to take a split cyclic permutation for $v \in \overline{\mathcal{D}}$. Consider conjugating $w = x^{\alpha_{1}}y^{\beta_{1}}\dots y^{\beta_{n}}$ by an element $u = x^{a_{1}}y^{b_{1}}\dots x^{a_{\tau}}y^{b_{\tau}}\Delta^{c}$, where $a_{i} \in [-(2k-1), 2k-1] \cap \Z_{\neq 0},$ $b_{j} \in \Z_{\neq 0} \; (1 \leq i \leq \tau, 1 \leq j \leq \tau),$ and $c \in \Z_{\neq 0}$. We assume up to free cancellation that 
$ u^{-1}wu =_{G} y^{-b_{\tau}}\dots x^{-a_{1}}x^{\alpha_{1}}\dots y^{\beta_{n}}x^{a_{1}}\dots y^{b_{\tau}}.$
Since $w \in \overline{\mathcal{D}}$, we can assume $\alpha_{1} \in [-(2k-j-1), j] \cap \Z_{\neq 0}$. Also $a_{1} \in [-(2k-1), 2k-1] \cap \Z_{\neq 0}$, and so $-4k+j+2 \leq \alpha_{1}-a_{1} \leq 2k + j -1$. Suppose $\alpha_{1}-a_{1} = 2k+c$ where $c \in \{0, \dots, j-1\}$. Then
\begin{align*}
    u^{-1}wu &=_{G} y^{-b_{\tau}}\dots x^{c}x^{2k}\dots y^{\beta_{n}}x^{a_{1}}\dots y^{b_{\tau}} \\
    &=_{G} y^{-b_{\tau}}\dots x^{c}\dots y^{\beta_{n}}x^{\alpha_{1}-2k-c}\dots y^{b_{\tau}}\Delta \\
    &=_{G} y^{-b_{\tau}}\dots x^{c}\dots y^{\beta_{n}}x^{\alpha_{1}-c}\dots y^{b_{\tau}}.
\end{align*}
Up to cyclic reduction, in order for $u^{-1}wu \in \overline{\mathcal{D}}$, we then have to apply RR3, which is equivalent to a split cyclic permutation of $w$. The proof for $\alpha_{1} - a_{1} \leq -2k$ follows a symmetric proof. For all other values of $\alpha_{1}-a_{1}$, we have that $u^{-1}wu$ is equal to $w$ up to cyclic reduction. A symmetric argument holds when $x$ letters concatenate in $wu$, which completes the proof. 
\end{proof}

\begin{cor}\label{cor:CDodd}
    Modulo cyclic permutations, the set $\overline{\mathcal{C}}$ gives a set of minimal length conjugacy representatives. Modulo split cyclic permutations, the set $\overline{\mathcal{D}}$ gives a set of minimal length conjugacy representatives. 
\end{cor}

\begin{proof}
    We have shown that the elements of $\overline{\mathcal{C}}$ are conjugacy geodesics, unique up to cyclic permutation by \cref{prop:conj geos cyc perm}. Similarly the elements of $\overline{\mathcal{D}}$ are conjugacy geodesics, unique up to a split cyclic permutation by \cref{prop:sets B}. It remains to show that any element in $G(m)$ is conjugate to an element represented by a word in $\overline{\mathcal{C}}$ or $\overline{\mathcal{D}}$. This follows immediately since any geodesic of the form in \cref{prop:geos 2k} can be cyclically permuted to an element of either $\overline{\mathcal{C}}$ or $\overline{\mathcal{D}}$, using the rewrite rules RR1 and RR3 if necessary. 
\end{proof}
\subsubsection{The even dihedral group $G(4)$: $p = 2k = 2$}\label{subsec: p=2}
The classification of geodesics follows similarly to \cref{prop:geos 2k}, with some small differences when $k=1$. This can be derived from \cite[Section 3]{EdjJohn92} and we collect the types in Table \ref{table:BS(2,2)}. Type (3) geodesics from \cref{prop:geos 2k} reduce to words of the form $y^{\beta}$ where $\beta \in \Z_{\neq 0}$. Also, similar to \cref{rmk:alternative}, any Type 2 geodesic can be written in the form
\[  w = A_{1}xA_{2}x\dots A_{\tau_{1}}xA_{\tau_{1}+1}x^{-1}A_{\tau_{1} + 2}x^{-1}\dots x^{-1}A_{\tau_{1} + \tau_{2} + 1},
\]
where each $A_{s}=y^t$ ($t \in \Z_{\neq 0}, 1 \leq s \leq \tau_{1}+\tau_{2}+1$). Note that if $\tau_{1} = 0$ or $\tau_{2} = 0$, then $w$ is in fact a Type 1 geodesic. This implies that the mutually exclusive geodesic normal forms are:
\begin{enumerate}
    \item[(i)] Type 1 geodesics, which are unique,
    \item[(ii)] Type 2 geodesics such that $\tau_{1}, \tau_{2} \geq 1$, which are non-unique, and
    \item [(iii)] Type 3 geodesics, which are unique.
\end{enumerate}
As before, let $\overline{\mathcal{C}}$ consist of Type 1 geodesics which start and end with opposite letters, and Type 3 geodesics. Let $\overline{\mathcal{D}}$ consist of Type 2 geodesics such that $\tau_{1}, \tau_{2} \geq 1$, which start and end with opposite letters. Analogous results to \cref{prop:conj geos cyc perm} and 
\cref{prop:sets B} follow using the same proof method. 

\begin{exmp}
    We consider examples of minimal length conjugacy representatives for the group
    $G(4) \cong \mathrm{BS}(2,2)$. First consider the word $u_{1} = xy^{5}xy^{-3}xy^{2}\Delta^{2} \in \overline{\mathcal{C}}$, which is of Type (1). Here $u_{1}$ is conjugate to the word $v_{1} = y^{-3}xy^{2}xy^{5}x\Delta^{2} \in \overline{\mathcal{C}}$, which is a cyclic permutation of $u_{1}$ (up to the Garside element), by the following relations:
    \[ w^{-1}u_{1}w = \left(xy^{5}x\right)^{-1}\cdot xy^{5}xy^{-3}xy^{2}\Delta^{2}\cdot xy^{5}x =_{G} y^{-3}xy^{2}xy^{5}x\Delta^{2} = v_{1}.
    \]
    Now consider the word $u_{2} = yxy^{-2}xy^{11}x^{-1} \in \overline{\mathcal{D}}$, which is of Type (2). Here $u_{2}$ is conjugate to $v_{2} = y^{-2}xy^{11}xyx^{-1} \in \overline{\mathcal{D}}$, which is a split cyclic permutation of $u_{2}$, by the following relations:
    \[ w^{-1}u_{2}w = (yx)^{-1}\cdot yxy^{-2}xy^{11}x^{-1} \cdot (yx) =_{G} y^{-2}xy^{11}x^{-1}yx =_{G} y^{-2}xy^{11}xyx^{-1} = v_{2}.
    \]
\end{exmp}

\subsection{The even dihedral group $G(4k+2)$: Case $p = 2k+1$}\label{sec:G(4k+2)}
From \cite{EdjJohn92}, the geodesic normal form is almost identical to the $p=2k$ case, with some minor changes in the parameters. The same proofs will follow through from \cref{sec:2k}, and so we give a summary of the key results only (see also Table \ref{tab:p=2k+1}). 
\begin{prop}\label{prop:geos 2k+1}
    Every element in $G(4k+2)$ has a geodesic representative of one of:
    \begin{enumerate}
        \item[Type 1.] $(x^{\mu_{1}}y^{\ep_{1}}\dots x^{\mu_{n}}y^{\ep_{n}}\Delta^{c})^{\pm 1}$ where $c, n \geq 1$, $\mu_{1} \in [0, 2k] \cap \Z$,  $\mu_{i} \in [1, 2k] \cap \Z$ $(2 \leq i \leq n)$, $\ep_{i} \in \Z$ $(1 \leq i \leq n)$, $\ep_{i} \neq 0$ $(1 \leq i \leq n-1)$.  
        \item[Type 2.] $[(x^{\mu_{1j}}y^{\ep_{1j}}\dots y^{\ep_{n_{j}j}})x^{j}(y^{\delta_{1j}}x^{\gamma_{1j}}\dots x^{\gamma_{m_{j}j}})]^{\pm 1}$ for some $k+1 \leq j \leq 2k$, where $n_{j}, m_{j} \geq 0$, $\ep_{sj}, \delta_{tj} \in \Z_{\neq 0}$ $(1 \leq s \leq n_{j}, 1 \leq t \leq m_{j})$.
        Three separate cases for different values of $j$:
        \begin{enumerate}
            \item If $j = k+1$, $\mu_{s(k+1)} \in [-(k-1), k+1] \cap \Z_{\neq 0}, \; \gamma_{t(k+1)} \in [-k, k] \cap \Z_{\neq 0}$ $(1 \leq s \leq n_{j}, 1 \leq t \leq m_{j})$.
            \item If $j = 2k$, $\mu_{s(2k)} \in [1, 2k] \cap \Z, \gamma_{t(2k)} \in [-1, 2k-1] \cap \Z_{\neq 0}$ $(1 \leq s \leq n_{j}, 1 \leq t \leq m_{j})$.
            \item If $k < j < 2k$, $\mu_{sj} \in [-(2k-j), j] \cap \Z_{\neq 0}, \; \gamma_{tj} \in [-(2k-j+1), j-1] \cap \Z_{\neq 0}$ $(1 \leq s \leq n_{j}, 1 \leq t \leq m_{j})$.
        \end{enumerate}
        In all cases, $x^{\mu_{1j}}, x^{\gamma_{m_{j}j}}$ may also equal 0.
        \item[Type 3.] $x^{\alpha}y^{\ep_{1}}x^{\nu_{1}}\dots x^{\nu_{n}}$ where $n \geq 0$, $\ep_{i} \in \Z_{\neq 0}$ $(1 \leq i \leq n)$, $\alpha, \nu_{n} \in [-k, k] \cap \Z$, $\nu_{i} \in [-k, k] \cap \Z_{\neq 0}$ $(1 \leq i \leq n-1)$.
    \end{enumerate}
     Moreover, Types 1,2 and 3 are mutually exclusive.
\end{prop}
\begin{defn}\label{rmk:alternative1}
    Let $\tau = \tau_{1}+\tau_{2}$. Any Type 2 geodesic from \cref{prop:geos 2k+1} can be written in the form $w$ as follows, for all $k \leq j \leq 2k-1$.
    \begin{equation}\label{eq:any j2}
    w = A_{1}x^{j}A_{2}x^{j}\dots A_{\tau_{1}}x^{j}A_{\tau_{1}+1}x^{-(2k-j+1)}A_{\tau_{1}+2}x^{-(2k-j+1)}\dots x^{-(2k-j+1)}A_{\tau_{1}+\tau_{2}+1},
\end{equation}
    where $A_{s}$ $(1 \leq s \leq \tau+1)$ are reduced words over $x,y$, with any maximal subword of the form $x^{\alpha}$ in $A_{s}$ satisfying $\alpha \in [-(2k-j), j-1] \cap \Z_{\neq 0}$. Here $\tau_{1} \geq 1, \tau_{2} \geq 2$, $A_{1}$ ($A_{\tau+1}$ resp.) is either empty or ends (starts) with  $y^{\beta}$ for some $\beta \in \Z_{\neq 0}$, and $A_{s}$ blocks start and end with $y^{\beta}$ for some $\beta \in \Z_{\neq 0}$ $(2 \leq s \leq \tau)$.

    In this form, we split Type 2 geodesics into two disjoint sets, denoted Type $2a \; \sqcup$ Type $2b$, where Type $2a$ is the set of all Type 2 geodesics where $\tau_{2} = 0$, i.e. of the form 
     \[ w = A_{1}x^{j}A_{2}x^{j}\dots A_{\tau_{1}}x^{j}A_{\tau_{1}+1},
     \]
    and Type $2b$ is the set of Type 2 geodesics with $\tau_{2} \geq 1$, i.e. of the form
    \[ w =_{G} A_{1}x^{j}A_{2}x^{j}\dots A_{\tau_{1}}x^{j}A_{\tau_{1}+1}x^{-(2k-j+1)}A_{\tau_{1} + 2}x^{-(2k-j+1)}\dots x^{-(2k-j+1)}A_{\tau_{1}+\tau_{2}+1}.
    \] 
\comm{
    $(a)$ If $j = k+1$, then
    \begin{equation}\label{eq:j=k}
    w = (A_{1}x^{k+1}A_{2}x^{k+1}\dots A_{n}x^{k+1})(B_{1}x^{-k}B_{2}x^{-k}\dots x^{-k}B_{m}),
\end{equation}
where $A_{s}, B_{t}, (1 \leq s \leq n, 1 \leq t \leq m),$ are reduced words over $x,y$, with any maximal subword of the form $x^{\alpha}$ in $A_{s}, B_{t}$ satisfying $\alpha \in \{\pm 1, \dots, \pm (k-1) \} \; (1 \leq s \leq n, 1 \leq t \leq m)$.

$(b)$ If $j=2k$ then
\begin{equation}
    w = (A_{1}x^{2k}A_{2}x^{2k}\dots A_{n}x^{2k})(B_{1}x^{-1}B_{2}x^{-1}\dots x^{-1}B_{m}),
\end{equation}
where $A_{s}, B_{t}, (1 \leq s \leq n, 1 \leq t \leq m),$ are reduced words over $x,y$, with any maximal subword of the form $x^{\alpha}$ in $A_{s}, B_{t}$ satisfying $\alpha \in \{1, \dots, 2k-1\} \; (1 \leq s \leq n, 1 \leq t \leq m)$.

$(c)$ If $k+1 < j < 2k$ then
\begin{equation}\label{eq:any j2}
    w = (A_{1}x^{j}A_{2}x^{j}\dots A_{n}x^{j})(B_{1}x^{-(2k-j-1)}B_{2}x^{-(2k-j-1)}\dots x^{-(2k-j-1)}B_{m}),
\end{equation}
where $A_{s}, B_{t}, (1 \leq s \leq n, 1 \leq t \leq m),$ are reduced words over $x,y$, with any maximal subword of the form $x^{\alpha}$ in $A_{s}, B_{t}$ satisfying $\alpha \in \{\pm 1, \dots \pm 2k-j, 2k-j+1, \dots, j-1\} \; (1 \leq s \leq n, 1 \leq t \leq m)$.}
\end{defn}
Let $\mathcal{C}_{1}$ be the set of all Type 1 geodesics, $\mathcal{C}_{2}$ be the set of all Type $2a$ geodesics, and $\mathcal{C}_{3}$ be the set of all Type 3 geodesics, and define $\mathcal{C}, \mathcal{D}$ as 
\[ \mathcal{C} := \mathcal{C}_{1} \sqcup \mathcal{C}_{2} \sqcup \mathcal{C}_{3}
\]
while $\mathcal{D}$ is the set of all remaining geodesics, i.e. all Type $2b$ geodesics.
\begin{prop}
    The set $\mathcal{C}$ consists of unique geodesics, and $\mathcal{D}$ of non-unique geodesics. 
\end{prop}
Define the following subset of $\mathcal{C}$:
    \[
    \overline{\mathcal{C}} = 
    \begin{cases}
        w = (x^{\mu_{1j}}y^{\ep_{1}}\dots x^{\mu_{n}}y^{\ep_{n}}\Delta^{c})^{\pm 1} & w \in \mathcal{C}_{1}, \; \mu_{1j} = 0 \Leftrightarrow \ep_{n} = 0, \\
        w = x^{\mu_{1j}}y^{\ep_{1j}}\dots y^{\ep_{n_{j}j}} & w \in \mathcal{C}_{2}, \; \mu_{1j} = 0 \Leftrightarrow \ep_{n_{j}j} = 0,\\
        w = x^{\mu_{1}}y^{\ep_{1}}\dots y^{\ep_{n}} & w \in \mathcal{C}_{3}, \; \mu_{1} = 0 \Leftrightarrow \ep_{n} = 0.
    \end{cases}
    \]
Define $\overline{\mathcal{D}}$ to be all words in $\mathcal{D}$ of the form in \cref{rmk:alternative1}, such that words start and end with opposite factors.
\begin{cor}\label{cor:CDeven}
    Modulo cyclic permutations, the set $\overline{\mathcal{C}}$ gives a set of minimal length conjugacy representatives. Modulo split cyclic permutations, the set $\overline{\mathcal{D}}$ gives a set of minimal length conjugacy representatives. 
\end{cor}

\section{Conjugacy growth}\label{growth}
In this section we show that the conjugacy growth series of dihedral Artin groups is transcendental with respect to the free product generating set $X = \{x,y\}$. This follows from determining the asymptotics of conjugacy growth $c(n)$ using the conjugacy geodesic representatives found in \cref{sec:odd conj geos} and \ref{even conj geos}, which are collected in Tables 1--4. 
  The general strategy for computing the asymptotics for $c(n)$ is to use the fact that 
  finding conjugacy class representatives amounts, in most cases (but not all!), to picking words up to cyclic permutations.
\begin{defn}\label{def:eq}
    A sequence $a_n$, $n \geq 1$, has asymptotics $f(n)$, denoted as $a_n \sim f(n)$, if there exist two constants $c_1, c_2$ such that $c_1 f(n) \leq a_n \leq c_2 f(n)$ for all $n \geq 1$.
\end{defn}
\begin{exmp}\label{ex:standard}
In $G(3)$ (see also Example \ref{example:G(3)}) the six words $xyxyxy^{-1}$, $yxyxy^{-1}x$, $xyxy^{-1}xy$, $\dots$, $y^{-1}xyxyx$ are all the conjugacy geodesics in $[xyxyxy^{-1}]$ by Prop. \ref{prop:prod equal}, are of Type $(3^{0+}U)$, and are cyclic permutations of each other. It suffices to pick one of these $6$ words as representative for the conjugacy class $[xyxyxy^{-1}]$ and call this a `standard' case.
\end{exmp}
In the `standard' cases, when the geodesics are unique and conjugacy is essentially cyclic permutation of letters, there are  $\sim c\frac{\alpha^n}{n}$ representatives of length $n$, up to cyclic permutations, where $\alpha^n$ is the asymptotics for the number of words of length $n$ in a set like (\ref{set:alt}), and $c$ is a constant. It will be occasionally convenient to take representatives up to cyclic permutation of `syllables' $x^{\sigma}y^{\delta}$ instead of single letters, with similar arguments but different constants. 

However, for several `non-standard' types of conjugacy geodesics, cyclic permutation representatives do not capture conjugacy correctly, and split cyclic permutation representatives must be considered (recall Corollaries \ref{cor:CDodd}, \ref{cor:CDeven}). This makes the computations significantly harder.  

\begin{exmp}\label{ex:non-standard}
In $G(3)$, taking all the words of length $n$ in a set of the form (\ref{set:alt}) and dividing by $n$ will not give the correct number of conjugacy classes for each type. Take the word $u_{2} = x^{-1}yx^{-1}y^{-1}xy$ of Type $(3^{0*})$; $u_2$ is not a unique geodesic because $x^{-1}y^{-1}x =_{G} xy^{-1}x^{-1}$. So for example $u'_{2} = x^{-1}yxy^{-1}x^{-1}y=_G u_2$ is also a geodesic, and conjugacy geodesic as well.

We could consider the $6$ cyclic permutations of $u_2$ and pick out one representative, but this would be in the same conjugacy class as $u'_2$, which is not a cyclic permutation of $u_2$. This is why we restrict to counting words of Type $(3^{0*})$ only: we order all $x$ and $x^{-1}$ to have the negative powers first, so $(x^{-1}, x^{-1}, x)$, and take representatives of the $y$ terms $(y,y^{-1},y)$, up to cyclic permutations of the latter. This leads to split cyclic permutations of $u_2$ such as $x^{-1}yx^{-1}y^{-1}xy$, $x^{-1}yx^{-1}yxy^{-1}$, $\dots$, which preserve the order $(x^{-1}, x^{-1}, x)$, are all of Type $(3^{0*})$, and in the same conjugacy class. See also Example \ref{ex:split} and Definition \ref{defn: split odds}.
\end{exmp}
In all instances finding conjugacy growth relies on first finding the standard growth of sets of words with a given structure, as follows. Let $x,y$ be two fixed letters, and let $\Sigma$ and $\Lambda$ be two (finite or infinite) subsets of $\mathbb{Z}_{\neq 0}$. Finding the number of elements of length $n$, the growth series, and the growth rate for a set $\mathcal{A}_{\Sigma, \Lambda}$ as in (\ref{set:alt}), follows standard techniques (see \cite[Thm. I.1, p.27]{Flajolet2009} or \cite{EdjJohn92}). Note that computing growth asymptotics of alternating products of $x$ and $y$ does not depend on the nature of prefixes and suffixes, as explained below. Define 
\begin{equation}\label{set:alt}
\mathcal{A}_{\Sigma, \Lambda}:=\{x^{\sigma_1}y^{\delta_1} x^{\sigma_2} y^{\delta_2}\cdots x^{\sigma_r}y^{\delta_r} \mid 1\leq i \leq r, \sigma_i\in \Sigma, \delta_i\in \Lambda\},
\end{equation}
and let $\mathcal{A^e}_{\Sigma, \Lambda} \supset \mathcal{A}_{\Sigma, \Lambda}$ be the set containing $\mathcal{A}_{\Sigma, \Lambda}$, where the words are allowed to start and end with any letters, that is, are not required to start with $x^{\sigma}$ and end with $ y^{\delta}$:
$$\mathcal{A^e}_{\Sigma, \Lambda}:=\{y^{\sigma_0}x^{\sigma_1}y^{\delta_1} x^{\sigma_2} \cdots y^{\delta_r}x^{\sigma_{r+1}} \mid 1\leq i \leq r, \sigma_i\in \Sigma, \delta_i\in \Lambda, \sigma_0 \in \Lambda \cup \{0\}, \sigma_{r+1}\in \Sigma \cup \{0\}  \}.$$
When the sets $\Sigma$ and $\Lambda$ are given explicitly it is often possible to find the growth rate of the set $\mathcal{A}_{\Sigma, \Lambda}$, following \cite[Lemma 2]{EdjJohn92}, and compare this to the growth of $\mathcal{A^{e}}_{\Sigma, \Lambda}$.
\begin{prop}\label{prop:product growth} (see \cite[Lemma 2]{EdjJohn92})
Let $\Sigma(z)$ be the growth series of the set $\{x^{\sigma}\mid \sigma \in \Sigma\}$ and $\Lambda(z)$ be the growth series of the set $\{y^{\delta}\mid \delta \in \Lambda\}$. The growth series of the set $\mathcal{A}_{\Sigma, \Lambda}$ and the growth series of $\mathcal{A^e}_{\Sigma, \Lambda}$ have the same denominator $1-\Sigma(z)\Lambda(z)$, and the growth rates of the two sets are the same. 
\end{prop}


\subsection{Conjugacy growth in odd dihedral Artin groups}\label{subsec:odd}
Recall that for odd $m \geq 3$ we have $G(m)=\langle x, y \mid x^2=y^m \rangle$, and let $k=\frac{m-1}{2}.$ Any element in $G(m)$ can be represented by a geodesic word $w = x^{a_{1}}y^{b_{1}}\dots y^{b_{\tau}}\Delta^{c}$ with $\Delta=x^2$, and we classify these geodesics into several types, as in Table \ref{table:oddDA}.
We first compute the asymptotics for the numbers of conjugacy classes with unique geodesic representatives, so
of Types (1), (2), ($3^+$), ($3^-$), or ($3^+ \cap 3^-$). These exhibit `standard' behaviour (as in Example \ref{ex:standard}), in that picking out cyclic permutation representatives is sufficient. 

{\small
\begin{table}[!h]
\begin{adjustbox}{width=1\textwidth}
    \begin{tabular}{|c|c|c|c|}
        \hline
        Reference & Type & Conditions on $w = x^{a_{1}}y^{b_{1}}\dots x^{a_{\tau}}y^{b_{\tau}}\Delta^{c}$, where $\tau \in \Z_{>0}$ & Unique/Non-unique \\
        \hline \hline
        \multirow{5}{*}{\cref{prop:geos} }& $(1)$ & $\begin{array}{cc}
        c>0, &\\
        0 \leq a_{i} \leq 1 \; (1 \leq i \leq \tau), & a_{i} \neq 0 \; (2 \leq i \leq \tau),\\
        -(k-1) \leq b_{i} \leq k+1 \; (1 \leq i \leq \tau), & b_{i} \neq 0 \; (1 \leq i \leq \tau - 1).
        \end{array}$ & \multirow{7}{*}{Unique}\\ \cline{2-3}
        & $(2)$ & $\begin{array}{cc}
            c<0, & \\
            -1 \leq a_{i} \leq 0 \; (1 \leq i \leq \tau), & a_{i} \neq 0 \; (2 \leq i \leq \tau), \\
            -(k+1) \leq b_{i} \leq k-1 \; (1 \leq i \leq \tau), & b_{i} \neq 0 \; (1 \leq i \leq \tau - 1).
        \end{array}$ & \\\cline{2-3}
        & $(3^{+})$ & $\begin{array}{cc}
        c=0, & \\
        0 \leq a_{i} \leq 1 \; (1 \leq i \leq \tau), & a_{i} \neq 0 \; (2 \leq i \leq \tau),\\
        -(k-1) \leq b_{i} \leq k+1 \; (1 \leq i \leq \tau), & b_{i} \neq 0 \; (1 \leq i \leq \tau - 1).
    \end{array}$ & \\\cline{2-3}
        & $(3^{-})$ & $\begin{array}{cc}
        c=0, & \\
        -1 \leq a_{i} \leq 0 \; (1 \leq i \leq \tau), & a_{i} \neq 0 \; (2 \leq i \leq \tau), \\
            -(k+1) \leq b_{i} \leq k-1 \; (1 \leq i \leq \tau), & b_{i} \neq 0 \; (1 \leq i \leq \tau - 1).
    \end{array}$ & \\\cline{2-3}
         & $(3^{+}\cap 3^{-})$ & $w = y^{b}$ where $-(k-1) \leq b \leq k-1$.  & \\\cline{1-3}
         $\begin{array}{c}
              \text{\cref{30 cases}} \\
              \text{(\cref{rmk:alternative odd a})} 
         \end{array}$
            & $(3^{0+}U)$ & $\begin{array}{cc}
        c = 0 & \\
        0 \leq a_{i} \leq 1 \; (1 \leq i \leq \tau), & a_{i} \neq 0 \; (2 \leq i \leq \tau),\\ 
        -k \leq b_{i} \leq k \; (1 \leq i \leq \tau), & b_{i} \neq 0 \; (1 \leq i \leq \tau -1), \\ \text{There exists at least one} \; y^{-k} \; \text{term}. & \\
    \end{array}$&  \\\cline{1-3}
    $\begin{array}{c}
              \text{\cref{30 cases}} \\
              \text{(\cref{rmk:alternative odd b})} 
         \end{array}$
         & $(3^{0-}U)$ & $\begin{array}{cc}
        c = 0 & \\
        -1 \leq a_{i} \leq 0 \; (1 \leq i \leq \tau), & a_{i} \neq 0 \; (2 \leq i \leq \tau),\\ 
        -k \leq b_{i} \leq k \; (1 \leq i \leq \tau), & b_{i} \neq 0 \; (1 \leq i \leq \tau -1), \\ \text{There exist at least one} \; y^{k} \; \text{term}. & \\
    \end{array}$ & \\
        \hline
        $\begin{array}{c}
              \text{\cref{30 cases}} \\
              \text{(\cref{rmk:alternative odd a})} 
         \end{array}$
         & $(3^{0+}N)$ & $\begin{array}{cc}
        c = 0 & \\
        0 \leq a_{i} \leq 1 \; (1 \leq i \leq \tau), & a_{i} \neq 0 \; (2 \leq i \leq \tau),\\ 
        -k \leq b_{i} \leq k+1 \; (1 \leq i \leq \tau), & b_{i} \neq 0 \; (1 \leq i \leq \tau -1), \\ 
        \text{There exist both} \; y^{-k} \; \text{and} \; y^{k+1} \; \text{terms}. & \\
    \end{array}$ & \multirow{3}{*}{Non-unique} \\\cline{1-3}
         $\begin{array}{c}
              \text{\cref{30 cases}} \\
              \text{(\cref{rmk:alternative odd b})} 
         \end{array}$
     & $(3^{0-}N)$ & $\begin{array}{cc}
        c = 0 & \\
        -1 \leq a_{i} \leq 0 \; (1 \leq i \leq \tau), & a_{i} \neq 0 \; (2 \leq i \leq \tau),\\ 
        -(k+1) \leq b_{i} \leq k \; (1 \leq i \leq \tau), & b_{i} \neq 0 \; (1 \leq i \leq \tau -1), \\ \text{There exist both} \; y^{k} \; \text{and} \; y^{-(k+1)} \; \text{terms.}  &
    \end{array}$ & \\\cline{1-3}
    $\begin{array}{c}
              \text{\cref{30 cases}} \\
              \text{(\cref{defn:30*})} 
         \end{array}$
     & $(3^{0*})$ & $\begin{array}{cc}
        c = 0 & \\
        -1 \leq a_{i} \leq 1 \; (1 \leq i \leq \tau), & a_{i} \neq 0 \; (2 \leq i \leq \tau), \\ 
        -k \leq b_{i} \leq k \; (1 \leq i \leq \tau), & b_{i} \neq 0 \; (1 \leq i \leq \tau -1), \\
        \text{There exist both} \; x, x^{-1} \;\text{terms.} &
    \end{array}$ &   \\
    \hline
    \end{tabular}
    \end{adjustbox}
    \caption{Geodesics in odd dihedral Artin groups}
    \label{table:oddDA}
\end{table}
}

\begin{prop}\label{prop:oddA}
    The number of all conjugacy classes of length $n \geq 1$ with representatives of Types (1), (2), ($3^+$), ($3^-$), or ($3^+ \cap 3^-$) has asymptotics $\sim \frac{\rho^{-n}}{n}$, where $\rho$ is the real root of minimal absolute value of $p(z)=1-2z^2- \dots - 2z^{k} - z^{k+1} - z^{k+2}.$
\end{prop}
\begin{proof}
  There are only finitely many conjugacy classes with representatives of Type ($3^+ \cap 3^-$) (see Table \ref{table:oddDA}), so these will not play any role when computing the conjugacy growth rates, 
  or when establishing the complexity of the group conjugacy series. 
  
  We first consider Type ($3^+$). Define the set $\mathcal{A}_{(3^+)}:=\{xy^{b_1} x y^{b_1}\cdots xy^{b_{r}} \mid 1 \leq i \leq r, -k+1 \leq b_{i} \leq k+1, b_i\neq 0\}$ consisting of words of Type ($3^+$); these belong to $\overline{\mathcal{A}}$ (recall \cref{A bar}).
Let $\mathcal{S}_{(3^+)}:=\{xy^{b_i} \mid -k+1 \leq b_{i} \leq k+1, b_i\neq 0\}$ be the set of `syllables' that make up the words in $\mathcal{A}_{(3^+)}$. Then $\mathcal{A}_{(3^+)}=\mathcal{S}_{(3^+)}^+$ (the free semigroup on $\mathcal{S}_{(3^+)}$), so the generating function for $\mathcal{A}_{(3^+)}$ is $\mathcal{A}_{(3^+)}(z)= \frac{1}{1- \mathcal{S}_{(3^+)}(z)}$ (see \cite[Thm. I.1, p.27]{Flajolet2009}), where 
\begin{equation}\label{Sformula}
\mathcal{S}_{(3^+)}(z)=2z^2+ \dots + 2z^{k}+ z^{k+1}+ z^{k+2}
\end{equation}
 is the generating function of $\mathcal{S}_{(3^+)}$, so 
\begin{equation}\label{nonZseriesOdd}
\mathcal{A}_{(3^+)}(z)=\frac{1}{1-2z^2- \dots - 2z^{k} - z^{k+1} - z^{k+2}}.
\end{equation}
Its denominator, that is, the polynomial $p(z)=1-2z^2- \dots - 2z^{k} - z^{k+1} - z^{k+2}$, satisfies $p(0)=1$ and $p\left(\frac{1}{\sqrt{2}}\right)<0$, so it has a root $\rho \in \left(0,\frac{1}{\sqrt{2}}\right)$. Moreover, $p'(\alpha)<0$ for $0<\alpha<\frac{1}{\sqrt{2}}$, so $\rho$ is a simple root. Also, $p(z)=(1-2z^2)+(-2z^3-2z^4)+ \dots + (-z^{k} - z^{k+1})+( -z^{k}-z^{k+2})$ if $k$ is odd, and $p(z)=(1-2z^2)+(-2z^3-2z^4)+ \dots + ( -z^{k+1}-z^{k+2})$ if $k$ is even, and since each parenthesis is $>0$, $p$ has no root in $(-\frac{1}{\sqrt{2}},0]$. Thus the growth rate of the set $\mathcal{A}_{(3^+)}$ is $\frac{1}{\rho}>\sqrt{2}$, which implies (see \cite[Section IV.5]{Flajolet2009}) that the number of words of length $n$ in $\mathcal{A}_{(3^+)}$ is asymptotically $c(k) \rho^{-n}$, where $c(k)$ is a constant depending on $k$ (and therefore on $m$). 

In order to find the growth of the conjugacy classes of Type ($3^+$), we count the words in $\mathcal{A}_{(3^+)}$ up to the cyclic permutation of the syllables in $\mathcal{S}_{(3^+)}$.  For each word in $\mathcal{A}_{(3^+)}$ with $r$ syllables there are $r$ possible distinct cyclic permutations unless that word is a non-trivial power. Given that the number of powers is negligible compared to the total number of words (to see this, suppose that an alphabet $X$ consists of $x>1$ letters; while the number of words of length $n$ over $X$ is $x^n$, the number of proper powers of length $n$ is $\sum_{d|n, d\neq n}x^d < \sum_{i=1}^{n/2}x^i <x^{\frac{n}{2}+1}$, and $\frac{x^{\frac{n}{2}+1}}{x^n} \rightarrow 0$ as $n \rightarrow \infty$), for fixed $n$ and $r$ the number of cyclic representatives of words in $\mathcal{A}_{(3^+)}$ is $\sim c(k)\frac{\rho^{-n}}{r}$. In a word of length $n$ and $r$ `syllables' $xy^{b_i}$, since each syllables has length at least $2$ and at most $k+2$, we get  $\frac{n}{k+2} \leq r \leq \frac{n}{2}$,  
and so the number $a_n$ of cyclic representatives satisfies 
\begin{equation}\label{inequality:odd}
c(k)\frac{\rho^{-n}}{n} \leq a_n\leq c(k)\frac{(k+2)\rho^{-n}}{n}.
\end{equation}
The counting for Type ($3^-$) is exactly the same as there is a length-preserving bijection between the sets ($3^+$) and ($3^-$). Furthermore, there is a length-preserving bijection between Types ($1$) and ($2$), so it suffices to consider Type ($1$). For Type ($1$), we count words of a given length in the direct (or cartesian) product of the sets $\mathcal{A}_{(3^+)}$ and $\{\Delta^c=x^{2c} \mid c\geq 1\}$; the growth series of the latter is $\frac{z^2}{1-z^2}$, and thus the growth series of geodesics of Type (1) is $\mathcal{A}_{(3^+)}(z)\frac{z^2}{1-z^2}$ (see \cite[Thm. I.1, p.27]{Flajolet2009}), with the same minimal value real root of the denominator as that of $\mathcal{A}_{(3^+)}(z)$, and therefore the same asymptotics as $\mathcal{A}_{(3^+)}$. The cyclic representatives are counted as for $\mathcal{A}_{(3^+)}(z)$.

Thus the conjugacy classes of length $n \geq 1$ of each of the Types (1), (2), ($3^+$) and ($3^-$) have asymptotics $\sim \frac{\rho^{-n}}{n}$, and their union will therefore have the same asymptotics.
\end{proof}
We still need to consider several kinds of conjugacy classes, those with representatives of Types $(3^{0+}U)$, $(3^{0+}N)$, $(3^{0-}U)$, $(3^{0-}N)$ and $(3^{0*})$. 
It will be convenient to consider unique and non-unique geodesic representatives together in the context of Type $(3^{0+})$, which are all the words of Types $(3^{0+}U)$ or $(3^{0+}N)$ (see Proposition \ref{30 cases}) (similarly, 
Type $(3^{0-})$ geodesics are the words of Types $(3^{0-}U)$ or $(3^{0-}N)$) and we proceed with Types $(3^{0+})$ and $(3^{0-})$.

\begin{prop}\label{prop:oddB}
    The number of all conjugacy classes of length $n \geq 1$ with representatives of Types $(3^{0+})$ and $(3^{0-})$ has  asymptotics $\sim \frac{\mu^{-n}}{n}$, where $\mu$ is the real root of minimal absolute value of $q(z)=1-2z^2- \dots - 2z^{k+1}$.
\end{prop}

\begin{proof}
    Words in $(3^{0+})$ have the form $xy^{b_1}\dots xy^{b_r}$ with $-k \leq b_{i} \leq k+1$ for $1 \leq i \leq r$, where all the $y^{-k}$ terms must appear before the $y^{k+1}$ terms. That is, as in Definition \ref{defn: split odds}, a word $w \in (3^{0+})$ will have the form 
    $$w =A_{1}y^{-k}\dots A_{\tau_{1}}y^{-k}A_{\tau_{1} + 1}y^{k+1}A_{\tau_{1}+2}y^{k+1}\dots y^{k+1}A_{\tau_{1} + \tau_{2} + 1},$$
    where $\tau_{2} \geq 0$; the $A_i$ are words on $\{x, y^j \mid -(k-1) \leq j \leq k\}$, starting and ending with $x$.
Using the relation $y^{k+1}=_{G} y^{-k}\Delta=_{G} y^{-k}x^2$ and the fact that $x^{2}$ is central, we can replace all occurrences of $y^{k+1}$ in $w$ by $y^{-k}$s to get \begin{equation}\label{eqn:30+}
    w=_G w'=\underbrace{\left(A_{1}y^{-k}\dots A_{\tau_{1}}y^{-k}A_{\tau_{1} + 1}y^{-k}A_{\tau_{1}+2}y^{-k}\dots y^{-k}A_{\tau_{1} + \tau_{2} + 1}\right)}_u x^{2\tau_2},
\end{equation}
    and write $w'=ux^{2\tau_2}$, where $u$ is the word inside the parentheses; note that $$|w|=|w'|-\tau_2=|u|+\tau_2.$$
    By Proposition \ref{prop:conj geo B} (see also Definition \ref{defn: split odds}) finding conjugacy representatives for Type $(3^{0+})$ is equivalent to picking representatives, up to cyclic permutations, for the tuples $(A_1,\dots, A_{\tau_{1} + \tau_{2} + 1})$, where splitting some $A_i$ into two subwords is allowed (see Definition \ref{defn: split odds} (ii) Case 2). Counting words $w$ of length $n$ and Type $(3^{0+})$, up to the cyclic permutations of the $A_i$ blocks, is difficult, since $y^{-k}$ and $y^{k+1}$ have different lengths and number of occurrences.

    We will instead consider the word $u$ associated to each $w$ of Type $(3^{0+})$, as in (\ref{eqn:30+}). Since $|u|=|w|-\tau_2$, this is not a length-preserving correspondence, however $u$ and $w$ have the same $A_i$-tuple structure and for every $u$ of length $\leq n$, there is a unique $w$ of length $n$ corresponding to $u$. Then counting words of length $n$ and Type $(3^{0+})$, up to cyclic permutation, is equivalent to counting words, up to cyclic permutation, with structure
\begin{equation}\label{type_u}
    A_{1}y^{-k}\dots A_{\tau_{1}}y^{-k}A_{\tau_{1} + 1}y^{-k}A_{\tau_{1}+2}y^{-k}\dots y^{-k}A_{\tau_{1} + \tau_{2} + 1}
    \end{equation}
of length $n-\tau_2$, for $0\leq \tau_2 \leq \frac{n}{k+1}$.
Words as in (\ref{type_u}) can be collected (after a cyclic permutation) to form the set $\mathcal{A}_{(3^{0+})}=\{xy^{b_1} x y^{b_1}\cdots xy^{b_{r}} \mid 1 \leq i \leq r, -k \leq b_{i} \leq k, b_i\neq 0\}$. This set can be treated as in the proof of Proposition \ref{prop:oddA}: start with the set $\{xy^{b_i} \mid -k \leq b_{i} \leq k, b_i\neq 0\}$, with generating function $2z^2+ \dots + 2z^{k}+ 2z^{k+1}$ to obtain
\begin{equation}\label{case3_0}
\mathcal{A}_{(3^{0+})}(z)=\frac{1}{1-2z^2- \dots - 2z^{k+1}}=\frac{1}{q(z)}
\end{equation}
as the generating function for the set $\mathcal{A}_{(3^{0+})}$. The denominator satisfies $q(0)=1$ and $q\left(\frac{1}{\sqrt{2}}\right)<0$, so it has a root $\mu \in \left(0,\frac{1}{\sqrt{2}}\right)$. Moreover, $q'(\alpha)<0$ for $0<\alpha$, so $\mu$ is a simple root. Also, $q(z)=(1-2z^2)+(-2z^3-2z^4)+ \dots + (-2z^{k} - 2z^{k+1})$ if $k$ is odd, and $q(z)=(1-2z^2)+(-2z^3-2z^4)+ \dots + ( -2z^{k+1})$ if $k$ is even, and since each parenthesis is positive, $q$ has no root in $(-\frac{1}{\sqrt{2}},0]$. Thus the denominator $q(z)$ gives asymptotics of $\frac{\mu^{-n}}{n}$ for the representatives, up to cyclic permutation, of words of length $n$ in $\mathcal{A}_{(3^{0+})}$. 

However, words of Type $(3^{0+})$ of length $n$ correspond to words of a range of lengths in $\mathcal{A}_{(3^{0+})}$ (as explained before, we consider the $u$ counterpart for each $w$), so `spherical' growth for Type $(3^{0+})$ is equivalent to `cumulative' or `ball' growth for words of Type (\ref{type_u}); the asymptotics for conjugacy representatives in $(3^{0+})$ is 
\begin{equation}\label{eq:Coornaert}
\sum_{i=0}^{\frac{n}{k+1}} \frac{\mu^{-(n-i)}}{n-i} \sim \frac{\mu^{-n}}{n},
\end{equation}
where the $\sim$ holds by \cite[Lemma 3.2]{coornaert_asymptotic_2005}, which proves the proposition.
\end{proof}

We finally consider the last type of conjugacy classes for $G(2k+1)$, that is, the ones with representatives of Type $(3^{0*})$; here conjugacy behaviour is `non-standard', as in Example \ref{ex:non-standard}.
As in Proposition \ref{prop:oddB} and its proof, $\mu$ is the real root of minimal absolute value of $q(z)=1-2z^2- \dots - 2z^{k+1}$ and $\mathcal{A}_{(3^{0+})}(z)=\frac{1}{q(z)}$.

\begin{prop}\label{prop:oddBC}
    The growth series for the conjugacy classes with representatives of Type $(3^{0*})$ is given by $\mathcal{A}_{(3^{0+})}(z)+N(z),$ where $\mathcal{A}_{(3^{0+})}(z)$ is a rational series with coefficients having asymptotics $\sim\mu^{-n}$, and $N(z)$ is a series with coefficients having growth rate $< \mu^{-1}$. 
\end{prop}

\begin{proof}
Recall that words of Type ($3^{0*}$) have the form $x^{-1}y^{b_1}x^{-1}\dots y^{b_{\tau_{1}}}xy^{b_{\tau_{1} + 1}}x\dots xy^{b_{\tau_{1} + \tau_{2}}}$ and that counting conjugacy classes of Type $(3^{0*})$ is equivalent to picking representatives, up to cyclic permutations, for the tuples $(y^{b_1}, \dots, y^{b_{\tau_{1}}},y^{b_{\tau_{1} + 1}},\dots, y^{b_{\tau_{1} + \tau_{2}}})$ while keeping the $X$-structure fixed (see Definition \ref{defn:30*} and Proposition \ref{prop:conj geo B}).

Denote by $L(3^{0*})$ the words of Type $(3^{0*})$, and recall the set $\mathcal{A}_{(3^{0+})}$  from the previous proof: $$\mathcal{A}_{(3^{0+})}=\{xy^{b_1} \cdots xy^{b_r}\mid -k \leq b_{i} \leq k, b_i\neq 0, 1\leq i \leq r\}.$$
Let $\phi:L(3^{0*}) \rightarrow \mathcal{A}_{(3^{0+})}$ be the map extending $\phi(x)=\phi(x^{-1})=x, \phi(y)=y$, 
and associate to every word $w = x^{-1}y^{b_1}x^{-1}\dots y^{b_{\tau_{1}}}xy^{b_{\tau_{1} + 1}}\dots xy^{b_{\tau_{1} + \tau_{2}}}$ of Type ($3^{0*}$) the word $u=\phi(w)$, 
so $u=xy^{b_{1}}\dots xy^{b_{r}}$, where $r=\tau_1+\tau_2$. Since the length of a syllable $xy^{b_i}$ is $\geq 2$, we have $1\leq r\leq \frac{|u|}{2}$.
    Then $|w|=|u|$, $u$ and $w$ have the same $y^{b_i}$-tuple structure, and the number of preimages of $u$ in $L(3^{0*})$ satisfies $\#\{\phi^{-1}(u)\}=r$; that is, there are $r$ different $w$'s which project to the same $u$.
     So we count representatives in $\mathcal{A}_{(3^{0+})}$ up to cyclic permutation, and then multiply the number of such representatives (of length $n$) by the appropriate number of preimages in $L(3^{0*})$ to get the result about conjugacy classes of Type ($3^{0*}$).
    
    Suppose there are $l_n$ words of length $n$, and $l_{n,r}$ words of length $n$ and $r$ syllables ($1\leq r \leq n/2$) in $\mathcal{A}_{(3^{0+})}$, respectively. The number of representatives of length $n$, up to the cyclic permutation of the syllables, in $\mathcal{A}_{(3^{0+})}$ is $\geq
    \sum_{r=1}^{n/2}\frac{l_{n,r}}{r}$ ($\ast$); moreover, since each representative with $r$ syllables has $r$ different preimages (under $\phi$) of Type ($3^{0*}$), the number of conjugacy representatives of length $n$ and Type $3^{0*}$ is $\geq \sum_{r=1}^{n/2}\frac{l_{n,r}}{r} r=\sum_{r=1}^{n/2}l_{n,r}=l_n$.
 The sum ($\ast$) is a lower bound for the number of cyclic representatives in $\mathcal{A}_{(3^{0+})}$ because non-trivial powers have fewer distinct cyclic permutations than non-powers, and the division should be done by a proper divisor of $r$, and not $r$ itself to get the correct number of representatives.

 The number $c_{3^{0*}}(n)$ of conjugacy classes of length $n$ and Type $(3^{0*})$ is thus equal to $l_n$, plus $e_n$, which is (a fraction of) the number of the non-trivial powers in $\mathcal{A}_{(3^{0+})}$. Thus the conjugacy growth series for Type $(3^{0*})$ is the sum of the generating function for $l_n$ and the generating function for $e_n$. The first is $\mathcal{A}_{(3^{0+})}(z)$, and the second is the growth series of a set which is negligible in $\mathcal{A}_{(3^{0+})}$: that is, the growth rate of $e_n$ is strictly smaller than that of $\mathcal{A}_{(3^{0+})}$, which by Proposition \ref{prop:oddB} is $\mu^{-1}$. The result follows. 
\end{proof}

\begin{theorem}\label{thm:odd dihe trans}
    The conjugacy growth series of an odd dihedral Artin group $G(2k+1)$ is transcendental with respect to the free product generating set $\{x,y\}$.
\end{theorem}

\begin{proof}
We now sum up and compare the asymptotics for all conjugacy classes. 
Let $\rho$ and $\mu$ be as in Propositions \ref{prop:oddA} and \ref{prop:oddB}. Since $\rho$ and $\mu$ are roots of $p(z)$ and $q(z)$, respectively,
$ 0=1-2\mu^{2}\left(1+\dots + \mu^{k-1}\right) = 1-2\rho^{2}-\dots -\left(\rho^{k+1}+\rho^{k+2}\right),
$
and $\rho^{k+2} < \rho^{k+1}$ since $\rho <1$, so
\[ 1-2\mu^{2}\left(1+\dots + \mu^{k-1}\right) > 1-2\rho^{2}\left(1+\dots + \rho^{k-1}\right).
\]
This implies that $\mu < \rho$, and so $\mu^{-1}>\rho^{-1}$. By Proposition \ref{prop:oddA}, the conjugacy asymptotics for Types (1), (2), ($3^+$), ($3^-$), or ($3^+ \cap 3^-$), which is $\sim\frac{\rho^{-n}}{n}$, is therefore dominated by that of 
Types $(3^{0+})$ and $(3^{0-})$, which is $\sim\frac{\mu^{-n}}{n}$. 

By Proposition \ref{prop:oddBC}, the growth series of Type $(3^{0*})$ splits into a rational power series, which will not influence the global conjugacy series complexity, and the series of a set whose growth is again dominated by Types $(3^{0+})$ and $(3^{0-})$. 

    Finally, by \cite[Thm. D]{Flajolet1987} the generating function for any sequence with asymptotics of the form (\ref{inequality:odd}), in our case $\frac{\mu^{-n}}{n}$ (up to multiplicative constants), is transcendental.
\end{proof}
\begin{rmk}\label{rmk:rates}
    We note that by \cite{EdjJohn92} the growth rate of $G(2k+1)$ is $\frac{1}{\mu}$, and a careful analysis of the results in this section shows that the conjugacy growth rate is $\frac{1}{\mu}$ as well, since $c(n)\sim \mu^{-n}$.
\end{rmk}

\subsection{Conjugacy growth for even dihedral Artin groups}\label{sec:even}
We now study the $G(2p)=\mathrm{BS}(p,p)$ groups, and treat $p=2$, $p>2$ even, and $p$ odd separately.

\subsubsection{Conjugacy growth for $G(4)=\mathrm{BS}(2,2)$}
We now consider the conjugacy growth of the $G(4)=\mathrm{BS}(2,2)$ group, and use Table \ref{table:BS(2,2)}.  
{\small
\begin{table}[!h]
    \begin{adjustbox}{width=1\textwidth}
    \begin{tabular}{|c|c|c|c|}  
        \hline
         Reference & Type & Form and conditions  & Unique/Non-unique \\
        \hline \hline
        \multirow{3}{*}{\cref{subsec: p=2}} & (1) & $\begin{array}{c}
            w = (x^{\mu_{1}}y^{\beta_{1}}\dots x^{\mu_{r}}y^{\beta_{r}}\Delta^{c})^{\pm 1}  \\
            c \geq 0, r \geq 1,  \\
            \mu_{1} \in \{0, 1\}, \; \mu_{i} = 1 \;(2 \leq i \leq r), \\
            \beta_{i} \in \Z \; (1 \leq i \leq r), \beta_{i} \neq 0 \; (1 \leq i \leq r-1).
        \end{array}$ & Unique \\
        \cline{2-4}
        & (2) & $\begin{array}{c}
             w = y^{\beta_1}xy^{\beta_2}x\dots y^{\beta_{\tau_{1}}}xy^{\beta_{\tau_{1}+1}}x^{-1}y^{\beta_{\tau_{1} + 2}}x^{-1}\dots x^{-1}y^{\beta_{\tau_{1} + \tau_{2} + 1}}, \\
              \beta_{i} \in \Z_{\neq 0}, \tau_{1}, \tau_{2} \geq 1
        \end{array}$ & Non-unique \\
        \cline{2-4}
        & (3) & $w = y^{\beta}, \; \beta \in \Z_{\neq 0}$ & Unique \\
        \hline
    \end{tabular}
    \end{adjustbox}
    \caption{Subcase of even dihedral Artin groups: $\mathrm{BS}(2,2)=\langle x,y \mid y^{-1}x^2y=x^2 \rangle.$ ($p=2$)}
    \label{table:BS(2,2)}
\end{table}
}

\begin{theorem}\label{thm:BS 22}
    The conjugacy growth series of $\mathrm{BS}(2,2)$ is transcendental, with respect to the free product generating set. The conjugacy growth function satisfies $c(n) \sim 2^n$.
\end{theorem}

\begin{proof} 
First note that conjugacy classes of Type (3) have rational growth, so we only need to focus on Types (1) and (2).
The conjugacy growth for Type (1) elements can be seen as (a scalar multiple of) the conjugacy growth of the direct product of $G_1=\Z/2\Z \ast \Z =\langle x,y\rangle$ and an infinite cyclic group $G_2$ generated by $\Delta$. Since $G_1$ is (non-virtually cyclic) hyperbolic and $G_2$ is cyclic, the conjugacy growth series of Type 1 elements is the product of a transcendental (by \cite{AC2017}) and a rational series, so is transcendental. Moreover, by standard arguments, the asymptotics of Type (1) conjugacy classes is $\frac{\alpha^{-n}}{n}$, where $\alpha^{-1}$ is the standard growth rate of $\Z/2\Z \ast \Z$ (by \cite{AC2017}). A quick counting argument shows that in fact $\alpha^{-1}=2$.

Finding the asymptotics of conjugacy classes of Type (2) is similar to the method used in the proof of Proposition \ref{prop:oddBC}. The only difference is that instead of using the set $\{xy^{b_1} \cdots xy^{b_r}\mid -k \leq b_{i} \leq k, b_i\neq 0, 1\leq i \leq r\}$ to project onto and count representatives, up to cyclic permutation of the syllables, we take the exponents $b_i$ to be unbounded, that is $b_i \in \Z$. Thus the role of $\mathcal{A}_{(3^{0+})}$ is now taken by the hyperbolic (i.e. infinite order) elements in $\Z/2\Z \ast \Z$, which have rational growth series that we denote by $S(z)$ and growth rate of $\alpha^{-1}$. 

Similar to Proposition \ref{prop:oddBC}, the conjugacy growth of Type (2) is equal to (a scalar multiple of) the standard growth of hyperbolic elements of $\Z \ast \Z/2\Z$, plus the growth of a subset $N$ of elements which are non-trivial powers in $\Z \ast \Z/2\Z$. Thus the conjugacy series for Type (2) is the sum of the rational series $S(z)$ and that of a set with growth rate $<2$. If we exclude the words counted by $S(z)$ (because they give rationality), the remaining conjugacy classes of Type (1) and (2) have asymptotics $\sim\frac{2^{-n}}{n}$, as Type (1) dominates the asymptotics of set $N$. As in the previous section, this gives an overall transcendental conjugacy growth series.    

It can be seen from the above discussion that the dominating asymptotics comes from $S(z)$, which has a growth rate of $2$, so the conjugacy growth asymptotics for $BS(2,2)$ is $\sim 2^n$. The standard growth rate of $BS(2,2)$ is $2$ by \cite[p.290]{EdjJohn92}, so the conjugacy growth rate is the (standard, equal to the conjugacy) growth rate of $\Z \ast \Z/2\Z$, which is $\alpha^{-1}=2$ as well.
\end{proof}

\subsubsection{Conjugacy for the even dihedral Artin groups $G(2p)=\mathrm{BS}(p,p)$ with $p=2k$}
In this section we assume $k\geq 2$ and use the classification of geodesics in Table \ref{tab:BS(p,p)}.
\begin{table}[h]
    \begin{adjustbox}{width=1\textwidth}
    \begin{tabular}{|c|c|c|c|}
        \hline
        Reference & Type & Form and conditions  & Uniqueness \\
        \hline \hline
        \cref{prop:geos 2k} & (1) & $\begin{array}{c}
            w = (x^{\mu_{1}}y^{\ep_{1}}\dots x^{\mu_{n}}y^{\ep_{n}}\Delta^{c})^{\pm 1}  \\
            c, n \geq 1,  \\
            \mu_{1} \in [0, 2k-1] \cap \Z, \; \mu_{i} \in [1, 2k-1] \cap \Z \;(2 \leq i \leq n), \\
            \ep_{i} \in \Z \; (1 \leq i \leq n), \ep_{i} \neq 0 \; (1 \leq i \leq n-1).
        \end{array}$
          & \multirow{2}{*}{Unique}\\\cline{1-3}
        \multirow{2}{*}{Prop. \ref{prop:geos 2k}, Def. \ref{rmk:alternative}}& $(2a)$ & $\begin{array}{c}
            w = A_{1}x^{j}A_{2}x^{j}\dots A_{\tau_{1}}x^{j}A_{\tau_{1}+1}, \\
            k\leq j \leq 2k-1,\\
            A_{s} \; \text{reduced words over} \; x^{\alpha},y^{\beta}, \\
             \alpha \in [-2k+j+1, j-1] \cap \mathbb{Z}_{\neq 0}, \beta \in \Z_{\neq 0}\\ 
        \end{array}$  & \\
        \cline{2-4}
        & $(2b)$ & $\begin{array}{c}
            w = A_{1}x^{j}A_{2}x^{j}\dots A_{\tau_{1}}x^{j}A_{\tau_{1}+1}x^{-(2k-j)}A_{\tau_{1} + 2}x^{-(2k-j)}\dots x^{-(2k-j)}A_{\tau_{1} + \tau_{2} + 1},  \\
            A_{s} \; \text{as in} \; (2a) \\
            \tau_{2} \geq 1
            \end{array}$ & Non-unique \\
            \hline
        \cref{prop:geos 2k} & $(3)$ & $\begin{array}{c}
        w = x^{\alpha}y^{\ep_{1}}x^{\nu_{1}}\dots x^{\nu_{n}} \\
        n \geq 0, \ep_{i} \in \Z_{\neq 0} \; (1 \leq i \leq n)   \\
        \alpha, \nu_{n} \in [-(k-1), k-1] \cap \Z, \; \nu_{i} \in [-(k-1), k-1] \cap \Z_{\neq 0} \; (1 \leq i \leq n-1)  \\
        \end{array}$  & Unique \\
        \hline

    \end{tabular}
    \end{adjustbox}
    \caption{Even dihedral Artin groups: $\mathrm{BS}(p,p)$ with $p=2k>2$}
    \label{tab:BS(p,p)}
\end{table}
\\
First we establish which of the different types of geodesics grows fastest. This will turn out to be those of Type (2),  that is, Type (2a) and Type (2b) taken together.
\begin{prop}\label{prop: 3 bigger than 1}
    The growth rate of Type (3) words is larger than that of Type (1) words.
\end{prop}

\begin{proof}
    First consider the growth rate of Type (1), noting that we can ignore the Garside element $\Delta^{c}$, using the same argument as \cref{prop:oddA}. Define the set $\mathcal{S}_{(1)} := \{x^{\alpha}y^{\beta} \; | \; \alpha \in [1,2k-1], \beta \in \Z_{\neq 0} \}$. Then the generating function for all words of Type (1) which belong to $\overline{\mathcal{C}}$ (recall \cref{defn:C bar}) is $\mathcal{C}_{(1)}(z) = \frac{1}{1-\mathcal{S}_{(1)}(z)}$, where
    \[ \mathcal{S}_{(1)}(z) = \left(z+z^{2}+\dots + z^{2k-1}\right)\cdot 2z\frac{1}{1-z} = \frac{2z^{2}\left(1+z+\dots + z^{2k-2}\right)}{1-z}
    \]
    The denominator of $\mathcal{C}_{(1)}(z)$ is the polynomial $p_{1}(z) = 1-z-2z^{2}\left(1+z+\dots + z^{2k-2}\right)$, and one can show using a similar argument as in \cref{prop:oddA} that the growth of Type (1) has asymptotics $\sim \frac{\rho_{1}^{-n}}{n}$, where $\rho_{1} \in \left(0, \frac{1}{\sqrt{2}}\right)$ is the smallest simple positive root of $p_{1}(z).$ 
Now we consider the growth rate of Type (3). Define the set $\mathcal{S}_{(3)}(z) := \{x^{\alpha}y^{\beta} \; | \; \alpha \in [-(k-1), k-1], \alpha \neq 0, \beta \in \Z_{\neq 0}\}$. Then the generating function of all words of Type (3) in $\overline{\mathcal{C}}$ is $\mathcal{C}_{(3)}(z) = \frac{1}{1-\mathcal{S}_{(3)}(z)}$, where 
    \[ \mathcal{S}_{(3)}(z) = \left(2z+2z^{2}+\dots + 2z^{k-1}\right)\cdot 2z\frac{1}{1-z} = \frac{4z^{2}\left(1+z+\dots + z^{k-2}\right)}{1-z}.
    \]
    The denominator of $\mathcal{C}_{(3)}(z)$ is the polynomial $p_{3}(z) = 1-z-4z^{2}\left(1+z+\dots + z^{k-2}\right)$, and similar to previous arguments (as in Section \ref{subsec:odd}) the growth of Type (3) has asymptotics $\sim \frac{\rho_{3}^{-n}}{n}$, where $\rho_{3} \in \left(0, \frac{1}{2}\right)$ is the simple root of smallest absolute value.

We claim that $\rho_{3} < \rho_{1}$. Since $p_{1}(\rho_1)=p_{3}(\rho_3)=0$, we get
    \[0= 1-\rho_{3}-4\rho_{3}^{2}\left(1+\rho_{3}+\dots + \rho_{3}^{k-2}\right) = 1-\rho_{1}-2\rho_{1}^{2}\left(1+\rho_{1}+\dots \rho_{1}^{2k-2}\right),
    \]
    and $\rho_{1} < 1$ implies $\rho_{1} + \rho_{1}^{2k-2} < 2\rho_{1}$, $\rho_{1}^{2} + \rho_{1}^{2k-3} < 2\rho^{2}_{1}$ and so forth until $\rho_{1}^{k-1}+\rho_{1}^{k} < 2\rho_{1}^{k-1}$, so 
    \[  1-\rho_{3}-4\rho_{3}^{2}\left(1+\rho_{3}+\dots + \rho_{3}^{k-2}\right) > 1-\rho_{1} -2\rho_{1}^{2}\left(1+2\rho_{1}+\dots + 2\rho_{1}^{k-1}\right).
    \]
    Now $\rho_{1} < \frac{1}{2}$, so $2\rho_{1}^{k-1} < \frac{1}{2^{k-2}} \leq 1$ since $k \geq 2$. Therefore $1+2\rho_{1}^{k-1} < 2$, and so 
    \[ 1-\rho_{3}-4\rho_{3}^{2}\left(1+\rho_{3}+\dots + \rho_{3}^{k-2}\right)
    > 1-\rho_{1} -4\rho_{1}^{2}\left(1+2\rho_{1}+\dots + 2\rho_{1}^{k-2}\right).
    \]
    This implies that $\rho_{3} < \rho_{1}$, and so $\frac{1}{\rho_{3}}>\frac{1}{\rho_{1}}$. Therefore the growth rate of Type (3) is greater than that of Type (1). 
\end{proof}

\begin{prop}\label{prop:j=k}
    The growth rate of words of Type (2) is greatest when $j=k$.
\end{prop}
\begin{proof}
Recall that Type (2) is Type (2a) and (2b) together, and note that for a fixed $j$ the $x$ exponents in (2a) have range $ \in [-2k+j+1, j]$, and the same holds for the (2b) prefixes (i.e. $A_{1}x^{j}A_{2}x^{j}\dots A_{\tau_{1}}x^{j}A_{\tau_{1}+1}$) identical to words in (2a), while in the suffixes $x^{-(2k-j)}A_{\tau_{1} + 2}x^{-(2k-j)}\dots x^{-(2k-j)}A_{\tau_{1} + \tau_{2} + 1}$ without $x^j$ the range is $\in [-2k+j, j-1]$. The growth of Type (2b) words can be taken to be the maximum between the growth of the prefixes (identical to (2a)) and the suffixes, as in a direct product.

We first consider the prefixes of words of Type (2), that is, the words of form (2a).
For fixed $j$ with $k\leq j \leq 2k-1$
    define $\mathcal{S}_{(2,j)} := \{x^{\alpha}y^{\beta} \; | \; \alpha \in [-2k+j+1, j], \alpha \neq 0, \beta \in \Z_{\neq 0} \}$. Then the generating function for words of Type (2a) with fixed $j$ is $\mathcal{D}_j(z) = \frac{1}{1-\mathcal{S}_{(2,j)}(z)}$, where
    \[ \mathcal{S}_{(2,j)}(z) = \left(2z+2z^{2}+\dots + 2z^{2k-j-1}+z^{2k-j}+\dots + z^{j}\right)\cdot 2z\frac{1}{1-z}.
    \]
    The denominator of $\mathcal{D}_j(z)$ is $p_j(z) = 1-z-2z\left(2z+\dots + 2z^{2k-j-1}+z^{2k-j}+\dots + z^{j}\right)$. Let $\mu_{j}$ denote the smallest positive root of $p_j(z)$, and note $\mu_{j} < \frac{1}{2}<1$ since $p(0) = 1$ and $p\left(\frac{1}{2}\right) < 0$. We claim that $\mu_{j} \geq \mu_k$ for all $k \leq j \leq 2k-1$. By assumption we have
    \begin{align*}
       0 & =1 - \mu_{j} - 2\mu_{j}\left(2\mu_{j}+\dots + 2\mu_{j}^{2k-j-1}+\mu_{j}^{2k-j} + \dots + \mu_{j}^{j}\right) =p_j(\mu_j) \\
        &= 1-\mu_k-2\mu_k\left(2\mu_k+\dots + 2\mu_{k}^{k-1}+ \mu_{k}^k\right)=p_k(\mu_k).
    \end{align*}
    so $\mu_{j}+2\mu_{j}(2\mu_{j}+\dots + 2\mu_{j}^{2k-j-1}+\mu_{j}^{2k-j} + \dots + \mu_{j}^{j}) =\mu_k+2\mu_k(2\mu_k+\dots + 2\mu_{k}^{k-1}+ \mu_{k}^k)=R(\mu_k)$; furthermore, $\mu_{j}+2\mu_{j}(2\mu_{j}+\dots + 2\mu_{j}^{2k-j-1}+\mu_{j}^{2k-j} + \dots + \mu_{j}^{j}) \leq \mu_j+2\mu_j(2\mu_j+\dots + 2\mu_{j}^{k-1}+ \mu_{j}^k)=R(\mu_j)$ because $\mu_j<1$ and $j\geq k$. Then $R(\mu_j)\geq R(\mu_k)$ gives $\mu_{j} \geq \mu_k$ and so $\frac{1}{\mu_{k}} \geq \frac{1}{\mu_j}$, with equality only when $k=j$.

    A similar argument shows that the growth rate $\nu_j$ of suffixes in Type (2b) is maximal when $j=k$ or $j=k+1$, that is, $\nu_k=\nu_{k+1}$: in this case the range of the $x$-exponents is $[-k, k-1]$ or $[-(k-1),k]$, the second interval being symmetric to the first and giving a length-preserving bijection between the suffixes with $j=k$ and those with $j=k+1$.
\comm{
For a fixed $j$ with $k\leq j \leq 2k-1$
    define $\mathcal{S}_{(2,j)}(z) := \{x^{\alpha}y^{\beta} \; | \; \alpha \in [-2k+j, j], \alpha \neq 0, \beta \in \Z_{\neq 0} \}$. Then the generating function for words of Type (2) with fixed $j$ is $\mathcal{D}_j(z) = \frac{1}{1-\mathcal{S}_{(2,j)}(z)}$, where
    \[ \mathcal{S}_{(2,j)}(z) = \left(2z+2z^{2}+\dots + 2z^{2k-j}+z^{2k-j+1}+\dots + z^{j}\right)\cdot 2z\frac{1}{1-z}
    \]
    The denominator is therefore $p_j(z) = 1-z-2z\left(2z+\dots + 2z^{2k-j}+z^{2k-j+1}+\dots + z^{j}\right)$. Let $\mu_{j}$ denote the smallest positive root of $p_j(z)$, and note $\mu_{j} < \frac{1}{2}<1$ since $p(0) = 1$ and $p\left(\frac{1}{2}\right) < 0$. We claim that $\mu_{j} < \mu_{j+1}$ for all $k \leq j \leq 2k-1$. By assumption we have
    \begin{align*}
        &1 - \mu_{j} - 2\mu_{j}\left(2\mu_{j}+\dots + 2\mu_{j}^{2k-j}+\mu_{j}^{2k-j+1} + \dots + \mu_{j}^{j}\right) \\
        &= 1-\mu_{j+1}-2\mu_{j+1}\left(2\mu_{j+1}+\dots + 2\mu_{j+1}^{2k-j-1}+\mu_{j+1}^{2k-j}+\dots + \mu_{j+1}^{j+1}\right).
    \end{align*}
    Since $j+1 > 2k-j$ and $\mu_{j+1} < 1$, then $\mu_{j+1}^{j+1} < \mu_{j+1}^{2k-j}$, and so 
    \begin{align*}
        &1 - \mu_{j} - 2\mu_{j}\left(2\mu_{j}+\dots + 2\mu_{j}^{2k-j}+\mu_{j}^{2k-j+1} + \dots + \mu_{j}^{j}\right)\\
        &> 1 - \mu_{j+1} - 2\mu_{j+1}\left(2\mu_{j+1}+\dots + 2\mu_{j+1}^{2k-j}+\mu_{j+1}^{2k-j+1} + \dots + \mu_{j+1}^{j}\right).
    \end{align*}
    This implies that $\mu_{j} < \mu_{j+1}$ and in particular, $\mu_{k} < \mu_{j}$ for all $k \leq j \leq 2k-1$.
    }
    The range of $x$-exponents in (2a) is also $[-(k-1),k]$ for $j=k$ so altogether the growth rate of Type (2) words is greatest when $j=k$, and corresponds to $x$-exponents in $[-(k-1),k]$. 
\end{proof}

\begin{theorem}\label{thm:BS(2k,2k)}
    The conjugacy growth series of $\mathrm{BS}(p,p)$ with $p=2k>2$ is transcendental, with respect to the free product generating set.
\end{theorem}

\begin{proof}
When considering all types of geodesics in Table \ref{tab:BS(p,p)}, we note that for $j=k$ the $x$-exponents of Type (2) words contain the range of $x$-exponents for Type (3), and so will have a larger growth rate than Type (3). This, combined with \cref{prop: 3 bigger than 1}, implies that the growth rate of Type (2) when $j=k$ will dominate, so Type (1) and (3) do not play a role in the comjugacy growth asymptotics.

For Type (2b), using similar arguments to those in Section \ref{subsec:odd} leads to counting split cyclic representatives of words of the form
\[ w = A_{1}x^{k}A_{2}x^{k}\dots A_{\tau_{1}}x^{k}A_{\tau_{1}+1}x^{-k}A_{\tau_{1}+2}x^{-k}\dots x^{-k}A_{\tau_{1}+\tau_{2}+1}, 
\]
where $A_i$ are reduced words over $x^{\alpha},y^{\beta}$ with
             $\alpha \in [-(k-1), k-1] \cap \mathbb{Z}_{\neq 0}, \beta \in \Z_{\neq 0}$ (see Prop. \ref{prop:sets B}). That is, we count the tuples $(A_1, \dots, A_{\tau_{1}+\tau_{2}+1})$ up to cyclic permutations, while keeping the $X$-structure fixed. Similar to the proof of Proposition \ref{prop:oddBC}, the growth series of such split cyclic representatives is the sum of a rational series whose coefficients have asymptotics $\sim \mu_k^{-n}$, and a series whose coefficients have growth rate strictly lower than $\mu_k^{-1}.$

             For Type (2a) we get a conjugacy growth series whose coefficients have asymptotics $\frac{\mu_k^{-n}}{n}$, so the series is transcendental. 
             Altogether, the conjugacy growth asymptotics for $BS(p,p)$ is $\mu_k^{-1}$ (from Type (2b)) and the series is transcendental as the sum of a rational series (from Type (2b)) together with series whose coefficients have asymptotics $\sim\frac{\mu_k^{-n}}{n}.$
\comm{ As in the proof of \cref{prop:oddB} we write $x^{-k} = x^{-2k}x^{k}$ and move the central term $x^{-2k}$ to the right; the cumulative growth of cyclic representatives of words of the form \[ u = A_{1}x^{k}A_{2}x^{k}\dots A_{\tau_{1}}x^{k}A_{\tau_{1}+1}x^{k}A_{\tau_{1}+2}x^{k}\dots x^{k}A_{\tau_{1}+\tau_{2}+1}
\] can be counted instead, which gives asymptotics as in (\ref{eq:Coornaert}) that dominate all the other types of conjugacy classes, and so
transcendental growth follows.
}
\end{proof}

\subsubsection{Conjugacy for the even dihedral Artin groups $G(2p)=\mathrm{BS}(p,p)$, $p=2k+1$}
We use the classification of geodesics in Table \ref{tab:p=2k+1} and similar techniques as for $p=2k$. We only sketch the main steps in our argument since they follow the ideas from previous sections. 

For Type (2) we proceed as in Proposition \ref{prop:j=k} and can show that the maximum growth rate for words of Type (2a) is obtained when $j=k$, and the range of $x$-exponents is $[-k,k]$; in this case the range of $x$-exponents in the suffixes in (2b) is $[-(k+1),k-1]$. For the suffixes of Type (2b) words, the maximum growth rate is obtained when $j=k+1$ and their $x$-exponents are in $[-k,k]$, which will give a range of $[-(k-1),k+1]$ for (2a). Since there is a length-preserving symmetry between the (2a) $\cup$ (2b) words with $j=k$ and those with $j=k+1$ we can choose $j=k$ and conclude that 
 the growth rate of Type (2) reaches a maximum when $j=k$. The conjugacy growth of Type (2) classes can then be obtained by counting cyclic representatives, as per Section \ref{sec:G(4k+2)}, of words of the form
\[ w = A_{1}x^{k}A_{2}x^{k}\dots A_{\tau_{1}}x^{k}A_{\tau_{1}+1}x^{-(k+1)}A_{\tau_{1}+2}x^{-(k+1)}\dots x^{-(k+1)}A_{\tau_{1}+\tau_{2}+1}. 
\]
Here we use similar techniques as in \cref{prop:oddB}: since $x^{-(k+1)} = x^{-(2k+1)}x^{k}$, the central term $x^{-(2k+1)}$ can be moved to the right to count words of the form \[ w' = A_{1}x^{k}A_{2}x^{k}\dots A_{\tau_{1}}x^{k}A_{\tau_{1}+1}x^kA_{\tau_{1}+2}x^k\dots x^kA_{\tau_{1}+\tau_{2}+1}\Delta^{l}, 
\] similar to (\ref{eqn:30+}). The counting of conjugacy classes then boils down to the cumulative growth of words like $w'$, up to cyclic permutation. This gives asymptotics of the form $\sim \mu^{-n}/n$, while the Type (3) conjugacy classes also give $\sim \mu^{-n}/n$ asymptotics as the $x$-exponent range is the same. Type (2) and (3) dominate the Type (1) growth to give transcendental conjugacy growth for $\mathrm{BS}(2k+1,2k+1)$.
\begin{table}[!h]
    \begin{adjustbox}{width=1\textwidth}
    \begin{tabular}{|c|c|c|c|}
        \hline
        Reference & Type & Form and conditions  & Unique/Non-unique \\
        \hline \hline
        \cref{prop:geos 2k+1} & (1) & $\begin{array}{c}
            w = (x^{\mu_{1}}y^{\ep_{1}}\dots x^{\mu_{n}}y^{\ep_{n}}\Delta^{c})^{\pm 1}  \\
            c, n \geq 1,  \\
            \mu_{1} \in [0, 2k] \cap \Z, \; \mu_{i} \in [1, 2k] \cap \Z \;(2 \leq i \leq n), \\
            \ep_{i} \in \Z \; (1 \leq i \leq n), \ep_{i} \neq 0 \; (1 \leq i \leq n-1).
        \end{array}$
          & \multirow{2}{*}{Unique}\\\cline{1-3}
        \multirow{2}{*}{Prop. \ref{prop:geos 2k+1}, Def. \ref{rmk:alternative1}}& $(2a)$ & $\begin{array}{c}
            w = A_{1}x^{j}A_{2}x^{j}\dots A_{\tau_{1}}x^{j}A_{\tau_{1}+1}, \\
            k \leq j \leq 2k-1\\
            A_{s} \; \text{reduced words over} \; x^{\alpha},y^{\beta}, \\
            \alpha \in [-2k+j, j-1] \cap \mathbb{Z}_{\neq 0}, \beta \in \Z\\  
        \end{array}$  & \\
        \cline{2-4}
        & $(2b)$ & $\begin{array}{c}
            w = A_{1}x^{j}A_{2}x^{j}\dots A_{\tau_{1}}x^{j}A_{\tau_{1}+1}x^{-(2k-j+1)}A_{\tau_{1} + 2}x^{-(2k-j+1)}\dots x^{-(2k-j+1)}A_{\tau_{1} + \tau_{2} + 1},  \\
            A_{s} \text{ as in} \; (2a) \\
            \tau_{2} \geq 1
            \end{array}$ & Non-unique \\
            \hline
        \cref{prop:geos 2k+1} & $(3)$ & $\begin{array}{c}
        w = x^{\alpha}y^{\ep_{1}}x^{\nu_{1}}\dots x^{\nu_{n}} \\
        n \geq 0  \\
        \alpha, \nu_{n} \in [-k, k] \cap \Z, \; \nu_{i} \in [-k, k] \cap \Z_{\neq 0} \; (1 \leq i \leq n-1)  \\
        \ep_{i} \in \Z_{\neq 0} \; (1 \leq i \leq n) 
        \end{array}$  & Unique \\
        \hline

    \end{tabular}
    \end{adjustbox}
    \caption{Subcase of even dihedral Artin groups: $\mathrm{BS}(p,p)$ with $p=2k+1$}
    \label{tab:p=2k+1}
\end{table}\\
This together with Theorems \ref{thm:BS 22} and \ref{thm:BS(2k,2k)} give:
\begin{theorem}\label{thm:even dihe trans}
    The conjugacy growth series of an even dihedral Artin group is transcendental, with respect to the free product generating set.
\end{theorem}



\subsection{Conjugacy growth asymptotics and growth rates}
\comm{
Recall that Artin groups of XXL-type are those where all edge labellings are $\geq 5$. Several recent results from the literature give the following.

\begin{prop}\label{prop:rank least 3}
The conjugacy growth of Artin groups of XXL-type of rank at least 3 is transcendental, with respect to the standard generating set.
\end{prop}

\begin{proof}
By \cite[Thm. 1.4]{Haettel2022}, every Artin group of XXL-type is the fundamental group of a compact locally CAT(0) space, with some element which acts as a rank 1 isometry. By the Cartan-Hadamard Theorem (\cite{Bridson}, II.4.1) any simply connected, locally CAT(0) metric space is globally CAT(0), so the fundamental group of a compact, locally CAT(0) space acts cocompactly on the universal cover, which is globally CAT(0). Therefore, every  Artin group of XXL-type acts on a globally CAT(0) space with a rank 1 isometry which is contracting \cite{Bestvina2009}. This implies that the conjugacy growth series is transcendental, by \cite[Thm. 1.7]{GekhtYang}. 
\end{proof}
We note this result holds with respect to the standard generating set, rather than the free product generating set. Therefore by \cref{main result} and \cref{prop:rank least 3}, we get the following.

\XXLtype}
A careful analysis (see Remark \ref{rmk:rates}, Theorem \ref{thm:BS 22}, proofs of Theorems \ref{thm:BS(2k,2k)}, \ref{thm:even dihe trans}) of the asymptotics obtained in the last sections, together with the explicit standard growth rates of dihedral Artin groups from \cite{Fujii2018} and \cite{EdjJohn92}, gives the following.

\asymptotics

We then immediately get that the standard and conjugacy growth rates are the same. The same result can be obtained via arguments involving languages (see Remark \ref{rmk:equal}).  
\begin{cor}\label{cor:growth_rates}
    The conjugacy growth rate and the standard growth rate in a dihedral Artin group $G(m)$ are the same with respect to the free product generating set. That is, $$\lim_{n \rightarrow \infty} \sqrt[n]{c_{G(m),\{x,y\}}(n)}=\lim_{n \rightarrow \infty} \sqrt[n]{s_{G(m),\{x,y\}}(n)}.$$
\end{cor}

The standard growth rate has been studied intensively in a variety of groups, and the limit is known to always exist. For the conjugacy growth the limit doesn't always exist (\cite{Hull2013}), but the careful analysis of asymptotics shows that it does for dihedral Artin groups. 

\section{Conjugacy geodesics in dihedral Artin groups}\label{sec:conj geos language}
The aim of this section is to prove the following.
\conjgeos
Our strategy uses two ingredients, the first of which concerns the  \emph{permutation conjugacy length function}, defined by Antol\'{i}n and Sale \cite{Antolin2016}. If $u,v \in X^{\ast}$ represent conjugate elements in a group $G=\langle X\rangle$, then an element $w \in G$ is a \emph{permutation conjugacy conjugator} (PC-conjugator) for $u$ and $v$ if there exist cyclic permutations $u'$ and $v'$ of $u$ and $v$ respectively, such that $wu' =_{G} v'w$. Define
\[ \mathrm{PCL}_{G,X}(u,v) := \mathrm{min} \{l(w) \; | \; w \; \text{is a PC-conjugator for} \; u \; \text{and} \; v \}.
\]
The \emph{permutation conjugacy length function} of $G$ is defined as
\[ \mathrm{PCL}_{G,X} := \mathrm{max} \{\mathrm{PCL}_{G,X}(u,v) \; | \; u,v \; \text{geodesics satisfying} \; l(u)+l(v) \leq n \}.
\]
The following result links the PCL function to the language of conjugacy geodesics.

\begin{prop}\label{prop: PCL language}\cite[Prop. 2.3]{Antolin2016}\\
    Suppose that $\mathrm{PCL}_{G,X}$ is constant. If $(G,X)$ has the falsification by fellow traveler property, then $\mathsf{ConjGeo}(G,X)$ is regular.
\end{prop}
Using results from \cref{sec:odd conj geos} and \cref{even conj geos}, we show that dihedral Artin groups satisfy the two conditions of \cref{prop: PCL language}, with respect to the free product generating set. We first establish the constant bound on the PCL function.

\PCLconstant

\begin{proof}
    From \cref{sec:odd conj geos}, the language of conjugacy geodesics for odd dihedral Artin groups is precisely 
    \[ \mathsf{ConjGeo}(G,X) = \overline{\mathcal{A}} \sqcup \{ w \in \mathcal{B} \; | \; w \; \text{starts and ends with opposite letters} \}.
    \]
    By \cref{prop:prod equal}, if two words $u,v \in \overline{\mathcal{A}}$ represent conjugate elements, then they are equal up to a cyclic permutation, and so $\mathrm{PCL}_{G,X}(u,v) = 0$ in these cases. Now consider words of Type $\left(3^{0+}N\right)$. First note that for any word $w \in (3^{0+}N) \cap \mathsf{ConjGeo}(G,X)$, then $w =_{G} \overline{w} \in \overline{\mathcal{B_{+}}}$, by applying \cref{eq:a non unique} to relevant powers. Therefore for any $u,v \in (3^{0+}N)$,
    \begin{equation}\label{eq: sequence}
        u =_{G} \overline{u} \xleftrightarrow{s} \overline{v} =_{G} v,
    \end{equation}
    by \cref{prop:conj geo B}, where $\xleftrightarrow{s}$ represents a split cyclic permutation, and $\overline{u}, \overline{v} \in \overline{\mathcal{B_{+}}}$. In particular, if $u, v \in \overline{\mathcal{B_{+}}}$, then $\mathrm{PCL}_{G,X}(u,v) = 0$ by \cref{prop:conj geo B}. We claim that there exists a shorter sequence of the form $u \leftrightarrow w =_{G} v$, where $\leftrightarrow$ represents a cyclic permutation. To see this, let
    \[ u = A_{1}y^{a_{1}}A_{2}y^{a_{2}}\dots y^{a_{\tau}}A_{\tau+1}, \; v = y^{b_{1}}A_{t}y^{b_{2}}A_{t+1}\dots y^{b_{\tau}}A_{t-1},
    \]
    for some $1 \leq t \leq \tau + 1$, where the tuples $(y^{a_{1}}, \dots y^{a_{\tau}})$ and $(y^{b_{1}}, \dots y^{b_{\tau}})$ consist of $\tau_{1}$ $y^{-k}$ terms and $\tau_{2}$ $y^{k+1}$ terms respectively, in any order within each tuple. Note we can assume these tuples from $u$ and $v$ have the same number of $y^{-k}$ terms and $y^{k+1}$ terms by \cref{eq: sequence}. We can first cyclically permute $u$ to a word $w$ of the form
    \[ w = y^{a_{t-1}}A_{t}y^{a_{t}}A_{t+1}\dots y^{a_{t-2}}A_{t-1}.
    \]
    Then $w =_{G} v$ since the tuples $(y^{a_{t-1}}, \dots y^{a_{t-2}})$ and $(y^{b_{1}}, \dots, y^{b_{\tau}})$ are equivalent by repeated applications of \cref{eq:a non unique}. Hence $\mathrm{PCL}_{G,X}(u,v) = 0$. This method also holds for words of Type $(3^{0-}N)$ and $(3^{0*})$, which covers all words in $\mathsf{ConjGeo}(G,X)$.

    From \cref{even conj geos}, the language of conjugacy geodesics for even dihedral Artin groups is precisely
    \[ \mathsf{ConjGeo}(G,X) = \overline{\mathcal{C}} \sqcup \{ w \in \mathcal{D} \; | \; w \; \text{starts and ends with opposite letters} \}.
    \]
    Showing PCL is constant is analogous to the odd cases, using Propositions \ref{prop:conj geos cyc perm} and \ref{prop:sets B}.
\end{proof}
\begin{rmk}\label{rmk:equal}
    We can use \cite[Proposition 2.2]{Antolin2016} and \cref{thm:PCL constant} to obtain a different proof of \cref{cor:growth_rates}, which establishes the equality of the exponential growth rate and exponential conjugacy growth rate of any dihedral Artin groups with respect to the `free product' generators.
\end{rmk}
For a group $G = \langle X \rangle$ and $k \geq 0$, we say $w, w' \in X^{\ast}$ \emph{k-fellow travel} if for each $i \geq 0$, $|\mathrm{pre}_{i}(w)^{-1}\mathrm{pre}_{i}(w')| \leq k$. We say $G = \langle X \rangle$ satisfies the \emph{falsification by fellow traveler property} (FFTP) if, for some fixed constant $k$, any non-geodesic word $w \in X^{\ast}$ $k$-fellow travels with a shorter word. 

We now prove the second condition of \cref{prop: PCL language}.

\FFTP

\begin{proof}
We will show that the Cayley graph $\Gamma = \Gamma(G,X)$ satisfies the FFTP by working with labels of paths in $\Gamma$, which are words over $X$. Our strategy is to consider odd and even cases in turn, and create a series of reductions, all of which have bounded fellow traveler constant, such that any words $w \in X^{\ast}$, where we cannot apply any further reductions, is geodesic. 

First consider the odd case $G= \langle x,y\mid x^2=y^{2k+1}\rangle$. Our series of reductions is as follows: 
\begin{enumerate}
\item[(R0)] If $w\in X^*$ is not freely reduced, then it $2$-fellow travels with a shorter word that is obtained by performing a single free reduction to $w$.

\item[(R1)] If $w\in X^*$ has the form $w= w_1 x^{2\e} w_2 x^{- \e}w_3$  (resp. $w= w_1 x^{-\e} w_2 x^{2\e}w_3$), with $\e\in \{\pm 1\}$, then $w$ $4$-fellow travels with $w_1w_2x^{\e}w_3$ (resp. $w_1x^{-\e}w_2w_3$).

\item[(R2)] If $w\in X^*$ has the form $w= w_1 y^{\e} x^{2l} y^{-\e}w_3$, with $l \in \Z_{\neq 0}$ and $\e \in \{\pm 1\}$, then $w$ $2$-fellow travels with $w_1x^{2l}w_3$.

\item[(R3)] If $w\in X^*$ contains $y^b$ with $|b|\geq k$, then

\begin{itemize}
\item[(R3a)] If $|b|>k+1$, then $w$ $(2k+2)$-fellow travels with a shorter word, that is obtained by replacing $y^{k+2}$ by $x^2y^{-(k-1)}$ (resp. $y^{-k-2}$ by $x^{-2}y^{k-1}$).
\item[(R3b)] If $|b|=k+1$ and $w$ has the form $w= w_1x^{\e}w_2y^{-\e(k+1)}w_3$ (resp.  $w= w_1y^{-\e(k+1)}w_2x^{\e}w_3$), with $\e\in \{\pm 1\}$,  then $w$ $(2k+9)$-fellow travels with $w_1 x^{-\e} w_2 y^{\e k}w_3$ (resp. $w_1  y^{\e k}w_2x^{-\e}w_3$). To see this, notice that $w$ 4-fellow travels with $w_1x^{\e}x^{-2\e}x^{2\e}w_2y^{-\e (k+1)}w_3$,
 which then $2$-fellow travels with $w_1x^{-\e}w_2 x^{2\e}y^{-\e (k+1)}w_3$, which finally $(2k+3)$-fellow travels with $w_1x^{-\e}w_2 y^{\e k}w_3$.
 
 \item[(R3c)]  If $|b|=k$ and $w$ has the form $w_1x^{2\e}w_2y^{-k\e}w_3$  or $w_1y^{-k\e}w_2x^{2\e}w_3$, with $\e \in \{\pm 1\}$, then $w$ $4$-fellow travels with $w_1w_2x^{2\e}y^{-k\e}w_3$, which then $(2k+1)$-fellow travels with $w_1w_2 y^{\e (k+1)}w_3$ (see Figure \ref{fig:x2k}).

\begin{figure}[h]
    \centering
    \begin{tikzpicture}
    [decoration ={markings, mark=at position 0.5 with {\arrow{stealth}}}]
        \filldraw[black] (0,0) circle (1pt);
        \filldraw[black] (2,0) circle (1pt);
        \filldraw[black] (3,2) circle (1pt);
        \filldraw[black] (7,2) circle (1pt);
        \filldraw[black] (6,0) circle (1pt);
        \filldraw[black] (8,1) circle (1pt);
        \filldraw[black] (10,1) circle (1pt);
        \draw[postaction={decorate}, blue] (0,0) -- (2,0) ;
        \draw[postaction={decorate}, green] (2,0) -- (3,2);
        \draw[postaction={decorate}, blue] (3,2) -- (7,2);
        \draw[postaction={decorate}, blue] (2,0) -- (6,0);
        \draw[postaction={decorate}, green] (6,0) -- (7,2);
        \draw[postaction={decorate}, red] (6,0) -- (8,1);
        \draw[postaction={decorate}, red] (8,1) -- (7,2);
        \draw[postaction={decorate}, blue] (8,1) -- (10,1);
        \draw[dashed, green] (3,0) -- (4,2);
        \draw[dashed, green] (4,0) -- (5,2);
        \draw[dashed, green] (5,0) -- (6,2);
        \node at (1, 0.3) {$w_{1}$};
        \node at (2.1, 1) {$x^{2}$};
        \node at (6.1, 1) {$x^{2}$};
        \node at (4, -0.3) {$w_{2}$};
        \node at (5, 2.3) {$w_{2}$};
        \node at (7.7, 1.8) {$y^{k}$};
        \node at (9, 1.2) {$w_{3}$};
        \node at (7.1, 0.2) {$y^{k+1}$};
    \end{tikzpicture}
    \caption{Case (R3c)}
    \label{fig:x2k}
\end{figure}
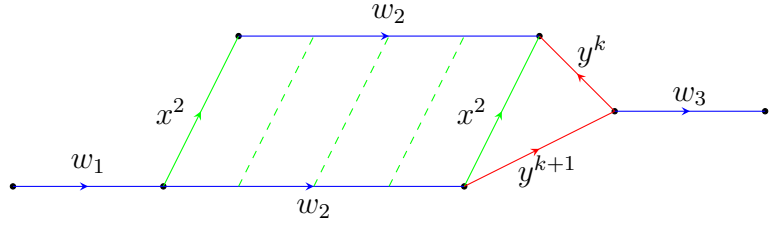
\end{itemize} 
\item[(R4)] If $w\in X^*$ has the form $w = w_1 y^{a}x^{2l} y^b w_2$ with $l \in \Z_{\neq 0}$, $a\cdot b >0$ and $|a+b|>k+1$, then $w$ $2a$-fellow travels with $w_1x^{2l}y^{a+b}w_2$, which by (R3) $(2k+2)$-fellow travels with a shorter word. By (R3) we can assume that $|a|\leq k+1$, otherwise we could have applied (R3) first. Hence $w$ $2(a+k+1)$-fellow travels with a shorter word, provided that (R3) cannot be applied (see Figure \ref{fig:ax2b}).

\begin{figure}[h]
    \centering
    \begin{tikzpicture}
    [decoration ={markings, mark=at position 0.5 with {\arrow{stealth}}}]
        \filldraw[black] (0,0) circle (1pt);
        \filldraw[black] (2,0) circle (1pt);
        \filldraw[black] (3,2) circle (1pt);
        \filldraw[black] (7,2) circle (1pt);
        \filldraw[black] (6,0) circle (1pt);
        \filldraw[black] (8,1) circle (1pt);
        \filldraw[black] (10,1) circle (1pt);
        \filldraw[black] (7,0) circle (1pt);
        \draw[postaction={decorate}, blue] (0,0) -- (2,0) ;
        \draw[postaction={decorate}, green] (2,0) -- (3,2);
        \draw[postaction={decorate}, blue] (3,2) -- (7,2);
        \draw[postaction={decorate}, blue] (2,0) -- (6,0);
        \draw[postaction={decorate}, green] (6,0) -- (7,2);
        \draw[postaction={decorate}, red] (7,0) -- (7.75,1.25);
        \draw[postaction={decorate}, blue] (6,0) -- (7,0);
        \draw[postaction={decorate}, red] (7,2) -- (8,1);
        \draw[postaction={decorate}, blue] (8,1) -- (10,1);
        \draw[dashed, green] (3,0) -- (4,2);
        \draw[dashed, green] (4,0) -- (5,2);
        \draw[dashed, green] (5,0) -- (6,2);
        \node at (1, 0.3) {$w_{1}$};
        \node at (2.1, 1) {$y^{a}$};
        \node at (6.2, 1) {$y^{a}$};
        \node at (4, -0.3) {$x^{2l}$};
        \node at (5, 2.3) {$x^{2l}$};
        \node at (7.7, 1.8) {$y^{b}$};
        \node at (9, 1.2) {$w_{2}$};
        \node at (7.8, 0.5) {$y^{k-1}$};
        \node at (6.4, -0.3) {$x^{2}$};
    \end{tikzpicture}
    \caption{Case (R4)}
    \label{fig:ax2b}
\end{figure}

%
%
%
\end{enumerate}
Let $w\in X^*$ be a word for which we cannot apply any reduction (R0)-(R4). Since $x^2$ is central, we can move any $x^2$ terms to the right to obtain a word $w'$ of the same length as $w$, representing the same element as $w$. Note that $w$ is geodesic if and only if $w'$ is geodesic, and we can apply a reduction from (R0)-(R4) to $w$ if and only if we can apply a reduction from (R0)-(R4) to $w'$.
Thus, without loss of generality we can assume $w$ has the form
$$w = w_0x^{2l}, \qquad w_0 = x^{\alpha_0} y^{\beta_1} x^{\alpha_1} \cdots y^{\beta_{n_{w}}}x^{\alpha_{n_{w}}},$$
where $\alpha_0,\alpha_{n_{w}}\in \{-1,0,1\}$, $\alpha_i\in \{\pm 1 \} \; (1 \leq i \leq n_{w}-1)$, and $\beta_j \in [-(k+1), k+1] \cap \Z_{\neq 0} \; (1 \leq j \leq n_{w})$.
We can assume $l$ has been chosen so that $\ell(w)=|2l|+\ell(w_0)$. Let $g \in X^{\ast}$ be a geodesic word representing the same element as $w$. Again we can assume $g$ has the form
$$g = g_0x^{2s}, \qquad g_0 = x^{\sigma_0} y^{\gamma_1} x^{\sigma_1} \cdots y^{\gamma_{n_{g}}}x^{\sigma_{n_{g}}},$$
with $\sigma_0,\sigma_{n_{g}} \in \{-1,0,1\}$, $\sigma_{i} \in \{\pm 1\} \; (1 \leq i \leq n_{g}-1)$, $\gamma_j \in [-(k+1), k+1] \cap \Z_{\neq 0} \; (1 \leq j \leq n_{g})$, and $\ell(g)=|2s|+\ell(g_0)$. 

Since $g$ and $w$ represent the same element of $G$, when viewed as elements of $\overline{G} \cong C_{2} \ast C_{2k+1} =\langle x,y \mid x^2= y^{2k+1}=1\rangle$, they represent the same element. Therefore $w_0$ and $g_0$ label paths with the same endpoints in $\overline{\Gamma}=\Gamma(\overline{G},X)$, and so $n_{w}=n_{g}$.  
Also, $\alpha_0=0$ if and only if $\sigma_0=0$, and similarly $\alpha_{n_{g}}=0$ if and only if $\sigma_{n_{g}}=0$. We now consider the $x$ and $y$ exponents of $w_{0}$ and $g_{0}$, and when these differ. We denote these by $\pmb{\beta}=(\beta_1,\dots, \beta_{n_{g}})$,  $\pmb{\gamma}=(\gamma_1,\dots, \gamma_{n_{g}})$, $\pmb{\alpha}= (\alpha_0,\dots, \alpha_{n_{g}})$ and $\pmb{\sigma}=(\sigma_0,\dots, \sigma_{n_{g}})$. Note that if $|\beta_i|<k$ then $\beta_i=\gamma_i$, since one cannot perform a reduction (R3) to $w_0$ or $g_0$, and $g_{0}$ is geodesic. Therefore if $\pmb{\beta}$ and $\pmb{\gamma}$ do not coincide at the $i$-th coordinate, then  $\beta_i=\pm k$ (resp. $=\mp (k+1)$) and $\gamma_i$ = $\pm (k+1)$  (resp. $=\mp k$). 

Let $c\geq 0$ be the number of different coordinates of  $\pmb{\beta}$ and $\pmb{\gamma}$. Let $i_1<\dots< i_c\in \{1,\dots, n_{g}\}$ such that $\beta_{i_j}\neq \gamma_{i_j}$, and suppose there exists $h,j$ such that $\beta_{i_h}\cdot \beta_{i_j}< 0$.
For simplicity, we denote $\beta_{i_h}$ by $\beta_h$ and $\beta_{i_j}$ by $\beta_j$. Similarly denote by $\gamma_h$ and $\gamma_j$ the corresponding coordinates of $\bm{\gamma}$. We cannot have $|\beta_h|=|\beta_j|=k+1$, as then we can apply (R3b) with some $x^{\e}$ of opposite sign. 
If we have $|\beta_h|=|\beta_j|=k$, then $|\gamma_h|=|\gamma_j|= k+1$ and $\gamma_h\cdot \gamma_j <0$, and then we can apply (R3b) to $g_0$, which is a contradiction since $g_0$ is geodesic. 
Therefore $|\beta_h|\neq |\beta_j|$, which implies that $\gamma_h=\beta_j$ and $\gamma_j=\beta_h$.
Then $w_0 = w_1 y^{\beta_h}w_2 y^{\beta_j}w_3$ $(2k+2)$-fellow travels with $w_0' = w_1 y^{\beta_j}w_2 y^{\beta_h}w_3$. For example, Figure \ref{fig:upanddown} describes the case where $(\beta_h, \beta_j)=(k, -(k+1))$.

\begin{figure}[h]
    \centering
    \begin{tikzpicture}
    [decoration ={markings, mark=at position 0.5 with {\arrow{stealth}}}]
        \filldraw[black] (0,0) circle (1pt);
        \filldraw[black] (2,0) circle (1pt);
        \filldraw[black] (3,1) circle (1pt);
        \filldraw[black] (3,-1) circle (1pt);
        \filldraw[black] (6,1) circle (1pt);
        \filldraw[black] (6,-1) circle (1pt);
        \filldraw[black] (7,0) circle (1pt);
        \filldraw[black] (9,0) circle (1pt);
        \draw[postaction={decorate}, blue] (0,0) -- (2,0);
        \draw[postaction={decorate}, red] (2,0) -- (3,1);
        \draw[postaction={decorate}, red]  (3,-1) --(2,0);
        \draw[postaction={decorate}, blue] (3,1) -- (6,1);
        \draw[postaction={decorate}, blue] (3,-1) -- (6,-1);
        \draw[postaction={decorate}, red]  (6,-1) -- (7,0);
        \draw[postaction={decorate}, red] (7,0) -- (6,1);
        \draw[postaction={decorate}, blue] (7,0) -- (9,0);
        \draw[dashed, red] (3.5,1) -- (3.5,-1);
        \draw[dashed, red] (4,1) -- (4,-1);
        \draw[dashed, red] (4.5,1) -- (4.5,-1);
        \draw[dashed, red] (5,1) -- (5,-1);
        \draw[dashed, red] (5.5,1) -- (5.5,-1);
        \node at (1, 0.2) {$w_{1}$};
        \node at (4.5, 1.2) {$w_{2}$};
        \node at (4.5, -1.3) {$w_{2}$};
        \node at (8,0.2) {$w_{3}$};
        \node at (2.3, 0.7) {$y^{k}$};
        \node at (2.1, -0.8) {$y^{k+1}$};
        \node at (7.1, 0.7) {$y^{k+1}$};
        \node at (6.7, -0.7) {$y^{k}$};
    \end{tikzpicture}
    \caption{Two paths of the same length that $(2k+2)$-fellow travel.}
    \label{fig:upanddown}
\end{figure}
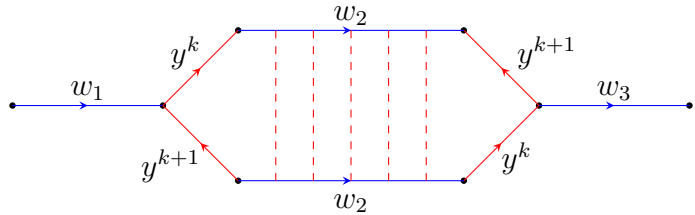
Notice that $\ell(w_0')= \ell(w_0)$ and $w_0=_G w_0'$, and so $w = w_0x^{2l}$ is geodesic if and only if $w' = w_0'x^{2l}$ is geodesic.
Let $\pmb{\beta}'$ be the vector of $y$ exponents of $w'$.
Then the number of disagreements between $\pmb{\beta}'$ and $\pmb{\gamma}$ is $c-2$. Note also that since we cannot apply (R0)-(R4) to $w$, we cannot apply them to $w'$.
Hence, by the previous argument, we can assume that one of the following holds:
\begin{multicols}{2}
    \begin{enumerate}
\item[(Y--)] For all $\beta_i\neq \gamma_i$, $\beta_i<0<\gamma_i$.
\item[(Y+)] For all $\beta_i\neq \gamma_i$, $\beta_i>0>\gamma_i$.
\end{enumerate}
\end{multicols}
Let $d\geq 0$ be the  number of different coordinates of  $\pmb{\alpha}$ and $\pmb{\sigma}$. Similar to the argument above, if there are $h\neq j$ such that $\alpha_{i_j}\cdot \alpha_{i_h}<0$ and $\alpha_{i_j}\neq \sigma_{i_j}$ and $\alpha_{i_h}\neq \sigma_{i_h}$, then we can change the signs of $\alpha_{i_j}$ and $\alpha_{i_h}$, and obtain a word of the same length as $w$, which represents the same element as $w$ and (R0)--(R4) are also satisfied. Hence $d$ has decreased by $2$, and so we can assume that the following condition holds.
\begin{enumerate}
\item[(X+)] For all $1 \leq i \leq m$ such that $\alpha_i\neq \sigma_i$ we have that $\alpha_i=1$ and $\sigma_i=-1$.
\end{enumerate}
Indeed,  the other option is that for all $i$ such that $\alpha_i\neq \sigma_i$, we have that $\alpha_i=-1$ and $\sigma_i=1$. However we can assume (X+) by changing $w$ and $g$ by $w^{-1}$ and $g^{-1}$ if necessary. Note that when assuming (X+), we have that $l\geq 0$ and $s\leq 0$ by (R1). If (Y+) holds then $g_0^{-1}w_{0}=_G x^{2(d+c)}$. Since $w=_G g$, this implies that $g_0x^{2(l+d+c)}=_{G} g_0x^{2s}$. Since $s\leq 0$ and $l,d,c\geq 0$, we must have that $s=l=d=c=0$. Hence $w = w_0 = g_0 = g$ and $w$ is geodesic. 

If (Y--) holds, then $g^{-1}_{0}w_0=_G x^{2(d-c)}$ and similar to before we have that $l+d-c = s$. If $d = c= 0$,  then $w_0 = g_0$ and $s=l=0$, therefore $w = g$ and $w$ is geodesic. If $c\neq 0$, then $l=0=s$. Indeed, otherwise if $l>0$, then we can apply (R3c) to $w$, and similarly if $s<0$, then we can apply (R3c) to $g$. As no reduction on $w$ and $g$ is possible, then $l=s = 0$. Moreover, $d=c$ since $l+d-c=s$. As $d>0$, then $w_0$ must contain a positive $x$-term, and then by (R3c) all the $\beta_i\neq \gamma_i$ must be equal to $-k$. But this means that all $\gamma_i\neq \beta_i$ must be equal to $k+1$ and therefore $\ell(w)=\ell(w_0)=\ell(g_0)-c <\ell(g)$, which is a contradiction. Hence $c=0$. If $d\neq 0$, since $0\leq s = l+d-c$, then $c>0$. However this can't hold by the previous case, and so $d=0$. This completes the proof for odd dihedral Artin groups.

For even cases, recall $G = \langle x, y \; | \; y^{-1}x^{p}y  = x^{p} \rangle$. We give details of the proof when $p = 2k$, and the case for $p=2k+1$ follows similarly. Again we define our series of reductions, all of which have bounded fellow traveler constant.
\begin{enumerate}
\item[(R0)] If $w\in X^*$ is not freely reduced then it $2$-fellow travels with a shorter word that is obtained by performing a single free reduction to $w$.

\item[(R1)] If $w\in X^*$ has the form $w= w_1 x^{2k\e} w_2 x^{-\e l}w_3$  (resp. $w= w_1 x^{-\e l} w_2 x^{2k\e}w_3$), with $\e\in \{\pm 1\}$ and $1 \leq l \leq 2k-1$, then it $(4k)$-fellow travels with $w_1w_2x^{\e(2k-l)}w_3$ (resp.  $w_1x^{\e(2k-l)}w_2w_3$).

\item[(R2)] If $w\in X^*$ has the form $w= w_1 y^{\e} x^{2kl} y^{-\e}w_3$, with $l\neq 0$ and $\e \in \{\pm 1\}$, then it $2$-fellow travels with $w_1x^{2kl}w_3$.

\item[(R3)] If $w \in X^{\ast}$ has the form $w = w_{1}x^{\sigma}w_{2}x^{\delta}w_{3}$ (resp. $w = w_{1}x^{\delta}w_{2}x^{\sigma}w_{3}$), where $\sigma > 0 > \delta$ and $\sigma + |\delta | > 2k$, then it $8k$-fellow travels with $w_{1}x^{\sigma-2k}w_{2}x^{\delta+2k}w_{3}$ (resp. $w = w_{1}x^{\delta+2k}w_{2}x^{\sigma-2k}w_{3}$).

\item[(R4)] If $w\in X^*$ has the form $w = w_1 y^{a}x^{2kl} y^b w_2$ with $l\neq 0$, $a\cdot b >0$, then it $2a$-fellow travels with $w_1x^{2kl}y^{a+b}w_2$.
\end{enumerate}
Let $w\in X^*$ be a word for which we can not apply any reduction (R0)-(R4). Since $x^{2k}$ is central, we can move all $x^{2k}$ terms to the right to obtain a word $w'$ of the same length as $w$, representing the same element as $w$. Note that $w$ is geodesic if and only if $w'$ is geodesic, and we cannot apply a reduction from (R0)-(R4) to $w$ if and only if we cannot apply a reduction from (R0)-(R4) to $w'$.
Thus, without loss of generality we can assume $w$ has the form
$$w = w_0x^{2kl}, \qquad w_0 = x^{\alpha_0} y^{\beta_1} x^{\alpha_1} \cdots y^{\beta_{n_{w}}}x^{\alpha_{n_{w}}},$$
with $\alpha_0,\alpha_{n_{w}} \in [-(2k-1), 2k-1] \cap \mathbb{Z}, \; \alpha_{i} \in [-(2k-1), 2k-1] \cap \mathbb{Z}_{\neq 0} \; (1 \leq i \leq n_{w}-1)$, $\beta_{j} \in \mathbb{Z}_{\neq 0} \; (1 \leq j \leq n_{w})$, and $\ell(w)=|2kl|+\ell(w_0)$. Let $g \in X^{\ast}$ be a geodesic word representing the same element as $w$. Similarly we can assume $g$ has the form
$$g = g_0x^{2ks}, \qquad g_0 = x^{\sigma_0} y^{\gamma_1} x^{\sigma_1} \cdots y^{\gamma_{n_{g}}}x^{\sigma_{n_{g}}},$$
with $\sigma_0, \sigma_{n_{g}}\in [-(2k-1), 2k-1] \cap \mathbb{Z}$, $\sigma_{i} \in [-(2k-1), 2k-1] \cap \mathbb{Z}_{\neq 0} \; (1 \leq i \leq n_{g}-1)$, $\gamma_{j} \in \mathbb{Z}_{\neq 0} \; (1 \leq j \leq n_{g})$, and $\ell(g)=|2ks|+\ell(g_0)$. 

Since $g$ and $w$ represent the same element of $G$, when viewed as elements of $\overline{G} \cong C_{2k} \ast \mathbb{Z} =\langle x,y \mid x^{2k} =1\rangle$, they represent the same element. Therefore $w_0$ and $g_0$ label paths with the same endpoints in $\overline{\Gamma}=\Gamma(\overline{G},X)$, and so $n_{w}=n_{g}$.  
Also, $\alpha_0=0$ if and only if $\sigma_0=0$ and similarly $\alpha_{n_{g}}=0$ if and only if $\sigma_{n_{g}}=0$. Finally, note that $\beta_{i} = \gamma_{i}$ for all $1 \leq i \leq m$, and so we consider the vectors $\pmb{\alpha}= (\alpha_0,\dots, \alpha_{n_{g}})$ and $\pmb{\sigma}=(\sigma_0,\dots, \sigma_{n_{g}})$ of $x$ exponents of $w_{0}$ and $g_{0}$.

Let $d\geq 0$ be the  number of different coordinates of  $\pmb{\alpha}$ and $\pmb{\sigma}$. Let $i_1<\dots< i_d\in \{1,\dots, n_{g}\}$ such that $\alpha_{i_j}\neq \sigma_{i_j}$, and suppose there exists $h,j$ such that $\alpha_{i_h}\cdot \alpha_{i_j}< 0$.
For simplicity, let us denote $\alpha_{i_h}$ by $\alpha_h$ and $\alpha_{i_j}$ by $\alpha_j$. Similarly denote by $\sigma_h$ and $\sigma_j$ the corresponding coordinates of $\pmb{\sigma}$. Without loss of generality let $\alpha_{h} > 0 > \alpha_{j}$ and $\sigma_{h} < 0 < \sigma_{j}$, and note that $\alpha_{h} = 2k + \sigma_{h}$ and $\alpha_{j} = -2k + \sigma_{j}$. Since $g_{0}$ is geodesic, $\sigma_{j} + |\sigma_{h}| = \sigma_{j}-\sigma_{h} \leq 2k$, since otherwise we could apply (R3). We then have
\[ \alpha_{h} + |\alpha_{j}| = 2k + \sigma_{h} + |-2k+ \sigma_{j}| = 2k + \sigma_{h} + 2k - \sigma_{j} = 4k - \sigma_{j} + \sigma_{h}
     \geq 4k - 2k = 2k.
\]
Since we cannot apply (R3) to $w_{0}$, we require $\alpha_{h} + |\alpha_{j}| \leq 2k$, and so $\alpha_{h} + |\alpha_{j}| = 2k = \sigma_{j} + |\sigma_{h}|$. This implies that $\alpha_{j} = \sigma_{h}$ and $\alpha_{h} = \sigma_{j}$. Then $w_0 = w_1 x^{\alpha_h}w_2 x^{\alpha_j}w_3$ $(4k-2)$-fellow travel with $w_0' = w_1 x^{\alpha_j}w_2 x^{\alpha_h}w_3$. Notice that $\ell(w_0')= \ell(w_0)$ and $w_0=_G w_0'$, so $w = w_0x^{2l}$ is geodesic if and only if $w' = w_0'x^{2l}$ is geodesic.
Let $\pmb{\alpha}'$ be the vector of $x$ exponents of $w'$.
Then the number of disagreements between $\pmb{\alpha}'$ and $\pmb{\sigma}$ is $d-2$. Note also that since we cannot apply (R0)-(R4) to $w$, we cannot apply (R0)-(R4) to $w'$.
Hence by the previous argument, we can assume that one of the following holds:
\begin{multicols}{2}
    \begin{enumerate}
\item[(X--)] For all $\alpha_i\neq \sigma_i$, $\alpha_i<0< \sigma_i$.
\item[(X+)] For all $\alpha_i\neq \sigma_i$, $\alpha_i>0> \sigma_i$.
\end{enumerate}
\end{multicols}
Note that when assuming (X--), we have that $l \leq 0$ and $s \geq 0$  by (R1). Similarly when assuming (X+), we have that $l\geq 0$ and $s\leq 0$.

If (X+) holds then $g_0^{-1}w_0 =_G x^{2kd}$. As $w=_G g$, this implies that $g_0x^{2k(l+d)}=_{G} g_0x^{2ks}$. Since $s\leq 0$ and $l,d \geq 0$, we must have that $s=l=d=0$. If (X--) holds, then $g^{-1}_0w_0=_G x^{-2kd}$ and so $l-d = s$. If $d \neq 0$, then $l-s = d > 0$. However $l-s \leq 0$, since $l \leq 0$ and $s \geq 0$, and so $s=l=d=0$. In both cases, $w = w_0 = g_0 = g$ and $w$ is geodesic. This completed the proof for $p=2k$, and the case for $p=2k+1$ follows a similar method. 
\end{proof}
The proof of \cref{conjgeo result} follows from \cref{prop: PCL language}, \cref{thm:PCL constant} and \cref{prop:FFTP}.

\begin{appendices}
\section{Geodesic normal forms}\label{sec:appendix}
Here we provide details on results from \cref{sec:odd conj geos} and \cref{even conj geos} about geodesic normal forms.

\subsection{Odd dihedral Artin groups}\label{append1:odd}
\begin{lemma}\label{lem:geo conditions}\cite[Prop. 3.5]{Fujii2018}
Let $w \in (3)$ be geodesic in $G(m) \cong \langle x,y \; | \; x^{2} = y^{m} \rangle$, $m=2k+1$. Then the following conditions must all be satisfied:
\begin{multicols}{2}
    \begin{enumerate}
    \item $\Pos_{x}(w) + \Neg_{x}(w) \leq 2$
    \item $\Pos_{y}(w) + \Neg_{y}(w) \leq 2k+1$
    \item $\Pos_{x}(w) + \Neg_{y}(w) \leq k+1$
    \item $\Pos_{y}(w) + \Neg_{x}(w) \leq k+1$
\end{enumerate}
\end{multicols}
\end{lemma}
Suppose $w \in (3^{0})$ is geodesic. If $w$ contains an $x$ term, then $w$ cannot contain a $y^{-(k+1)}$ term by \cref{lem:geo conditions}, since otherwise
\[ \Pos_{x}(w) + \Neg_{y}(w) = 1 + k+1 > k+1.
\]
Similarly, if $w$ contains an $x^{-1}$ term, then $w$ cannot contain a $y^{k+1}$ term. Moreover, no element of Type $\left(3^{0}\right)$ is also of Type $(1)$ or $(2)$, since the conditions in \cref{prop:geos} are mutually exclusive. Therefore we can split geodesics of Type $\left(3^{0}\right)$ into disjoint cases as in \cref{30 cases}.
\subsection{Even dihedral Artin groups}\label{appen2: even}
\begin{proof}[Proof of \cref{prop:geos 2k}]
    Observe that each word listed in \cref{prop:geos 2k} is reduced in the sense of HNN-extensions (see \cite[Page 181]{Lyndon2001}). To show each type is mutually exclusive, we have six cases to check:
     $(i) \; g,h \in (1)$,
        $(ii) \; g,h \in (2)$,
        $(iii) \; g,h \in (3)$,
        $(iv) \; g \in (1), h \in (2)$,
        $(v) \; g \in (1), h \in (3)$, and
        $(vi) \; g \in (2), h \in (3)$. Since the proof method is similar for each case, we give details for Cases $(i)$ and $(iv)$ only. 
        
    Suppose $g, h \in (1)$ and
    $ g = x^{\mu_{1}}y^{\ep_{1}}\dots x^{\mu_{n}}y^{\ep_{n}}\Delta^{c} =_{G} x^{\alpha_{1}}y^{\beta_{1}}\dots x^{\alpha_{m}}y^{\beta_{m}}\Delta^{d} = h.
    $
    Without loss of generality, $c, d \geq 1$. Then $gh^{-1}$ must HNN-reduce to the trivial word, and so $\alpha_{1}-\mu_{1} = 0 \pmod{2k}$. Since $\alpha_{1}, \mu_{1} \in \{0, 1, \dots 2k-1\}$, this forces $\alpha_{1} = \mu_{1} = 0$, and $\ep_{1} = \beta_{1}$. Continued reduction must occur, and so we have $n=m, \alpha_{i} = \mu_{i}$ and $\ep_{i} = \beta_{i}$ for all $1 \leq i \leq n$. Therefore $gh^{-1}$ reduces to $\Delta^{d-c} = x^{2k(d-c)}$. Since $G(m)$ is torsion-free, $c=d$, which proves $(i)$.  

    Now suppose $g \in (1), h \in (2)$ and 
    \[ g = x^{\mu_{1}}y^{\ep_{1}}\dots x^{\mu_{n}}y^{\ep_{n}}\Delta^{c} =_{G} (x^{\mu_{1j}}y^{\ep_{1j}}\dots y^{\ep_{n_{j}j}})x^{j}(y^{\delta_{1j}}x^{\gamma_{1j}}\dots x^{\gamma_{m_{j}j}}) = h.
    \]
    Without loss of generality, $c \geq 1$. For notation consider $g^{-1}h = 1$. This forces $\mu_{1j}-\mu_{1} = 0 \pmod{2k}$. If $\mu_{1j} \geq 0$, then $\mu_{1j} = \mu_{1}$. Otherwise, we have $\mu_{1j} - \mu_{1} = -2k$, which gives $x^{-2k} = \Delta^{-1}$. Also $\ep_{1} = \ep_{1j}$. This pattern of either free cancellation or $x^{-2k}$ powers occurs for all pairs, and after continued reduction we are left with $x^{-2k(c + d)} = 1$, where $d$ is the number of $x^{-2k}$ terms which occur from $g^{-1}h$. Since $G(m)$ is torsion-free and $k \geq 2$, then $c = -d$. This is a contradiction since $c \geq 1$, and so $g \neq h$. 
\end{proof}
\begin{rmk}
    Since we use different notation to \cite{EdjJohn92}, we clarify which of their geodesic forms is equivalent to each Type from \cref{prop:geos 2k}. Case $(2)$ from \cite{EdjJohn92} is equivalent to Type 1, and Case $(5)$ from \cite{EdjJohn92} is equivalent to Type 3. Due to a small error, Case $(4)$ is reduced to Case $(3)$ in \cite{EdjJohn92} (see \cref{sec: error info}), which is then equivalent to Type 2. Finally we have absorbed all words from Case (1) of \cite{EdjJohn92} within each of the three types from \cref{prop:geos 2k}.
\end{rmk}
\begin{proof}[Proof of \cref{prop:even unique non unique}]
The relation of $G(m)$ implies that RR3 is the only rewrite rule that we can apply to transform a geodesic into another geodesic.
    We check whether RR3 applies to each set of geodesics.
    \par 
    $\mathcal{C}_{1}$: consider $(x^{\mu_{1}}y^{\ep_{1}}\dots x^{\mu_{n}}y^{\ep_{n}}\Delta^{c})^{\pm 1} \in \mathcal{C}_{1}$. For all $x^{\mu_{i}}$, either $\mu_{i} \geq 0$ or $\mu_{i} \leq 0$ for all $1 \leq i \leq n$. In either case, RR3 never holds, and so these are unique geodesics.
    
    $\mathcal{C}_{2}$: The case where $w = x^{j}$ is clear, so suppose $w = x^{\mu_{1j}}y^{\ep_{1j}}\dots y^{\ep_{n_{j}j}}x^{j}$. Then $\mu_{ik} \in [-(2k-j-1), j] \cap \Z_{\neq 0}$ for all $1 \leq i \leq n_{j}$ . For RR3 to hold, we would need some $\sigma > 0 > \delta$ where $\sigma, \delta \in [-(2k-j-1), j] \cap \Z_{\neq 0}$ such that $\sigma + |\delta| = 2k$. However if that was the case, then
    \[ \sigma + |\delta| \leq j + 2k-j-1 = 2k-1 < 2k.
    \]
    Hence RR3 never holds, and so geodesics are unique.
    
    $\mathcal{C}_{3}$: consider $x^{\alpha}y^{\ep_{1}}x^{\nu_{1}}\dots x^{\nu_{n}} \in \mathcal{C}_{3}$. We have $\alpha, \nu_{i} \in [-(k-1), k-1] \cap \Z$ for all $1 \leq i \leq n$. Hence if RR3 holds for some $\sigma>0>\delta$, then
    \[ \sigma+|\delta| \leq 2(k-1) = 2k-2 < 2k.
    \]
    Hence RR3 can never be applied, and so geodesics are unique. This completes the set $\mathcal{C}$ of unique geodesics.
    \par 
    For the set $\mathcal{D}$, assume $w$ is of the form in \cref{eq:any j}. If RR3 holds for some $\sigma > 0 > \delta$ such that $\sigma \neq j$ or $\delta \neq -(2k-j)$, then since $\sigma, \delta \in [-(2k-j-1), j-1] \cap \Z_{\neq 0}$, we have 
    \[ \sigma + |\delta| \leq j-1 + 2k-j-1 = 2k-2 < 2k,
    \]
    which is a contradiction. The only other possibility to apply RR3 is using any $x^{j}$ and $x^{-(2k-j)}$ terms. Since $j + |-(2k-j)| = 2k$, we can apply RR3 to rewrite our word in the form 
    \[ ux^{j}vx^{-(2k-j)}w =_{G} ux^{-(2k-j)}vx^{j}w,
    \]
    for all pairs of $x^{j}, x^{-(2k-j)}$ terms. Hence $\mathcal{D}$ consists of non-unique geodesics.
\end{proof}
\begin{rmk}\label{sec: error info}
    In the original classification of geodesics in \cite{EdjJohn92}, there is a fourth type of geodesic defined as follows when $p = 2k$:
\begin{enumerate}
    \item[Type 4.] $\left(x^{\mu_{1k}}y^{\ep_{1}}\dots x^{\mu_{nk}}y^{\ep_{n}}\right)x^{k}\left(y^{\ep}x^{\nu_{1}}y^{\delta_{1}}\dots x^{\nu_{m}}y^{\delta_{m}}\right)x^{-k}\left(y^{\gamma_{1}}x^{\lambda_{1}}\dots y^{\gamma_{l}}x^{\lambda_{l}}\right)$ where:
    \begin{itemize}
        \item $n,m,l \geq 0$,
        \item $\ep_{s}, \ep, \delta_{t}, \gamma_{u} \in \Z_{\neq 0} \; (1 \leq s \leq n, 1 \leq t \leq m, 1 \leq u \leq l)$,
        \item $\mu_{sk} \in [-(k-1), k] \cap \Z_{\neq 0} \; (1 \leq s \leq n)$, 
        \item $\nu_{t} \in [-(k-1), k-1] \cap \Z_{\neq 0} \; (1 \leq t \leq m)$, 
        \item $\lambda_{u} \in [-k, k-1] \cap \Z_{\neq 0} \; (1 \leq u \leq l)$.
    \end{itemize}
    \end{enumerate}
    Due to an error in the original paper, specifically $\mu_{sk} \neq -k$, Type 4 is in fact equivalent to our Type 2 geodesics as defined in \cref{prop:geos 2k}, when $j = k$. This can be seen by the fact that $\nu_{t}, \lambda_{u} \in [-k, k-1] \cap \Z_{\neq 0} \subset [-k, k] \cap \Z_{\neq 0} \; (1 \leq t \leq m, 1 \leq u \leq l)$. 
\end{rmk}

\comm{
\begin{table}[h]
    \begin{adjustbox}{width=1\textwidth}
    \begin{tabular}{|c|c|c|c|}  
        \hline
         Reference & Type & Form and conditions  & Unique/Non-unique \\
        \hline \hline
        \multirow{3}{*}{\cref{subsec: p=2}} & (1) & $\begin{array}{c}
            w = (x^{\mu_{1}}y^{\ep_{1}}\dots x^{\mu_{n}}y^{\ep_{n}}\Delta^{c})^{\pm 1}  \\
            c, n \geq 1,  \\
            \mu_{1} \in \{0, 1\}, \; \mu_{i} = 1 \;(2 \leq i \leq n), \\
            \ep_{i} \in \Z \; (1 \leq i \leq n), \ep_{i} \neq 0 \; (1 \leq i \leq n-1).
        \end{array}$ & Unique \\
        \cline{2-4}
        & (2) & $\begin{array}{c}
             w = A_{1}xA_{2}x\dots A_{\tau_{1}}xA_{\tau_{1}+1}x^{-1}A_{\tau_{1} + 2}x^{-1}\dots x^{-1}A_{\tau_{1} + \tau_{2} + 1}, \\
             A_i=y^{\ep_{i}}, \ep_{i} \in \Z_{\neq 0}\\
             \tau_{1}, \tau_{2} \geq 1
        \end{array}$ & Non-unique \\
        \cline{2-4}
        & (3) & $w = y^{\beta}, \; \beta \in \Z_{\neq 0}$ & Unique \\
        \hline
    \end{tabular}
    \end{adjustbox}
    \caption{Subcase of even dihedral Artin groups: $\mathrm{BS}(2,2)$ (that is, $p=2$)}
    \label{tab:my_label}
\end{table}
}

\end{appendices}
\section*{Acknowledgments}
The authors would like to thank Yago Antol\'{i}n, Martin Edjvet, Thomas Haettel and Fujii Michihiko for their generous advice and helpful discussions. 

\bibliography{references}
\bibliographystyle{plain}
\uppercase{\footnotesize{Institute for Mathematics, TU Berlin, Germany \& Dept. of Mathematics, Heriot-Watt University \& Maxwell Institute for Mathematical Sciences, Edinburgh}}
\par
\uppercase{\footnotesize{Department of Mathematics, University of Manchester M13 9PL, UK and the Heilbronn Institute for Mathematical Research, Bristol, UK}}
\par 
\textit{Email addresses:} \texttt{L.Ciobanu@hw.ac.uk, ciobanu@math.tu-berlin.de, gemma.crowe@manchester.ac.uk}

\end{document}